\documentclass[oneside, reqno, 10pt, a4paper]{amsart}

\title[Reduced basis solvers for unfitted methods on parameterized domains]{Reduced basis solvers for unfitted methods on parameterized domains}
\date{\today}
\address{$^\dagger$School of mathematics\\Monash university\\Clayton\\Victoria 3800\\Australia}
\address{$^\P $School of Mathematics and Statistics \\University of Sydney\\New South Wales 2006\\Australia}
\author[N. Mueller]{Nicholas Mueller$^\dagger$}
\email{nicholas.mueller@monash.edu}
\author[S. Badia]{Santiago Badia$^{\dagger}$}
\email{santiago.badia@monash.edu}
\author[Y. Zhao]{Yiran Zhao$^\P$}
\email{yiran.zhao@monash.edu}

\usepackage{dutchcal}
\usepackage{microtype,times}
\AtBeginDocument{
  \DeclareSymbolFont{AMSb}{U}{msb}{m}{n}
  \DeclareSymbolFontAlphabet{\mathbb}{AMSb}
}
\DeclareFontFamily{U}{mathx}{\hyphenchar\font45}
\DeclareFontShape{U}{mathx}{m}{n}{<-> mathx10}{}
\DeclareSymbolFont{mathx}{U}{mathx}{m}{n}
\DeclareMathAccent{\widebar}{0}{mathx}{"73}
\DeclareMathAccent{\widecheck}{0}{mathx}{"71}

\usepackage[dvipsnames]{xcolor}
\usepackage{graphicx}
\usepackage{subfig}
\usepackage{svg}
\usepackage{tikz}
\tikzstyle{arrow} = [thick,->,>=stealth]
\usetikzlibrary{shapes.misc,matrix,fit,positioning,arrows.meta,decorations.pathreplacing,calc,shapes.geometric,arrows}
\tikzset{Matrix/.style={matrix of nodes, font=\footnotesize,text height=1pt, text depth=0.5pt, text width=8.5pt, align=center, column sep=0pt, row sep=0pt, nodes in empty cells}}
\tikzset{cross/.style={cross out, draw=black, minimum size=2*(#1-\pgflinewidth), inner sep=0pt, outer sep=0pt},cross/.default={3pt}}
\usepackage[a4paper, pdftex, left=2cm, top=2cm, right=2cm, bottom=2cm]{geometry}
\usepackage[final]{pdfpages}
\usepackage{changepage}

\usepackage[foot]{amsaddr}

\usepackage{url}
\usepackage{hyperref}
\hypersetup{
    breaklinks=true,
    bookmarksopen=true,
    pdftitle={Reduced basis solvers for unfitted methods on parameterized domains}, 
    pdfauthor={Nicholas Mueller, Santiago Badia, Yiran Zhao, Tiangang Cui}, 
    pdfsubject={Model order reduction, tensor train decomposition, reduced basis method, parameterized geometries}, 
    pdfkeywords={complexity reduction, algorithmic efficiency, parametric dependence}, 
    colorlinks=true,
    linkcolor=black,
    citecolor=blue,
    filecolor=black,
    urlcolor=blue
}

\usepackage{csquotes}
\usepackage[backend=biber, maxbibnames=5, style=numeric-comp, isbn=false, bibencoding=utf8, safeinputenc, url=false, doi=true, giveninits=true]{biblatex}

\PassOptionsToPackage{normalem}{ulem}
\usepackage{ulem}
\providecolor{added}{rgb}{0,0,1}
\providecolor{deleted}{rgb}{1,0,0}

\usepackage{float}
\usepackage{framed}
\usepackage{verbatim}
\usepackage{fancyvrb}
\usepackage{booktabs}
\usepackage{colortbl}
\usepackage{multirow}
\usepackage{algorithm}
\usepackage{algpseudocode}
\makeatletter
\algdef{S}[IF]{IfNoThen}[1]{\algorithmicif\ #1}
\makeatother
\usepackage[inline]{enumitem}
\usepackage{siunitx}
\usepackage{mathtools}
\mathtoolsset{showonlyrefs}

\newtheorem{remark}{Remark}

\newtheorem{theorem}{Theorem}

\usepackage[breakable]{tcolorbox}
\tcbuselibrary{theorems}
\tcbuselibrary{skins}
\tcbset{
	commonstyle/.style={
		theorem style=plain,
		enhanced jigsaw,
		fonttitle=\bfseries,
		fontupper=\itshape,
		halign=justify,
		separator sign=:,
		description delimiters none,
		description font=\bfseries, 
		terminator sign={.\hspace{0.25em}},
		arc=0mm,outer arc=0mm,
		boxrule=0pt,toprule=0pt,bottomrule=0pt,leftrule=0pt,rightrule=0pt,
		titlerule=0pt,toptitle=0pt,bottomtitle=0pt,top=0pt,
		colback=white,coltitle=black,
		boxsep=0pt, bottom=0pt, left=0pt, 
	}
}
\newtcbtheorem[]{myproblem}{Problem}%
{center, commonstyle, fonttitle=\bfseries}{pb}
\newtcbtheorem[]{mydefinition}{Definition}
{center, commonstyle, fonttitle=\bfseries}{pb}
\newtcbtheorem[]{myassumption}{Assumption}
{center, commonstyle, fonttitle=\bfseries}{pb}

\definecolor{shadecolor}{gray}{.92}
\definecolor{incolor}{rgb}{0,0,.7}
\definecolor{outcolor}{rgb}{.65,0,0}
\definecolor{syntaxcolor}{rgb}{.65,0,0}
\definecolor{bg}{rgb}{0.93,0.93,0.93}
\definecolor{myblue}{RGB}{93,188,210}
\definecolor{mygreen}{RGB}{189,210,93}
\definecolor{myorange}{RGB}{210,173,93}
\definecolor{myred}{RGB}{210,93,130}
\definecolor{mydarkblue}{RGB}{93,130,210}
\definecolor{mydarkgreen}{RGB}{93,210,173}

\usepackage{amsmath, amsfonts, amssymb, amscd, bm, bbm, mathtools}

\renewcommand{\vec}[1]{\underline{#1}}

\newcommand{\R}{\mathbb{R}}

\newcommand{\norm}[1]{\left\lVert#1\right\rVert}

\DeclareMathOperator*{\argmax}{\arg\max}

\usepackage{acronym}
\acrodef{pde}[PDE]{partial differential equation}
\acrodef{fe}[FE]{finite element}
\acrodef{fem}[FEM]{finite element method}
\acrodef{dof}[DOF]{degree of freedom}
\acrodefplural{dof}[DOFs]{degrees of freedom}
\acrodef{be}[BE]{Backward Euler}
\acrodef{hf}[HF]{high fidelity}
\acrodef{fom}[FOM]{full order model}
\acrodef{lhs}[LHS]{left hand side}
\acrodef{rhs}[RHS]{right hand side}
\acrodef{rom}[ROM]{reduced order model}
\acrodef{rb}[RB]{reduced basis}
\acrodefplural{rb}[RBs]{reduced bases}
\acrodef{svd}[SVD]{singular value decomposition}
\acrodef{sthosvd}[ST-HOSVD]{sequentially truncated high order singular value decomposition}
\acrodef{pod}[POD]{proper orthogonal decomposition}
\acrodef{tpod}[TPOD]{truncated proper orthogonal decomposition}
\acrodef{strb}[ST-RB]{space-time reduced basis}
\acrodef{eim}[EIM]{empirical interpolation method}
\acrodef{deim}[DEIM]{discrete empirical interpolation method}
\acrodef{mdeim}[MDEIM]{empirical interpolation method in matrix form}
\acrodef{stmdeimrb}[ST-MDEIM-RB]{space-time MDEIM-RB}
\acrodef{dl}[DL]{deep learning}
\acrodef{nn}[NN]{neural network}
\acrodef{tt}[TT]{tensor-train}
\acrodef{ttrb}[TT-RB]{tensor-train reduced basis}
\acrodef{podrb}[POD-RB]{proper orthogonal decomposition reduced basis}
\acrodef{ttsvd}[TT-SVD]{tensor-train SVD}
\acrodef{ttcross}[TT-CROSS]{tensor-train cross}
\acrodef{ttmdeim}[TT-MDEIM]{tensor-train MDEIM}
\acrodef{jit}[JIT]{just-in-time}
\acrodef{rbf}[RBF]{radial basis function}
\acrodefplural{rbf}[RBFs]{radial bases functions}
\acrodef{ale}[ALE]{arbitrary Eulerian-Lagriangian}

\graphicspath{{img/}}

\begin{document}

\maketitle

\begin{abstract}
    In this paper, we present a unified framework for reduced basis approximations of parametrized partial differential equations defined on parameter-dependent domains. Our approach combines unfitted finite element methods with both classical and tensor-based reduced basis techniques -- particularly the tensor-train reduced basis method -- to enable efficient and accurate model reduction on general geometries. To address the challenge of reconciling geometric variability with fixed-dimensional snapshot representations, we adopt a deformation-based strategy that maps a reference configuration to each parameterized domain. Furthermore, we introduce a localization procedure to construct dictionaries of reduced subspaces and hyper-reduction approximations, which are obtained via matrix discrete empirical interpolation in our work. We extend the proposed framework to saddle-point problems by adapting the supremizer enrichment strategy to unfitted methods and deformed configurations, demonstrating that the supremizer operator can be defined on the reference configuration without loss of stability. Numerical experiments on two- and three-dimensional problems -- including Poisson, linear elasticity, incompressible Stokes and Navier-Stokes equations -- demonstrate the flexibility, accuracy and efficiency of the proposed methodology.
\end{abstract}

\section{Introduction}
\label{sec: introduction}

Projection-based \acp{rom} are advanced numerical techniques designed to approximate parametric \ac{hf} models, which typically involve finely resolved spatio-temporal discretizations of \acp{pde}. These methods aim to capture the \ac{hf} parameter-to-solution manifold within a carefully chosen vector subspace. The process generally consists of a computationally intensive offline phase, during which the subspace is constructed and a (Petrov-)Galerkin projection of the \ac{hf} equations onto this subspace is performed. This stage often includes the hyper-reduction of nonaffinely parameterized \ac{hf} quantities, such as residuals and Jacobians. Once the reduced model is assembled, a subsequent online phase enables the rapid evaluation of accurate solutions for new parameter values, at a cost largely independent of the full-order problem size.

The \ac{rb} method is one of the most prominent projection-based \acp{rom}, with a broad range of applications, including unsteady \cite{doi:10.1137/22M1509114,MUELLER2024115767}, nonlinear \cite{quarteroni2015reduced,dalsanto2019hyper}, and saddle-point problems \cite{ballarin2015supremizer,negri2015reduced}. It constructs the projection subspace using a dictionary of high-fidelity solutions -- commonly referred to as snapshots. For reducible problems \cite{quarteroni2015reduced}, i.e., those whose parameter-to-solution maps exhibit low Kolmogorov $N$-width \cite{Unger_2019}, the reduced space can be of much smaller dimension than the \ac{hf} space, while retaining high accuracy. Subspace extraction techniques include direct compression methods \cite{kunisch2001galerkin,Legresley2006Alonso} and greedy algorithms \cite{prudhomme:hal-00798326,Prudhomme}. Direct approaches benefit from a priori estimates, such as those derived from the Schmidt-Eckart-Young Theorem \cite{Eckart_Young_1936,Mirsky1960}, which underpins the widely used \ac{tpod} method \cite{Volkwein2012}. Greedy procedures, in contrast, iteratively build the reduced basis using a posteriori error estimators, though often at a higher computational cost. More recently, tensor-based approaches such as the \ac{ttrb} method \cite{MAMONOV2022115122,mamonov2024priorianalysistensorrom,MUELLER2026116790} have been proposed, in which snapshots are treated as multidimensional arrays and compressed along Cartesian axes using \ac{ttsvd} or \ac{ttcross} \cite{oseledets2011tensor,gorodetsky2019continuous,oseledets2010tt,bigoni2016spectral,cui2021deep,dolgov2018approximation}. These methods offer cheaper offline costs, but in the existing literature their applicability has been limited to problems posed on tensor-product grids -- a restriction that the present work overcomes.

A multi-purpose \ac{rom} framework implementing both classical approaches as well as the more novel \ac{ttrb} on arbitrarily parameterized domains can be effectively built upon unfitted (or embedded-boundary) \ac{fe} methods \cite{HANSBO20025537,https://doi.org/10.1002/nme.4823,Elfverson_2018,Belytschko2001}. These techniques allow for the discretization of \acp{pde} on complex geometries without requiring body-fitted meshes, which are often challenging to generate. Instead, a background Cartesian mesh is used to define the \ac{fe} space, while the geometry is specified analytically. Although early unfitted approaches suffered from ill-conditioning due to small cut cells \cite{DEPRENTER2017297}, more recent developments such as ghost penalty methods \cite{https://doi.org/10.1002/nme.4823,BurmanHansbo} and the aggregated \ac{fe} method \cite{BADIA2018533} have addressed these issues, achieving mesh-independent condition numbers, controlled penalty parameters, and optimal convergence rates in both mesh size and polynomial order.

Despite the advantages offered by unfitted \ac{fe} methods, constructing \acp{rom} on parameterized geometries remains a challenging task. The first major difficulty lies in reconciling geometric variability with the requirement that all snapshots reside in a fixed-dimensional vector space. Two primary strategies have been developed to address this issue. The first, known as the extension-based approach, involves computing each snapshot on its parameterized domain and subsequently extending it onto the background mesh. This method is widely adopted in embedded-domain frameworks due to its conceptual simplicity and compatibility with structured discretizations \cite{NOUY20113066,BALAJEWICZ2014489,karatzas2021reducedordermodelstable,Zeng_2022,Katsouleas2020,Karatzas_2019,KARATZAS2020833}, but the quest for a robust and accurate extension process is still very active.

A more accurate alternative -- though geometrically less flexible \cite{CIRROTTOLA2021110177} -- is the deformation-based approach. In this setting, a fixed reference configuration is employed, and each parameterized domain is mapped onto it through a smooth deformation map. The weak form of the \ac{pde} is then integrated over the reference mesh via pull-back operations. This avoids remeshing and artificial extensions altogether, preserving snapshot fidelity. However, it requires solving additional high-fidelity systems to compute the deformation maps. Common choices include harmonic \cite{Baker2002MeshMovement}, biharmonic \cite{Helenbrook2003Biharmonic}, and elastic deformations \cite{Stein2003MeshMovingFSI,Stein2004SolidExtension,Tezduyar1992ASME}, many of which stem from \ac{ale} formulations in fluid-structure interaction \cite{Wick2011MeshMotionFSI,BASTING2017312}. In this work, we adopt the deformation-based approach as our default strategy; for \ac{ttrb} approximations, we supplement it with an extension-based procedure to ensure that snapshot values are defined at every index of the background Cartesian grid.

A second major challenge arises from the limited reducibility of problems posed on parameterized domains. To mitigate this, we employ a localization strategy that enhances the compressibility of the reduced model. The use of local reduced subspaces has recently gained attention as an effective means to achieve meaningful dimensionality reduction \cite{https://doi.org/10.1002/nme.4371,doi:10.1137/130924408,PAGANI2018530}. Among the few works incorporating this idea within unfitted discretization frameworks, \cite{CHASAPI2023115997} constructs local subspaces for both projection and hyper-reduction. Their approach relies on traditional \ac{tpod} bases and local hyper-reduction via \ac{rbf} interpolation \cite{10.1093/oso/9780198534396.003.0003}. By contrast, we explore both \ac{tpod} and \ac{tt} subspaces in conjunction with \ac{mdeim}-based hyper-reduction. While the latter is more computationally expensive, we find it substantially more accurate than \ac{rbf}-based methods, making it better suited for the high-fidelity requirements of our applications. Another key distinction from \cite{CHASAPI2023115997} lies in our reliance on deformation-based rather than extension-based snapshot representations, which, in our experience, results in improved compressibility and overall reducibility of the problem.

In summary, the main contributions of the work are the following. 
\begin{itemize}
    \item We develop an interface that combines an unfitted \ac{fe} method with deformation mappings from a fixed reference configuration to physical domains, enabling the generation of \ac{hf} snapshots on a background Cartesian grid. This framework facilitates the straightforward application of both classical and tensor-based \ac{rb} methods -- namely \ac{ttrb} -- not only on general geometries, but also on parameterized ones. To the best of our knowledge, this represents the first such attempt in the literature. 
    \item We implement a localized strategy to construct dictionaries of reduced subspaces and corresponding hyper-reduction approximations for the residuals and Jacobians of parameterized problems. While this idea was previously explored in \cite{CHASAPI2023115997}, our framework employs \ac{mdeim}-based hyper-reduction, which we find to be significantly more accurate than their \ac{rbf}-based approach.
    \item We extend the applicability of unfitted \acp{rom} on parameterized geometries to saddle-point problems. Specifically, we demonstrate that the supremizer enrichment technique introduced in \cite{ballarin2015supremizer} for ensuring the well-posedness of reduced saddle-point systems can be adapted to our problems by computing the supremizer operator on the reference configuration, instead of on the deformed one. This significantly improves the feasibility and efficiency of the enrichment procedure.
\end{itemize}

This article is organized as follows. In Section \ref{sec: unfitted fe method}, we introduce the \ac{fom}, which consists of a generic parameterized \ac{pde} solved using the aggregated \ac{fe} method in combination with deformation mappings. Section \ref{sec: rb} provides a general overview of the \ac{rom} corresponding to the full-order problem, while Section \ref{sec: ttrb} focuses specifically on the \ac{ttrb} method. In Section \ref{sec: saddle-point}, we extend the discussion to stabilized \ac{rb} formulations for saddle-point problems discretized using the aggregated \ac{fe} method and deformation mappings. Section \ref{sec: results} presents numerical results across several benchmark test cases. Finally, in Section \ref{sec: conclusions}, we summarize the main findings and outline possible directions for future work.

\section{Full order model}
\label{sec: unfitted fe method}

In this section, we present the \ac{hf} framework used for solving parameterized \acp{pde}. We first define the notion of mesh deformation maps for efficient handling of shape parameters. We then present the aggregated cell method \cite{BADIA2018533}, which serves as the unfitted scheme employed in this work. Next, we discuss the extension of nodal values from the physical domain to the background grid. Finally, we define the induced norms on the previously introduced \ac{fe} spaces.

\subsection{Mesh deformation techniques}
\label{subs: mesh def}
We introduce a \textit{reference} spatial domain $\widetilde{\Omega} \subset \R^d$, where $d$ denotes the spatial dimension. The reference boundary $\partial\widetilde{\Omega}$ is partitioned into two portion with prescribed non-zero $\widetilde{\Gamma}$ and zero displacement $\partial\widetilde{\Omega} \setminus \widetilde{\Gamma}$, resp. On $\widetilde{\Omega}$ (resp. $\widetilde{\Gamma}$), we define a conforming, quasi-uniform mesh denoted by $\widetilde{\Omega}_h$ (resp. $\widetilde{\Gamma}_h$), with mesh size $h$. A boundary displacement $\vec{\psi}_{\Gamma}: \widetilde{\Gamma}_h \to \Gamma_h$ is prescribed on $\widetilde{\Gamma}$ . The goal of mesh deformation techniques is to construct a deformation map $\vec{\psi}: \widetilde{\Omega}_h \to \Omega_h$ as a diffeomorphism that extends $\vec{\psi}_{\Gamma}$ to the bulk. Once $\vec{\psi}$ is computed, the deformed domain is defined as
\begin{equation*}
    \Omega = \{\bm{x} \in \R^d : \bm{x} = \widetilde{\bm{x}} + \vec{\psi}(\widetilde{\bm{x}}), \ \widetilde{\bm{x}} \in \widetilde{\Omega} \}.
\end{equation*}
In this work, we employ elastic deformations, where $\vec{\psi}$ solves the linear elasticity problem:
\begin{align} 
    \label{eq: strong form elasticity equation}
    \left\{
    \begin{aligned}
        -\bm{\nabla} \cdot (\vec{\vec{\sigma}}(\vec{\psi})) &= \vec{0}  &&\text{in } \widetilde{\Omega}, \\
        \vec{\psi} &= \vec{\psi}_{\Gamma} &&\text{on } \widetilde{\Gamma}, \\
        \vec{\psi} &= \vec{0} &&\text{on } \partial \widetilde{\Omega} \setminus \widetilde{\Gamma},
    \end{aligned} 
    \right.
\end{align}
with stress tensor
\begin{equation}
    \label{eq: stress tensor}
    \vec{\vec{\sigma}}(\vec{\psi}) = 2\gamma \vec{\vec{\epsilon}}(\vec{\psi}) + \lambda \bm{\nabla} \cdot \vec{\psi} \ \vec{\vec{I}}_d,
\end{equation}
where $\vec{\vec{\epsilon}}$ is the symmetric gradient, $\vec{\vec{I}}_d$ is the $ d\times d$ identity matrix, and $\lambda$, $\gamma$ are the Lamé coefficients, expressed in terms of the Young modulus $E$ and the Poisson ratio $\nu$ as:
\begin{align}
    \lambda = \frac{E \nu}{(1+\nu)(1-2\nu)}, \qquad
    \gamma = \frac{E}{2(1+\nu)}.
\end{align}  
To obtain the deformed mesh, we discretize \eqref{eq: strong form elasticity equation} using the \ac{fe} method and solve for the discrete displacement $\vec{\psi}_h$. The deformed mesh is then given by
\begin{equation*} 
    \Omega_h = \{\bm{x} \in \R^d : \bm{x} = \widetilde{\bm{x}} + \vec{\psi}_h(\widetilde{\bm{x}}), \widetilde{\bm{x}} \in \widetilde{\Omega}_h \}.
\end{equation*}
An example of a deformed mesh is shown in Fig.~\ref{fig:domains}.

Mesh deformation maps are widely used in fluid-structure interaction problems based on \ac{ale} formulations \cite{Wick2011MeshMotionFSI,BASTING2017312}. They are also valuable in problems involving shape parameters, where one aims to define families of parameterized domains. For instance, consider a \ac{pde} defined on the unit square with a circular hole whose center moves along the diagonal and whose radius varies. The domain is parameterized as:
\begin{equation}
    \Omega(\bm{\mu}) \doteq [0,1]^2 \setminus \mathcal{B}_{(\mu_1,\mu_1)}(\mu_2), \quad 
    \bm{\mu} \doteq (\mu_1,\mu_2),
\end{equation}
where $\mathcal{B}_{\bm{c}}(R)$ denotes a ball of radius $R$ centered at $\bm{c}$. Solving a \ac{pde} on $\Omega(\bm{\mu})$ for multiple values of $\bm{\mu}$ can be costly in standard \ac{fe} codes due to repeated remeshing and basis construction. This difficulty is alleviated by mesh deformation techniques \cite{shamanskiy2020mesh}, which proceed as follows:
\begin{itemize}
    \item Choose a reference parameter $\widetilde{\bm{\mu}}$ corresponding to the reference domain $\widetilde{\Omega}$ -- ideally one that minimizes mesh skewness (e.g., a hole centered in the square).
    \item Solve \eqref{eq: strong form elasticity equation} with boundary displacement
    \begin{equation}
        \vec{\psi}_{\Gamma}(\bm{\mu}): (x,y) \mapsto \left(\mu_1-x+\mu_2(x-\widetilde{\mu}_1)/\widetilde{\mu}_2,\ \mu_1-y+\mu_2(y-\widetilde{\mu}_1)/\widetilde{\mu}_2\right)
    \end{equation}
    applied on $\widetilde{\Gamma} \doteq \partial \mathcal{B}_{(\widetilde{\mu}_1,\widetilde{\mu}_1)}(\widetilde{\mu}_2)$, yielding a deformation map $\vec{\psi}(\bm{\mu})$.
    \item Recast the original \ac{pde} -- defined on $\Omega(\bm{\mu})$ -- onto the reference configuration $\widetilde{\Omega}$.
\end{itemize}

\begin{figure}[t]
    \centering
    \begin{tikzpicture}
        \node[anchor=south west, inner sep=0] (img1) at (0,0) {\includegraphics[width=0.5\textwidth]{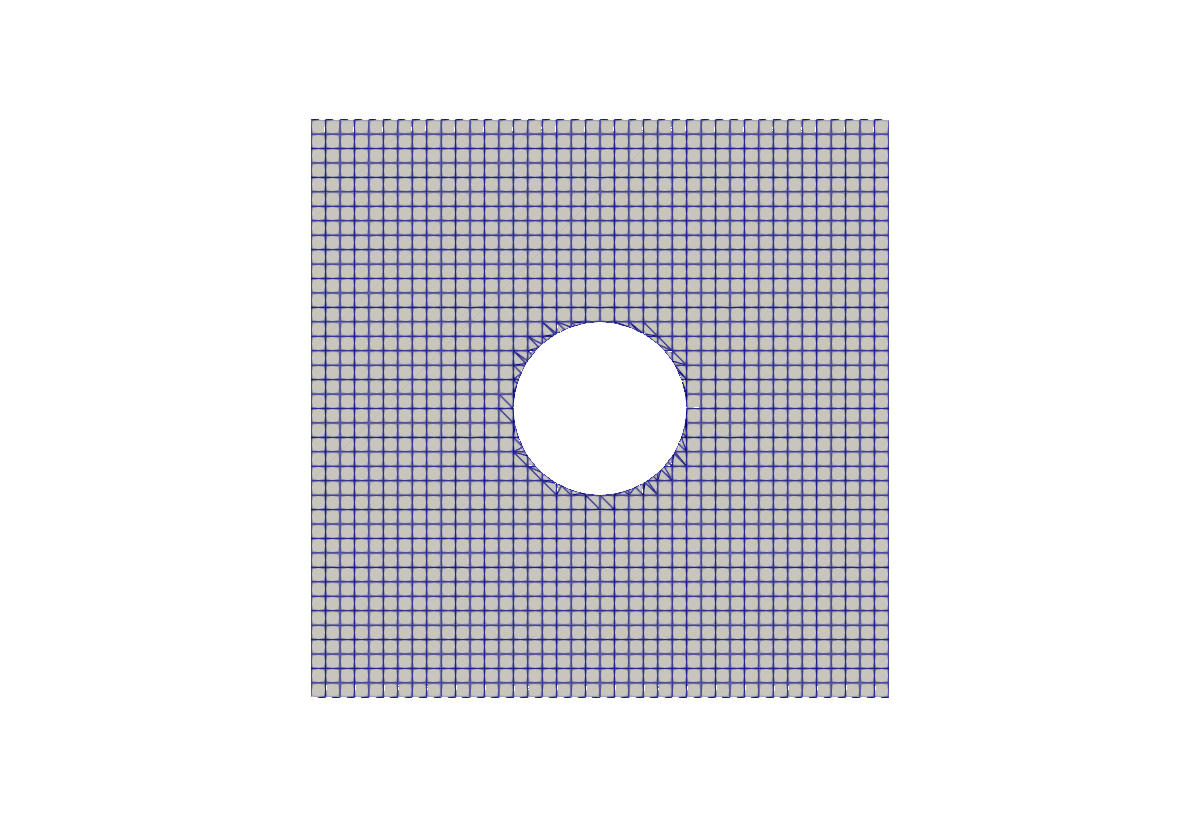}};
        \node at ($(img1.north) - (0.0,0.2)$) {$\widetilde{\Omega}_h$};

        \node[anchor=south west, inner sep=0] (img2) at (6,0) {\includegraphics[width=0.5\textwidth]{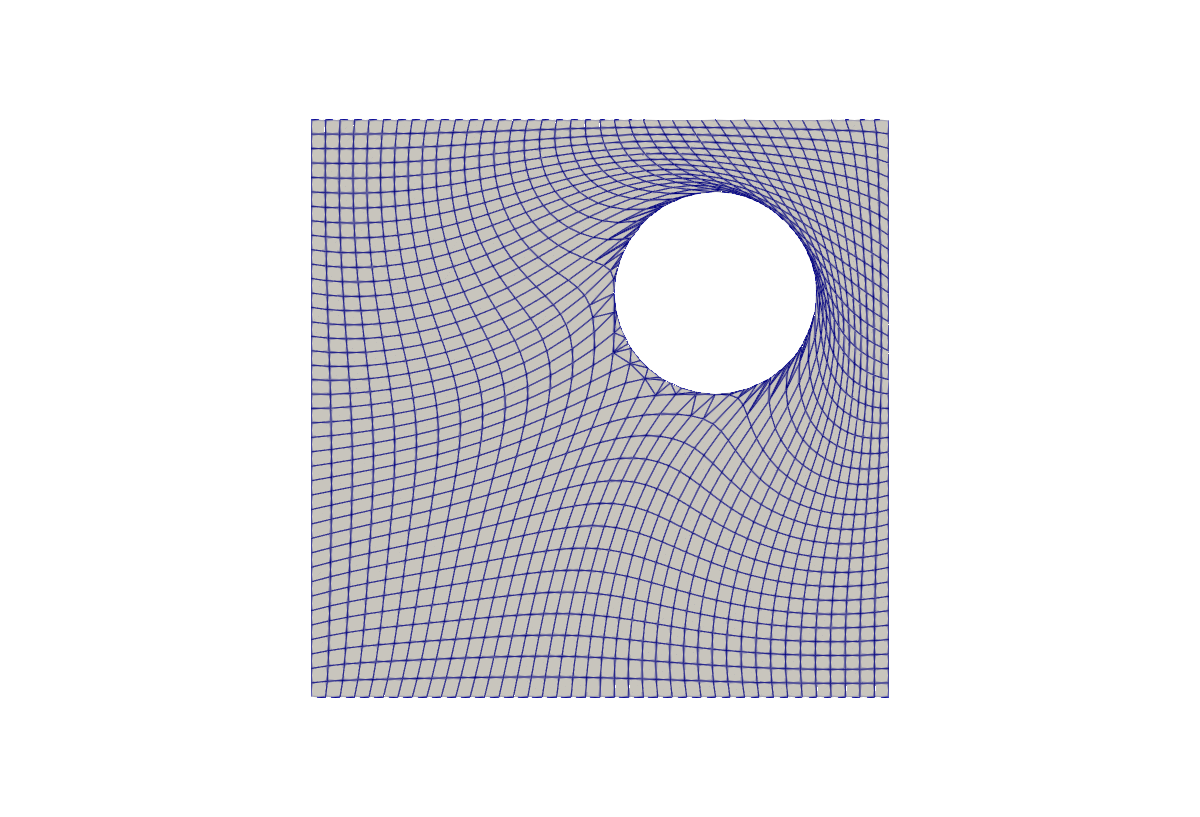}};
        \node at ($(img2.north) - (0.0,0.2)$) {$\Omega_h(\bm{\mu})$};
    \end{tikzpicture}
    \caption{Reference and deformed mesh obtained respectively for $\widetilde{\bm{\mu}} = (0.5,0.12)$ and $\bm{\mu} = (0.6,0.15)$.}
    \label{fig:domains}
\end{figure}

To achieve the last step, gradients, integral measures, and boundary normals on the deformed configuration must be pulled back to the reference configuration. Denoting the Jacobian of the deformation map by $\vec{\vec{J}}(\bm{\mu}) \doteq \pmb{\nabla} \vec{\psi} \in \R^{d \times d}$, the following formulas hold:
\begin{align} 
    \label{eq: mapped operators}
    \begin{aligned}
        \bm{\nabla} &\doteq \vec{\vec{J}}^{-T}(\bm{\mu}) \widetilde{\bm{\nabla}}, \quad 
        &&\int_{\Omega(\bm{\mu})} d\Omega(\bm{\mu}) &&\doteq \int_{\widetilde{\Omega}} \mathrm{det}(\vec{\vec{J}}(\bm{\mu})) d\widetilde{\Omega}, \\ 
        \vec{n}(\bm{\mu}) &\doteq \frac{ \vec{\vec{J}}^{-T}(\bm{\mu}) \widetilde{\vec{n}} }{\|\vec{\vec{J}}^{-T}(\bm{\mu}) \widetilde{\vec{n}}\|_2}, \quad 
        &&\int_{\Gamma(\bm{\mu})} d\Gamma(\bm{\mu}) &&\doteq \int_{\widetilde{\Gamma}}\|\vec{\vec{J}}^{-T}(\bm{\mu}) \widetilde{\vec{n}}\|_2 \mathrm{det}(\vec{\vec{J}}(\bm{\mu})) d\widetilde{\Gamma}, 
    \end{aligned} 
\end{align}
where:
\begin{itemize}
    \item $\bm{\nabla}$ (resp. $\widetilde{\bm{\nabla}}$) is the gradient operator on the deformed (resp. reference) domain;
    \item $d\Omega(\bm{\mu})$ (resp. $d\widetilde{\Omega}$) is the volume measure on the deformed (resp. reference) domain;
    \item $d\Gamma(\bm{\mu})$ (resp. $d\widetilde{\Gamma}$) is the surface measure on the deformed (resp. reference) boundary;
    \item $\vec{n}(\bm{\mu})$ (resp. $\widetilde{\vec{n}}$) is the outward unit normal vector on the deformed (resp. reference) boundary.
\end{itemize}
Finally, all physical quantities (e.g., boundary conditions, coefficients) must be expressed in the reference configuration via composition with the inverse mapping $\vec{\psi}(\bm{\mu})^{-1}$. The relations in \eqref{eq: mapped operators} enable the definition of weak forms and finite element spaces on the reference configuration, and can thus be reused for all parameters. Throughout this work, unless otherwise stated, all expressions involving integration over deformed configurations are understood via the transformations in \eqref{eq: mapped operators}.

\subsection{Aggregated cell method}
\label{subs: agfem}

We consider a space of shape parameters $\mathcal{D} \subset \R^p$ that determines the deformed domain $\Omega(\bm{\mu})$. We also define a non-parametric background domain $\widehat{\Omega}$, which serves as a bounding box for $\Omega(\bm{\mu})$ for any choice of $\bm{\mu}$. As before, we denote by $\widehat{\Omega}_h$ a conforming, quasi-uniform partition of $\widehat{\Omega}$, with the additional requirement that it be quadrangular. The cells of the partition can be classified as follows:
\begin{itemize}
    \item Cells that are fully contained within $\widetilde{\Omega}$ are \textit{internal}.
    \item Cells that are fully contained within $\widehat{\Omega} \setminus \widetilde{\Omega}$ are \textit{external}.
    \item The remaining are \textit{cut} cells.
\end{itemize}
Fig. \ref{fig: embedded cells} provides a graphical representation of the classification above. We refer to the set of \textit{active} cells as the union of internal and cut cells. As illustrated in Fig. \ref{fig: embedded nodes}, we may also label the nodes as active if they belong to an active cell, and as inactive (or external) otherwise. 
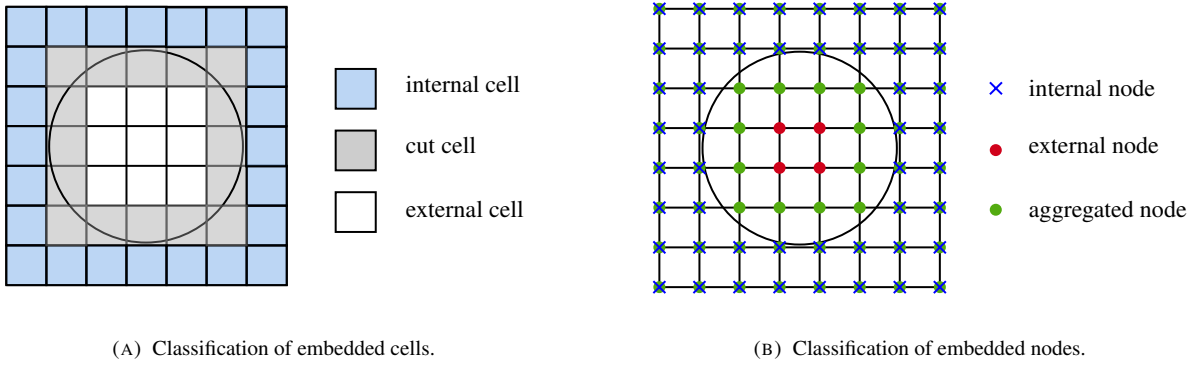
\begin{figure}[t]
    \centering
    \tikzset{every picture/.style={line width=0.75pt}} 

    \begin{tikzpicture}[x=0.75pt,y=0.75pt,yscale=-1,xscale=1]
    
    \draw  [draw opacity=0] (47,19) -- (187.42,19) -- (187.42,159.25) -- (47,159.25) -- cycle ; \draw   (47,19) -- (47,159.25)(67,19) -- (67,159.25)(87,19) -- (87,159.25)(107,19) -- (107,159.25)(127,19) -- (127,159.25)(147,19) -- (147,159.25)(167,19) -- (167,159.25)(187,19) -- (187,159.25) ; \draw   (47,19) -- (187.42,19)(47,39) -- (187.42,39)(47,59) -- (187.42,59)(47,79) -- (187.42,79)(47,99) -- (187.42,99)(47,119) -- (187.42,119)(47,139) -- (187.42,139)(47,159) -- (187.42,159) ; \draw    ;
    \draw  [color={rgb, 255:red, 155; green, 155; blue, 155 }  ,draw opacity=0.4 ][fill={rgb, 255:red, 155; green, 155; blue, 155 }  ,fill opacity=0.4 ] (67,99) -- (87,99) -- (87,119) -- (67,119) -- cycle ;
    \draw  [color={rgb, 255:red, 155; green, 155; blue, 155 }  ,draw opacity=0.4 ][fill={rgb, 255:red, 155; green, 155; blue, 155 }  ,fill opacity=0.4 ] (147,99) -- (167,99) -- (167,119) -- (147,119) -- cycle ;
    \draw  [color={rgb, 255:red, 155; green, 155; blue, 155 }  ,draw opacity=0.4 ][fill={rgb, 255:red, 155; green, 155; blue, 155 }  ,fill opacity=0.4 ] (67,79) -- (87,79) -- (87,99) -- (67,99) -- cycle ;
    \draw  [color={rgb, 255:red, 155; green, 155; blue, 155 }  ,draw opacity=0.4 ][fill={rgb, 255:red, 155; green, 155; blue, 155 }  ,fill opacity=0.4 ] (147,59) -- (167,59) -- (167,79) -- (147,79) -- cycle ;
    \draw  [color={rgb, 255:red, 155; green, 155; blue, 155 }  ,draw opacity=0.4 ][fill={rgb, 255:red, 155; green, 155; blue, 155 }  ,fill opacity=0.4 ] (67,59) -- (87,59) -- (87,79) -- (67,79) -- cycle ;
    \draw  [color={rgb, 255:red, 155; green, 155; blue, 155 }  ,draw opacity=0.4 ][fill={rgb, 255:red, 155; green, 155; blue, 155 }  ,fill opacity=0.4 ] (127,39) -- (147,39) -- (147,59) -- (127,59) -- cycle ;
    \draw  [draw opacity=0] (373,20) -- (513.42,20) -- (513.42,160.25) -- (373,160.25) -- cycle ; \draw   (373,20) -- (373,160.25)(393,20) -- (393,160.25)(413,20) -- (413,160.25)(433,20) -- (433,160.25)(453,20) -- (453,160.25)(473,20) -- (473,160.25)(493,20) -- (493,160.25)(513,20) -- (513,160.25) ; \draw   (373,20) -- (513.42,20)(373,40) -- (513.42,40)(373,60) -- (513.42,60)(373,80) -- (513.42,80)(373,100) -- (513.42,100)(373,120) -- (513.42,120)(373,140) -- (513.42,140)(373,160) -- (513.42,160) ; \draw    ;
    \draw  [color={rgb, 255:red, 0; green, 0; blue, 0 }  ,draw opacity=1 ][fill={rgb, 255:red, 74; green, 144; blue, 226 }  ,fill opacity=0.37 ] (211.42,50.25) -- (231.42,50.25) -- (231.42,70.25) -- (211.42,70.25) -- cycle ;
    \draw  [color={rgb, 255:red, 0; green, 0; blue, 0 }  ,draw opacity=1 ][fill={rgb, 255:red, 155; green, 155; blue, 155 }  ,fill opacity=0.5 ] (211.42,81.25) -- (231.42,81.25) -- (231.42,101.25) -- (211.42,101.25) -- cycle ;
    \draw  [color={rgb, 255:red, 0; green, 0; blue, 0 }  ,draw opacity=1 ][fill={rgb, 255:red, 255; green, 255; blue, 255 }  ,fill opacity=1 ] (211.42,112.25) -- (231.42,112.25) -- (231.42,132.25) -- (211.42,132.25) -- cycle ;
    \draw  [color={rgb, 255:red, 208; green, 2; blue, 27 }  ,draw opacity=1 ][fill={rgb, 255:red, 208; green, 2; blue, 27 }  ,fill opacity=1 ] (450.5,80) .. controls (450.5,78.62) and (451.62,77.5) .. (453,77.5) .. controls (454.38,77.5) and (455.5,78.62) .. (455.5,80) .. controls (455.5,81.38) and (454.38,82.5) .. (453,82.5) .. controls (451.62,82.5) and (450.5,81.38) .. (450.5,80) -- cycle ;
    \draw  [color={rgb, 255:red, 208; green, 2; blue, 27 }  ,draw opacity=1 ][fill={rgb, 255:red, 208; green, 2; blue, 27 }  ,fill opacity=1 ] (430.5,80) .. controls (430.5,78.62) and (431.62,77.5) .. (433,77.5) .. controls (434.38,77.5) and (435.5,78.62) .. (435.5,80) .. controls (435.5,81.38) and (434.38,82.5) .. (433,82.5) .. controls (431.62,82.5) and (430.5,81.38) .. (430.5,80) -- cycle ;
    \draw  [color={rgb, 255:red, 208; green, 2; blue, 27 }  ,draw opacity=1 ][fill={rgb, 255:red, 208; green, 2; blue, 27 }  ,fill opacity=1 ] (450.5,100) .. controls (450.5,98.62) and (451.62,97.5) .. (453,97.5) .. controls (454.38,97.5) and (455.5,98.62) .. (455.5,100) .. controls (455.5,101.38) and (454.38,102.5) .. (453,102.5) .. controls (451.62,102.5) and (450.5,101.38) .. (450.5,100) -- cycle ;
    \draw  [color={rgb, 255:red, 208; green, 2; blue, 27 }  ,draw opacity=1 ][fill={rgb, 255:red, 208; green, 2; blue, 27 }  ,fill opacity=1 ] (430.5,100) .. controls (430.5,98.62) and (431.62,97.5) .. (433,97.5) .. controls (434.38,97.5) and (435.5,98.62) .. (435.5,100) .. controls (435.5,101.38) and (434.38,102.5) .. (433,102.5) .. controls (431.62,102.5) and (430.5,101.38) .. (430.5,100) -- cycle ;
    \draw  [color={rgb, 255:red, 83; green, 171; blue, 16 }  ,draw opacity=1 ][fill={rgb, 255:red, 83; green, 171; blue, 16 }  ,fill opacity=1 ] (490.5,20) .. controls (490.5,18.62) and (491.62,17.5) .. (493,17.5) .. controls (494.38,17.5) and (495.5,18.62) .. (495.5,20) .. controls (495.5,21.38) and (494.38,22.5) .. (493,22.5) .. controls (491.62,22.5) and (490.5,21.38) .. (490.5,20) -- cycle ;
    \draw  [color={rgb, 255:red, 83; green, 171; blue, 16 }  ,draw opacity=1 ][fill={rgb, 255:red, 83; green, 171; blue, 16 }  ,fill opacity=1 ] (510.5,40) .. controls (510.5,38.62) and (511.62,37.5) .. (513,37.5) .. controls (514.38,37.5) and (515.5,38.62) .. (515.5,40) .. controls (515.5,41.38) and (514.38,42.5) .. (513,42.5) .. controls (511.62,42.5) and (510.5,41.38) .. (510.5,40) -- cycle ;
    \draw  [color={rgb, 255:red, 83; green, 171; blue, 16 }  ,draw opacity=1 ][fill={rgb, 255:red, 83; green, 171; blue, 16 }  ,fill opacity=1 ] (510.5,60) .. controls (510.5,58.62) and (511.62,57.5) .. (513,57.5) .. controls (514.38,57.5) and (515.5,58.62) .. (515.5,60) .. controls (515.5,61.38) and (514.38,62.5) .. (513,62.5) .. controls (511.62,62.5) and (510.5,61.38) .. (510.5,60) -- cycle ;
    \draw  [color={rgb, 255:red, 83; green, 171; blue, 16 }  ,draw opacity=1 ][fill={rgb, 255:red, 83; green, 171; blue, 16 }  ,fill opacity=1 ] (510.5,80) .. controls (510.5,78.62) and (511.62,77.5) .. (513,77.5) .. controls (514.38,77.5) and (515.5,78.62) .. (515.5,80) .. controls (515.5,81.38) and (514.38,82.5) .. (513,82.5) .. controls (511.62,82.5) and (510.5,81.38) .. (510.5,80) -- cycle ;
    \draw  [color={rgb, 255:red, 83; green, 171; blue, 16 }  ,draw opacity=1 ][fill={rgb, 255:red, 83; green, 171; blue, 16 }  ,fill opacity=1 ] (510.5,100) .. controls (510.5,98.62) and (511.62,97.5) .. (513,97.5) .. controls (514.38,97.5) and (515.5,98.62) .. (515.5,100) .. controls (515.5,101.38) and (514.38,102.5) .. (513,102.5) .. controls (511.62,102.5) and (510.5,101.38) .. (510.5,100) -- cycle ;
    \draw  [color={rgb, 255:red, 83; green, 171; blue, 16 }  ,draw opacity=1 ][fill={rgb, 255:red, 83; green, 171; blue, 16 }  ,fill opacity=1 ] (510.5,120) .. controls (510.5,118.62) and (511.62,117.5) .. (513,117.5) .. controls (514.38,117.5) and (515.5,118.62) .. (515.5,120) .. controls (515.5,121.38) and (514.38,122.5) .. (513,122.5) .. controls (511.62,122.5) and (510.5,121.38) .. (510.5,120) -- cycle ;
    \draw  [color={rgb, 255:red, 83; green, 171; blue, 16 }  ,draw opacity=1 ][fill={rgb, 255:red, 83; green, 171; blue, 16 }  ,fill opacity=1 ] (510.5,140) .. controls (510.5,138.62) and (511.62,137.5) .. (513,137.5) .. controls (514.38,137.5) and (515.5,138.62) .. (515.5,140) .. controls (515.5,141.38) and (514.38,142.5) .. (513,142.5) .. controls (511.62,142.5) and (510.5,141.38) .. (510.5,140) -- cycle ;
    \draw  [color={rgb, 255:red, 83; green, 171; blue, 16 }  ,draw opacity=1 ][fill={rgb, 255:red, 83; green, 171; blue, 16 }  ,fill opacity=1 ] (510.5,160) .. controls (510.5,158.62) and (511.62,157.5) .. (513,157.5) .. controls (514.38,157.5) and (515.5,158.62) .. (515.5,160) .. controls (515.5,161.38) and (514.38,162.5) .. (513,162.5) .. controls (511.62,162.5) and (510.5,161.38) .. (510.5,160) -- cycle ;
    \draw  [color={rgb, 255:red, 83; green, 171; blue, 16 }  ,draw opacity=1 ][fill={rgb, 255:red, 83; green, 171; blue, 16 }  ,fill opacity=1 ] (510.5,20) .. controls (510.5,18.62) and (511.62,17.5) .. (513,17.5) .. controls (514.38,17.5) and (515.5,18.62) .. (515.5,20) .. controls (515.5,21.38) and (514.38,22.5) .. (513,22.5) .. controls (511.62,22.5) and (510.5,21.38) .. (510.5,20) -- cycle ;
    \draw  [color={rgb, 255:red, 83; green, 171; blue, 16 }  ,draw opacity=1 ][fill={rgb, 255:red, 83; green, 171; blue, 16 }  ,fill opacity=1 ] (370.5,20) .. controls (370.5,18.62) and (371.62,17.5) .. (373,17.5) .. controls (374.38,17.5) and (375.5,18.62) .. (375.5,20) .. controls (375.5,21.38) and (374.38,22.5) .. (373,22.5) .. controls (371.62,22.5) and (370.5,21.38) .. (370.5,20) -- cycle ;
    \draw  [color={rgb, 255:red, 83; green, 171; blue, 16 }  ,draw opacity=1 ][fill={rgb, 255:red, 83; green, 171; blue, 16 }  ,fill opacity=1 ] (370.5,40) .. controls (370.5,38.62) and (371.62,37.5) .. (373,37.5) .. controls (374.38,37.5) and (375.5,38.62) .. (375.5,40) .. controls (375.5,41.38) and (374.38,42.5) .. (373,42.5) .. controls (371.62,42.5) and (370.5,41.38) .. (370.5,40) -- cycle ;
    \draw  [color={rgb, 255:red, 83; green, 171; blue, 16 }  ,draw opacity=1 ][fill={rgb, 255:red, 83; green, 171; blue, 16 }  ,fill opacity=1 ] (370.5,60) .. controls (370.5,58.62) and (371.62,57.5) .. (373,57.5) .. controls (374.38,57.5) and (375.5,58.62) .. (375.5,60) .. controls (375.5,61.38) and (374.38,62.5) .. (373,62.5) .. controls (371.62,62.5) and (370.5,61.38) .. (370.5,60) -- cycle ;
    \draw  [color={rgb, 255:red, 83; green, 171; blue, 16 }  ,draw opacity=1 ][fill={rgb, 255:red, 83; green, 171; blue, 16 }  ,fill opacity=1 ] (370.5,80) .. controls (370.5,78.62) and (371.62,77.5) .. (373,77.5) .. controls (374.38,77.5) and (375.5,78.62) .. (375.5,80) .. controls (375.5,81.38) and (374.38,82.5) .. (373,82.5) .. controls (371.62,82.5) and (370.5,81.38) .. (370.5,80) -- cycle ;
    \draw  [color={rgb, 255:red, 83; green, 171; blue, 16 }  ,draw opacity=1 ][fill={rgb, 255:red, 83; green, 171; blue, 16 }  ,fill opacity=1 ] (370.5,100) .. controls (370.5,98.62) and (371.62,97.5) .. (373,97.5) .. controls (374.38,97.5) and (375.5,98.62) .. (375.5,100) .. controls (375.5,101.38) and (374.38,102.5) .. (373,102.5) .. controls (371.62,102.5) and (370.5,101.38) .. (370.5,100) -- cycle ;
    \draw  [color={rgb, 255:red, 83; green, 171; blue, 16 }  ,draw opacity=1 ][fill={rgb, 255:red, 83; green, 171; blue, 16 }  ,fill opacity=1 ] (370.5,120) .. controls (370.5,118.62) and (371.62,117.5) .. (373,117.5) .. controls (374.38,117.5) and (375.5,118.62) .. (375.5,120) .. controls (375.5,121.38) and (374.38,122.5) .. (373,122.5) .. controls (371.62,122.5) and (370.5,121.38) .. (370.5,120) -- cycle ;
    \draw  [color={rgb, 255:red, 83; green, 171; blue, 16 }  ,draw opacity=1 ][fill={rgb, 255:red, 83; green, 171; blue, 16 }  ,fill opacity=1 ] (370.5,140) .. controls (370.5,138.62) and (371.62,137.5) .. (373,137.5) .. controls (374.38,137.5) and (375.5,138.62) .. (375.5,140) .. controls (375.5,141.38) and (374.38,142.5) .. (373,142.5) .. controls (371.62,142.5) and (370.5,141.38) .. (370.5,140) -- cycle ;
    \draw  [color={rgb, 255:red, 83; green, 171; blue, 16 }  ,draw opacity=1 ][fill={rgb, 255:red, 83; green, 171; blue, 16 }  ,fill opacity=1 ] (370.5,160) .. controls (370.5,158.62) and (371.62,157.5) .. (373,157.5) .. controls (374.38,157.5) and (375.5,158.62) .. (375.5,160) .. controls (375.5,161.38) and (374.38,162.5) .. (373,162.5) .. controls (371.62,162.5) and (370.5,161.38) .. (370.5,160) -- cycle ;
    \draw   (68.4,89.22) .. controls (68.4,62.51) and (90.05,40.85) .. (116.76,40.85) .. controls (143.47,40.85) and (165.13,62.51) .. (165.13,89.22) .. controls (165.13,115.93) and (143.47,137.58) .. (116.76,137.58) .. controls (90.05,137.58) and (68.4,115.93) .. (68.4,89.22) -- cycle ;
    \draw  [color={rgb, 255:red, 155; green, 155; blue, 155 }  ,draw opacity=0.4 ][fill={rgb, 255:red, 155; green, 155; blue, 155 }  ,fill opacity=0.4 ] (147,79) -- (167,79) -- (167,99) -- (147,99) -- cycle ;
    \draw  [color={rgb, 255:red, 155; green, 155; blue, 155 }  ,draw opacity=0.4 ][fill={rgb, 255:red, 155; green, 155; blue, 155 }  ,fill opacity=0.4 ] (87,39) -- (107,39) -- (107,59) -- (87,59) -- cycle ;
    \draw  [color={rgb, 255:red, 155; green, 155; blue, 155 }  ,draw opacity=0.4 ][fill={rgb, 255:red, 155; green, 155; blue, 155 }  ,fill opacity=0.4 ] (106.71,39.29) -- (126.71,39.29) -- (126.71,59.29) -- (106.71,59.29) -- cycle ;
    \draw  [color={rgb, 255:red, 155; green, 155; blue, 155 }  ,draw opacity=0.4 ][fill={rgb, 255:red, 155; green, 155; blue, 155 }  ,fill opacity=0.4 ] (87,119) -- (107,119) -- (107,139) -- (87,139) -- cycle ;
    \draw  [color={rgb, 255:red, 155; green, 155; blue, 155 }  ,draw opacity=0.4 ][fill={rgb, 255:red, 155; green, 155; blue, 155 }  ,fill opacity=0.4 ] (107,119) -- (127,119) -- (127,139) -- (107,139) -- cycle ;
    \draw  [color={rgb, 255:red, 155; green, 155; blue, 155 }  ,draw opacity=0.4 ][fill={rgb, 255:red, 155; green, 155; blue, 155 }  ,fill opacity=0.4 ] (127,119) -- (147,119) -- (147,139) -- (127,139) -- cycle ;
    \draw  [color={rgb, 255:red, 0; green, 0; blue, 0 }  ,draw opacity=1 ][fill={rgb, 255:red, 74; green, 144; blue, 226 }  ,fill opacity=0.37 ] (47,19) -- (67,19) -- (67,39) -- (47,39) -- cycle ;
    \draw  [color={rgb, 255:red, 0; green, 0; blue, 0 }  ,draw opacity=1 ][fill={rgb, 255:red, 74; green, 144; blue, 226 }  ,fill opacity=0.37 ] (67,19) -- (87,19) -- (87,39) -- (67,39) -- cycle ;
    \draw  [color={rgb, 255:red, 0; green, 0; blue, 0 }  ,draw opacity=1 ][fill={rgb, 255:red, 74; green, 144; blue, 226 }  ,fill opacity=0.37 ] (87,19) -- (107,19) -- (107,39) -- (87,39) -- cycle ;
    \draw  [color={rgb, 255:red, 0; green, 0; blue, 0 }  ,draw opacity=1 ][fill={rgb, 255:red, 74; green, 144; blue, 226 }  ,fill opacity=0.37 ] (107,19) -- (127,19) -- (127,39) -- (107,39) -- cycle ;
    \draw  [color={rgb, 255:red, 0; green, 0; blue, 0 }  ,draw opacity=1 ][fill={rgb, 255:red, 74; green, 144; blue, 226 }  ,fill opacity=0.37 ] (126.71,19.29) -- (146.71,19.29) -- (146.71,39.29) -- (126.71,39.29) -- cycle ;
    \draw  [color={rgb, 255:red, 0; green, 0; blue, 0 }  ,draw opacity=1 ][fill={rgb, 255:red, 74; green, 144; blue, 226 }  ,fill opacity=0.37 ] (166.71,59.29) -- (186.71,59.29) -- (186.71,79.29) -- (166.71,79.29) -- cycle ;
    \draw  [color={rgb, 255:red, 0; green, 0; blue, 0 }  ,draw opacity=1 ][fill={rgb, 255:red, 74; green, 144; blue, 226 }  ,fill opacity=0.37 ] (146.71,19.29) -- (166.71,19.29) -- (166.71,39.29) -- (146.71,39.29) -- cycle ;
    \draw  [color={rgb, 255:red, 0; green, 0; blue, 0 }  ,draw opacity=1 ][fill={rgb, 255:red, 74; green, 144; blue, 226 }  ,fill opacity=0.37 ] (166.71,39.29) -- (186.71,39.29) -- (186.71,59.29) -- (166.71,59.29) -- cycle ;
    \draw  [color={rgb, 255:red, 0; green, 0; blue, 0 }  ,draw opacity=1 ][fill={rgb, 255:red, 74; green, 144; blue, 226 }  ,fill opacity=0.37 ] (166.71,19.29) -- (186.71,19.29) -- (186.71,39.29) -- (166.71,39.29) -- cycle ;
    \draw  [color={rgb, 255:red, 0; green, 0; blue, 0 }  ,draw opacity=1 ][fill={rgb, 255:red, 74; green, 144; blue, 226 }  ,fill opacity=0.37 ] (47,39) -- (67,39) -- (67,59) -- (47,59) -- cycle ;
    \draw  [color={rgb, 255:red, 0; green, 0; blue, 0 }  ,draw opacity=1 ][fill={rgb, 255:red, 74; green, 144; blue, 226 }  ,fill opacity=0.37 ] (47,59) -- (67,59) -- (67,79) -- (47,79) -- cycle ;
    \draw  [color={rgb, 255:red, 0; green, 0; blue, 0 }  ,draw opacity=1 ][fill={rgb, 255:red, 74; green, 144; blue, 226 }  ,fill opacity=0.37 ] (47,79) -- (67,79) -- (67,99) -- (47,99) -- cycle ;
    \draw  [color={rgb, 255:red, 0; green, 0; blue, 0 }  ,draw opacity=1 ][fill={rgb, 255:red, 74; green, 144; blue, 226 }  ,fill opacity=0.37 ] (47,99) -- (67,99) -- (67,119) -- (47,119) -- cycle ;
    \draw  [color={rgb, 255:red, 0; green, 0; blue, 0 }  ,draw opacity=1 ][fill={rgb, 255:red, 74; green, 144; blue, 226 }  ,fill opacity=0.37 ] (47,119) -- (67,119) -- (67,139) -- (47,139) -- cycle ;
    \draw  [color={rgb, 255:red, 0; green, 0; blue, 0 }  ,draw opacity=1 ][fill={rgb, 255:red, 74; green, 144; blue, 226 }  ,fill opacity=0.37 ] (67,139) -- (87,139) -- (87,159) -- (67,159) -- cycle ;
    \draw  [color={rgb, 255:red, 0; green, 0; blue, 0 }  ,draw opacity=1 ][fill={rgb, 255:red, 74; green, 144; blue, 226 }  ,fill opacity=0.37 ] (47,139) -- (67,139) -- (67,159) -- (47,159) -- cycle ;
    \draw  [color={rgb, 255:red, 0; green, 0; blue, 0 }  ,draw opacity=1 ][fill={rgb, 255:red, 74; green, 144; blue, 226 }  ,fill opacity=0.37 ] (87,139) -- (107,139) -- (107,159) -- (87,159) -- cycle ;
    \draw  [color={rgb, 255:red, 0; green, 0; blue, 0 }  ,draw opacity=1 ][fill={rgb, 255:red, 74; green, 144; blue, 226 }  ,fill opacity=0.37 ] (107,139) -- (127,139) -- (127,159) -- (107,159) -- cycle ;
    \draw  [color={rgb, 255:red, 0; green, 0; blue, 0 }  ,draw opacity=1 ][fill={rgb, 255:red, 74; green, 144; blue, 226 }  ,fill opacity=0.37 ] (127,139) -- (147,139) -- (147,159) -- (127,159) -- cycle ;
    \draw  [color={rgb, 255:red, 0; green, 0; blue, 0 }  ,draw opacity=1 ][fill={rgb, 255:red, 74; green, 144; blue, 226 }  ,fill opacity=0.37 ] (147,139) -- (167,139) -- (167,159) -- (147,159) -- cycle ;
    \draw  [color={rgb, 255:red, 0; green, 0; blue, 0 }  ,draw opacity=1 ][fill={rgb, 255:red, 74; green, 144; blue, 226 }  ,fill opacity=0.37 ] (167,139) -- (187,139) -- (187,159) -- (167,159) -- cycle ;
    \draw  [color={rgb, 255:red, 0; green, 0; blue, 0 }  ,draw opacity=1 ][fill={rgb, 255:red, 74; green, 144; blue, 226 }  ,fill opacity=0.37 ] (167,119) -- (187,119) -- (187,139) -- (167,139) -- cycle ;
    \draw  [color={rgb, 255:red, 0; green, 0; blue, 0 }  ,draw opacity=1 ][fill={rgb, 255:red, 74; green, 144; blue, 226 }  ,fill opacity=0.37 ] (167,99) -- (187,99) -- (187,119) -- (167,119) -- cycle ;
    \draw  [color={rgb, 255:red, 0; green, 0; blue, 0 }  ,draw opacity=1 ][fill={rgb, 255:red, 74; green, 144; blue, 226 }  ,fill opacity=0.37 ] (166.71,79.29) -- (186.71,79.29) -- (186.71,99.29) -- (166.71,99.29) -- cycle ;
    \draw   (394.54,90.08) .. controls (394.54,63.3) and (416.26,41.58) .. (443.04,41.58) .. controls (469.83,41.58) and (491.54,63.3) .. (491.54,90.08) .. controls (491.54,116.87) and (469.83,138.58) .. (443.04,138.58) .. controls (416.26,138.58) and (394.54,116.87) .. (394.54,90.08) -- cycle ;
    \draw  [color={rgb, 255:red, 83; green, 171; blue, 16 }  ,draw opacity=1 ][fill={rgb, 255:red, 83; green, 171; blue, 16 }  ,fill opacity=1 ] (390.5,40) .. controls (390.5,38.62) and (391.62,37.5) .. (393,37.5) .. controls (394.38,37.5) and (395.5,38.62) .. (395.5,40) .. controls (395.5,41.38) and (394.38,42.5) .. (393,42.5) .. controls (391.62,42.5) and (390.5,41.38) .. (390.5,40) -- cycle ;
    \draw  [color={rgb, 255:red, 83; green, 171; blue, 16 }  ,draw opacity=1 ][fill={rgb, 255:red, 83; green, 171; blue, 16 }  ,fill opacity=1 ] (410.5,60) .. controls (410.5,58.62) and (411.62,57.5) .. (413,57.5) .. controls (414.38,57.5) and (415.5,58.62) .. (415.5,60) .. controls (415.5,61.38) and (414.38,62.5) .. (413,62.5) .. controls (411.62,62.5) and (410.5,61.38) .. (410.5,60) -- cycle ;
    \draw  [color={rgb, 255:red, 83; green, 171; blue, 16 }  ,draw opacity=1 ][fill={rgb, 255:red, 83; green, 171; blue, 16 }  ,fill opacity=1 ] (390.5,60) .. controls (390.5,58.62) and (391.62,57.5) .. (393,57.5) .. controls (394.38,57.5) and (395.5,58.62) .. (395.5,60) .. controls (395.5,61.38) and (394.38,62.5) .. (393,62.5) .. controls (391.62,62.5) and (390.5,61.38) .. (390.5,60) -- cycle ;
    \draw  [color={rgb, 255:red, 83; green, 171; blue, 16 }  ,draw opacity=1 ][fill={rgb, 255:red, 83; green, 171; blue, 16 }  ,fill opacity=1 ] (390.5,80) .. controls (390.5,78.62) and (391.62,77.5) .. (393,77.5) .. controls (394.38,77.5) and (395.5,78.62) .. (395.5,80) .. controls (395.5,81.38) and (394.38,82.5) .. (393,82.5) .. controls (391.62,82.5) and (390.5,81.38) .. (390.5,80) -- cycle ;
    \draw  [color={rgb, 255:red, 83; green, 171; blue, 16 }  ,draw opacity=1 ][fill={rgb, 255:red, 83; green, 171; blue, 16 }  ,fill opacity=1 ] (390.5,100) .. controls (390.5,98.62) and (391.62,97.5) .. (393,97.5) .. controls (394.38,97.5) and (395.5,98.62) .. (395.5,100) .. controls (395.5,101.38) and (394.38,102.5) .. (393,102.5) .. controls (391.62,102.5) and (390.5,101.38) .. (390.5,100) -- cycle ;
    \draw  [color={rgb, 255:red, 83; green, 171; blue, 16 }  ,draw opacity=1 ][fill={rgb, 255:red, 83; green, 171; blue, 16 }  ,fill opacity=1 ] (410.5,120) .. controls (410.5,118.62) and (411.62,117.5) .. (413,117.5) .. controls (414.38,117.5) and (415.5,118.62) .. (415.5,120) .. controls (415.5,121.38) and (414.38,122.5) .. (413,122.5) .. controls (411.62,122.5) and (410.5,121.38) .. (410.5,120) -- cycle ;
    \draw  [color={rgb, 255:red, 83; green, 171; blue, 16 }  ,draw opacity=1 ][fill={rgb, 255:red, 83; green, 171; blue, 16 }  ,fill opacity=1 ] (430.5,140) .. controls (430.5,138.62) and (431.62,137.5) .. (433,137.5) .. controls (434.38,137.5) and (435.5,138.62) .. (435.5,140) .. controls (435.5,141.38) and (434.38,142.5) .. (433,142.5) .. controls (431.62,142.5) and (430.5,141.38) .. (430.5,140) -- cycle ;
    \draw  [color={rgb, 255:red, 83; green, 171; blue, 16 }  ,draw opacity=1 ][fill={rgb, 255:red, 83; green, 171; blue, 16 }  ,fill opacity=1 ] (450.5,160) .. controls (450.5,158.62) and (451.62,157.5) .. (453,157.5) .. controls (454.38,157.5) and (455.5,158.62) .. (455.5,160) .. controls (455.5,161.38) and (454.38,162.5) .. (453,162.5) .. controls (451.62,162.5) and (450.5,161.38) .. (450.5,160) -- cycle ;
    \draw  [color={rgb, 255:red, 83; green, 171; blue, 16 }  ,draw opacity=1 ][fill={rgb, 255:red, 83; green, 171; blue, 16 }  ,fill opacity=1 ] (390.5,140) .. controls (390.5,138.62) and (391.62,137.5) .. (393,137.5) .. controls (394.38,137.5) and (395.5,138.62) .. (395.5,140) .. controls (395.5,141.38) and (394.38,142.5) .. (393,142.5) .. controls (391.62,142.5) and (390.5,141.38) .. (390.5,140) -- cycle ;
    \draw  [color={rgb, 255:red, 83; green, 171; blue, 16 }  ,draw opacity=1 ][fill={rgb, 255:red, 83; green, 171; blue, 16 }  ,fill opacity=1 ] (410.5,160) .. controls (410.5,158.62) and (411.62,157.5) .. (413,157.5) .. controls (414.38,157.5) and (415.5,158.62) .. (415.5,160) .. controls (415.5,161.38) and (414.38,162.5) .. (413,162.5) .. controls (411.62,162.5) and (410.5,161.38) .. (410.5,160) -- cycle ;
    \draw  [color={rgb, 255:red, 83; green, 171; blue, 16 }  ,draw opacity=1 ][fill={rgb, 255:red, 83; green, 171; blue, 16 }  ,fill opacity=1 ] (390.5,120) .. controls (390.5,118.62) and (391.62,117.5) .. (393,117.5) .. controls (394.38,117.5) and (395.5,118.62) .. (395.5,120) .. controls (395.5,121.38) and (394.38,122.5) .. (393,122.5) .. controls (391.62,122.5) and (390.5,121.38) .. (390.5,120) -- cycle ;
    \draw  [color={rgb, 255:red, 83; green, 171; blue, 16 }  ,draw opacity=1 ][fill={rgb, 255:red, 83; green, 171; blue, 16 }  ,fill opacity=1 ] (410.5,140) .. controls (410.5,138.62) and (411.62,137.5) .. (413,137.5) .. controls (414.38,137.5) and (415.5,138.62) .. (415.5,140) .. controls (415.5,141.38) and (414.38,142.5) .. (413,142.5) .. controls (411.62,142.5) and (410.5,141.38) .. (410.5,140) -- cycle ;
    \draw  [color={rgb, 255:red, 83; green, 171; blue, 16 }  ,draw opacity=1 ][fill={rgb, 255:red, 83; green, 171; blue, 16 }  ,fill opacity=1 ] (430.5,160) .. controls (430.5,158.62) and (431.62,157.5) .. (433,157.5) .. controls (434.38,157.5) and (435.5,158.62) .. (435.5,160) .. controls (435.5,161.38) and (434.38,162.5) .. (433,162.5) .. controls (431.62,162.5) and (430.5,161.38) .. (430.5,160) -- cycle ;
    \draw  [color={rgb, 255:red, 83; green, 171; blue, 16 }  ,draw opacity=1 ][fill={rgb, 255:red, 83; green, 171; blue, 16 }  ,fill opacity=1 ] (390.5,160) .. controls (390.5,158.62) and (391.62,157.5) .. (393,157.5) .. controls (394.38,157.5) and (395.5,158.62) .. (395.5,160) .. controls (395.5,161.38) and (394.38,162.5) .. (393,162.5) .. controls (391.62,162.5) and (390.5,161.38) .. (390.5,160) -- cycle ;
    \draw  [color={rgb, 255:red, 83; green, 171; blue, 16 }  ,draw opacity=1 ][fill={rgb, 255:red, 83; green, 171; blue, 16 }  ,fill opacity=1 ] (390.5,20) .. controls (390.5,18.62) and (391.62,17.5) .. (393,17.5) .. controls (394.38,17.5) and (395.5,18.62) .. (395.5,20) .. controls (395.5,21.38) and (394.38,22.5) .. (393,22.5) .. controls (391.62,22.5) and (390.5,21.38) .. (390.5,20) -- cycle ;
    \draw  [color={rgb, 255:red, 83; green, 171; blue, 16 }  ,draw opacity=1 ][fill={rgb, 255:red, 83; green, 171; blue, 16 }  ,fill opacity=1 ] (410.5,40) .. controls (410.5,38.62) and (411.62,37.5) .. (413,37.5) .. controls (414.38,37.5) and (415.5,38.62) .. (415.5,40) .. controls (415.5,41.38) and (414.38,42.5) .. (413,42.5) .. controls (411.62,42.5) and (410.5,41.38) .. (410.5,40) -- cycle ;
    \draw  [color={rgb, 255:red, 83; green, 171; blue, 16 }  ,draw opacity=1 ][fill={rgb, 255:red, 83; green, 171; blue, 16 }  ,fill opacity=1 ] (410.5,20) .. controls (410.5,18.62) and (411.62,17.5) .. (413,17.5) .. controls (414.38,17.5) and (415.5,18.62) .. (415.5,20) .. controls (415.5,21.38) and (414.38,22.5) .. (413,22.5) .. controls (411.62,22.5) and (410.5,21.38) .. (410.5,20) -- cycle ;
    \draw  [color={rgb, 255:red, 83; green, 171; blue, 16 }  ,draw opacity=1 ][fill={rgb, 255:red, 83; green, 171; blue, 16 }  ,fill opacity=1 ] (430.5,40) .. controls (430.5,38.62) and (431.62,37.5) .. (433,37.5) .. controls (434.38,37.5) and (435.5,38.62) .. (435.5,40) .. controls (435.5,41.38) and (434.38,42.5) .. (433,42.5) .. controls (431.62,42.5) and (430.5,41.38) .. (430.5,40) -- cycle ;
    \draw  [color={rgb, 255:red, 83; green, 171; blue, 16 }  ,draw opacity=1 ][fill={rgb, 255:red, 83; green, 171; blue, 16 }  ,fill opacity=1 ] (430.5,20) .. controls (430.5,18.62) and (431.62,17.5) .. (433,17.5) .. controls (434.38,17.5) and (435.5,18.62) .. (435.5,20) .. controls (435.5,21.38) and (434.38,22.5) .. (433,22.5) .. controls (431.62,22.5) and (430.5,21.38) .. (430.5,20) -- cycle ;
    \draw  [color={rgb, 255:red, 83; green, 171; blue, 16 }  ,draw opacity=1 ][fill={rgb, 255:red, 83; green, 171; blue, 16 }  ,fill opacity=1 ] (450.5,40) .. controls (450.5,38.62) and (451.62,37.5) .. (453,37.5) .. controls (454.38,37.5) and (455.5,38.62) .. (455.5,40) .. controls (455.5,41.38) and (454.38,42.5) .. (453,42.5) .. controls (451.62,42.5) and (450.5,41.38) .. (450.5,40) -- cycle ;
    \draw  [color={rgb, 255:red, 83; green, 171; blue, 16 }  ,draw opacity=1 ][fill={rgb, 255:red, 83; green, 171; blue, 16 }  ,fill opacity=1 ] (470.5,60) .. controls (470.5,58.62) and (471.62,57.5) .. (473,57.5) .. controls (474.38,57.5) and (475.5,58.62) .. (475.5,60) .. controls (475.5,61.38) and (474.38,62.5) .. (473,62.5) .. controls (471.62,62.5) and (470.5,61.38) .. (470.5,60) -- cycle ;
    \draw  [color={rgb, 255:red, 83; green, 171; blue, 16 }  ,draw opacity=1 ][fill={rgb, 255:red, 83; green, 171; blue, 16 }  ,fill opacity=1 ] (490.5,80) .. controls (490.5,78.62) and (491.62,77.5) .. (493,77.5) .. controls (494.38,77.5) and (495.5,78.62) .. (495.5,80) .. controls (495.5,81.38) and (494.38,82.5) .. (493,82.5) .. controls (491.62,82.5) and (490.5,81.38) .. (490.5,80) -- cycle ;
    \draw  [color={rgb, 255:red, 83; green, 171; blue, 16 }  ,draw opacity=1 ][fill={rgb, 255:red, 83; green, 171; blue, 16 }  ,fill opacity=1 ] (450.5,20) .. controls (450.5,18.62) and (451.62,17.5) .. (453,17.5) .. controls (454.38,17.5) and (455.5,18.62) .. (455.5,20) .. controls (455.5,21.38) and (454.38,22.5) .. (453,22.5) .. controls (451.62,22.5) and (450.5,21.38) .. (450.5,20) -- cycle ;
    \draw  [color={rgb, 255:red, 83; green, 171; blue, 16 }  ,draw opacity=1 ][fill={rgb, 255:red, 83; green, 171; blue, 16 }  ,fill opacity=1 ] (470.5,40) .. controls (470.5,38.62) and (471.62,37.5) .. (473,37.5) .. controls (474.38,37.5) and (475.5,38.62) .. (475.5,40) .. controls (475.5,41.38) and (474.38,42.5) .. (473,42.5) .. controls (471.62,42.5) and (470.5,41.38) .. (470.5,40) -- cycle ;
    \draw  [color={rgb, 255:red, 83; green, 171; blue, 16 }  ,draw opacity=1 ][fill={rgb, 255:red, 83; green, 171; blue, 16 }  ,fill opacity=1 ] (490.5,60) .. controls (490.5,58.62) and (491.62,57.5) .. (493,57.5) .. controls (494.38,57.5) and (495.5,58.62) .. (495.5,60) .. controls (495.5,61.38) and (494.38,62.5) .. (493,62.5) .. controls (491.62,62.5) and (490.5,61.38) .. (490.5,60) -- cycle ;
    \draw  [color={rgb, 255:red, 83; green, 171; blue, 16 }  ,draw opacity=1 ][fill={rgb, 255:red, 83; green, 171; blue, 16 }  ,fill opacity=1 ] (470.5,20) .. controls (470.5,18.62) and (471.62,17.5) .. (473,17.5) .. controls (474.38,17.5) and (475.5,18.62) .. (475.5,20) .. controls (475.5,21.38) and (474.38,22.5) .. (473,22.5) .. controls (471.62,22.5) and (470.5,21.38) .. (470.5,20) -- cycle ;
    \draw  [color={rgb, 255:red, 83; green, 171; blue, 16 }  ,draw opacity=1 ][fill={rgb, 255:red, 83; green, 171; blue, 16 }  ,fill opacity=1 ] (490.5,40) .. controls (490.5,38.62) and (491.62,37.5) .. (493,37.5) .. controls (494.38,37.5) and (495.5,38.62) .. (495.5,40) .. controls (495.5,41.38) and (494.38,42.5) .. (493,42.5) .. controls (491.62,42.5) and (490.5,41.38) .. (490.5,40) -- cycle ;
    \draw  [color={rgb, 255:red, 83; green, 171; blue, 16 }  ,draw opacity=1 ][fill={rgb, 255:red, 83; green, 171; blue, 16 }  ,fill opacity=1 ] (450.5,140) .. controls (450.5,138.62) and (451.62,137.5) .. (453,137.5) .. controls (454.38,137.5) and (455.5,138.62) .. (455.5,140) .. controls (455.5,141.38) and (454.38,142.5) .. (453,142.5) .. controls (451.62,142.5) and (450.5,141.38) .. (450.5,140) -- cycle ;
    \draw  [color={rgb, 255:red, 83; green, 171; blue, 16 }  ,draw opacity=1 ][fill={rgb, 255:red, 83; green, 171; blue, 16 }  ,fill opacity=1 ] (470.5,160) .. controls (470.5,158.62) and (471.62,157.5) .. (473,157.5) .. controls (474.38,157.5) and (475.5,158.62) .. (475.5,160) .. controls (475.5,161.38) and (474.38,162.5) .. (473,162.5) .. controls (471.62,162.5) and (470.5,161.38) .. (470.5,160) -- cycle ;
    \draw  [color={rgb, 255:red, 83; green, 171; blue, 16 }  ,draw opacity=1 ][fill={rgb, 255:red, 83; green, 171; blue, 16 }  ,fill opacity=1 ] (470.5,120) .. controls (470.5,118.62) and (471.62,117.5) .. (473,117.5) .. controls (474.38,117.5) and (475.5,118.62) .. (475.5,120) .. controls (475.5,121.38) and (474.38,122.5) .. (473,122.5) .. controls (471.62,122.5) and (470.5,121.38) .. (470.5,120) -- cycle ;
    \draw  [color={rgb, 255:red, 83; green, 171; blue, 16 }  ,draw opacity=1 ][fill={rgb, 255:red, 83; green, 171; blue, 16 }  ,fill opacity=1 ] (490.5,140) .. controls (490.5,138.62) and (491.62,137.5) .. (493,137.5) .. controls (494.38,137.5) and (495.5,138.62) .. (495.5,140) .. controls (495.5,141.38) and (494.38,142.5) .. (493,142.5) .. controls (491.62,142.5) and (490.5,141.38) .. (490.5,140) -- cycle ;
    \draw  [color={rgb, 255:red, 83; green, 171; blue, 16 }  ,draw opacity=1 ][fill={rgb, 255:red, 83; green, 171; blue, 16 }  ,fill opacity=1 ] (470.5,140) .. controls (470.5,138.62) and (471.62,137.5) .. (473,137.5) .. controls (474.38,137.5) and (475.5,138.62) .. (475.5,140) .. controls (475.5,141.38) and (474.38,142.5) .. (473,142.5) .. controls (471.62,142.5) and (470.5,141.38) .. (470.5,140) -- cycle ;
    \draw  [color={rgb, 255:red, 83; green, 171; blue, 16 }  ,draw opacity=1 ][fill={rgb, 255:red, 83; green, 171; blue, 16 }  ,fill opacity=1 ] (490.5,160) .. controls (490.5,158.62) and (491.62,157.5) .. (493,157.5) .. controls (494.38,157.5) and (495.5,158.62) .. (495.5,160) .. controls (495.5,161.38) and (494.38,162.5) .. (493,162.5) .. controls (491.62,162.5) and (490.5,161.38) .. (490.5,160) -- cycle ;
    \draw  [color={rgb, 255:red, 83; green, 171; blue, 16 }  ,draw opacity=1 ][fill={rgb, 255:red, 83; green, 171; blue, 16 }  ,fill opacity=1 ] (490.5,120) .. controls (490.5,118.62) and (491.62,117.5) .. (493,117.5) .. controls (494.38,117.5) and (495.5,118.62) .. (495.5,120) .. controls (495.5,121.38) and (494.38,122.5) .. (493,122.5) .. controls (491.62,122.5) and (490.5,121.38) .. (490.5,120) -- cycle ;
    \draw  [color={rgb, 255:red, 83; green, 171; blue, 16 }  ,draw opacity=1 ][fill={rgb, 255:red, 83; green, 171; blue, 16 }  ,fill opacity=1 ] (490.5,100) .. controls (490.5,98.62) and (491.62,97.5) .. (493,97.5) .. controls (494.38,97.5) and (495.5,98.62) .. (495.5,100) .. controls (495.5,101.38) and (494.38,102.5) .. (493,102.5) .. controls (491.62,102.5) and (490.5,101.38) .. (490.5,100) -- cycle ;
    \draw  [color={rgb, 255:red, 83; green, 171; blue, 16 }  ,draw opacity=1 ][fill={rgb, 255:red, 83; green, 171; blue, 16 }  ,fill opacity=1 ] (430.5,60) .. controls (430.5,58.62) and (431.62,57.5) .. (433,57.5) .. controls (434.38,57.5) and (435.5,58.62) .. (435.5,60) .. controls (435.5,61.38) and (434.38,62.5) .. (433,62.5) .. controls (431.62,62.5) and (430.5,61.38) .. (430.5,60) -- cycle ;
    \draw  [color={rgb, 255:red, 83; green, 171; blue, 16 }  ,draw opacity=1 ][fill={rgb, 255:red, 83; green, 171; blue, 16 }  ,fill opacity=1 ] (450.5,60) .. controls (450.5,58.62) and (451.62,57.5) .. (453,57.5) .. controls (454.38,57.5) and (455.5,58.62) .. (455.5,60) .. controls (455.5,61.38) and (454.38,62.5) .. (453,62.5) .. controls (451.62,62.5) and (450.5,61.38) .. (450.5,60) -- cycle ;
    \draw  [color={rgb, 255:red, 83; green, 171; blue, 16 }  ,draw opacity=1 ][fill={rgb, 255:red, 83; green, 171; blue, 16 }  ,fill opacity=1 ] (470.5,80) .. controls (470.5,78.62) and (471.62,77.5) .. (473,77.5) .. controls (474.38,77.5) and (475.5,78.62) .. (475.5,80) .. controls (475.5,81.38) and (474.38,82.5) .. (473,82.5) .. controls (471.62,82.5) and (470.5,81.38) .. (470.5,80) -- cycle ;
    \draw  [color={rgb, 255:red, 83; green, 171; blue, 16 }  ,draw opacity=1 ][fill={rgb, 255:red, 83; green, 171; blue, 16 }  ,fill opacity=1 ] (410.5,80) .. controls (410.5,78.62) and (411.62,77.5) .. (413,77.5) .. controls (414.38,77.5) and (415.5,78.62) .. (415.5,80) .. controls (415.5,81.38) and (414.38,82.5) .. (413,82.5) .. controls (411.62,82.5) and (410.5,81.38) .. (410.5,80) -- cycle ;
    \draw  [color={rgb, 255:red, 83; green, 171; blue, 16 }  ,draw opacity=1 ][fill={rgb, 255:red, 83; green, 171; blue, 16 }  ,fill opacity=1 ] (410.5,100) .. controls (410.5,98.62) and (411.62,97.5) .. (413,97.5) .. controls (414.38,97.5) and (415.5,98.62) .. (415.5,100) .. controls (415.5,101.38) and (414.38,102.5) .. (413,102.5) .. controls (411.62,102.5) and (410.5,101.38) .. (410.5,100) -- cycle ;
    \draw  [color={rgb, 255:red, 83; green, 171; blue, 16 }  ,draw opacity=1 ][fill={rgb, 255:red, 83; green, 171; blue, 16 }  ,fill opacity=1 ] (430.5,120) .. controls (430.5,118.62) and (431.62,117.5) .. (433,117.5) .. controls (434.38,117.5) and (435.5,118.62) .. (435.5,120) .. controls (435.5,121.38) and (434.38,122.5) .. (433,122.5) .. controls (431.62,122.5) and (430.5,121.38) .. (430.5,120) -- cycle ;
    \draw  [color={rgb, 255:red, 83; green, 171; blue, 16 }  ,draw opacity=1 ][fill={rgb, 255:red, 83; green, 171; blue, 16 }  ,fill opacity=1 ] (450.5,120) .. controls (450.5,118.62) and (451.62,117.5) .. (453,117.5) .. controls (454.38,117.5) and (455.5,118.62) .. (455.5,120) .. controls (455.5,121.38) and (454.38,122.5) .. (453,122.5) .. controls (451.62,122.5) and (450.5,121.38) .. (450.5,120) -- cycle ;
    \draw  [color={rgb, 255:red, 83; green, 171; blue, 16 }  ,draw opacity=1 ][fill={rgb, 255:red, 83; green, 171; blue, 16 }  ,fill opacity=1 ] (470.5,100) .. controls (470.5,98.62) and (471.62,97.5) .. (473,97.5) .. controls (474.38,97.5) and (475.5,98.62) .. (475.5,100) .. controls (475.5,101.38) and (474.38,102.5) .. (473,102.5) .. controls (471.62,102.5) and (470.5,101.38) .. (470.5,100) -- cycle ;
    \draw  [color={rgb, 255:red, 155; green, 155; blue, 155 }  ,draw opacity=0.4 ][fill={rgb, 255:red, 155; green, 155; blue, 155 }  ,fill opacity=0.4 ] (147,119) -- (167,119) -- (167,139) -- (147,139) -- cycle ;
    \draw  [color={rgb, 255:red, 155; green, 155; blue, 155 }  ,draw opacity=0.4 ][fill={rgb, 255:red, 155; green, 155; blue, 155 }  ,fill opacity=0.4 ] (67,119) -- (87,119) -- (87,139) -- (67,139) -- cycle ;
    \draw  [color={rgb, 255:red, 155; green, 155; blue, 155 }  ,draw opacity=0.4 ][fill={rgb, 255:red, 155; green, 155; blue, 155 }  ,fill opacity=0.4 ] (67,39) -- (87,39) -- (87,59) -- (67,59) -- cycle ;
    \draw  [color={rgb, 255:red, 155; green, 155; blue, 155 }  ,draw opacity=0.4 ][fill={rgb, 255:red, 155; green, 155; blue, 155 }  ,fill opacity=0.4 ] (146.71,39.29) -- (166.71,39.29) -- (166.71,59.29) -- (146.71,59.29) -- cycle ;
        
        \draw (373,20) node[cross,blue] {};
        \draw (393,20) node[cross,blue] {};
        \draw (413,20) node[cross,blue] {};
        \draw (433,20) node[cross,blue] {};
        \draw (453,20) node[cross,blue] {};
        \draw (473,20) node[cross,blue] {};
        \draw (493,20) node[cross,blue] {};
        \draw (513,20) node[cross,blue] {};
        \draw (373,40) node[cross,blue] {};
        \draw (393,40) node[cross,blue] {};
        \draw (413,40) node[cross,blue] {};
        \draw (433,40) node[cross,blue] {};
        \draw (453,40) node[cross,blue] {};
        \draw (473,40) node[cross,blue] {};
        \draw (493,40) node[cross,blue] {};
        \draw (513,40) node[cross,blue] {};
        \draw (373,60) node[cross,blue] {};
        \draw (393,60) node[cross,blue] {};
        \draw (493,60) node[cross,blue] {}; 
        \draw (513,60) node[cross,blue] {}; 
        \draw (373,80) node[cross,blue] {};
        \draw (393,80) node[cross,blue] {};
        \draw (493,80) node[cross,blue] {}; 
        \draw (513,80) node[cross,blue] {}; 
        \draw (373,100) node[cross,blue] {};
        \draw (393,100) node[cross,blue] {};
        \draw (493,100) node[cross,blue] {}; 
        \draw (513,100) node[cross,blue] {}; 
        \draw (373,120) node[cross,blue] {};
        \draw (393,120) node[cross,blue] {};
        \draw (493,120) node[cross,blue] {}; 
        \draw (513,120) node[cross,blue] {}; 
        \draw (373,140) node[cross,blue] {};
        \draw (393,140) node[cross,blue] {};
        \draw (413,140) node[cross,blue] {};
        \draw (433,140) node[cross,blue] {};
        \draw (453,140) node[cross,blue] {};
        \draw (473,140) node[cross,blue] {};
        \draw (493,140) node[cross,blue] {};
        \draw (513,140) node[cross,blue] {};
        \draw (373,160) node[cross,blue] {};
        \draw (393,160) node[cross,blue] {};
        \draw (413,160) node[cross,blue] {};
        \draw (433,160) node[cross,blue] {};
        \draw (453,160) node[cross,blue] {};
        \draw (473,160) node[cross,blue] {};
        \draw (493,160) node[cross,blue] {};
        \draw (513,160) node[cross,blue] {};

        \draw (541,60) node[cross,blue] {};
        
        \draw  [color={rgb, 255:red, 208; green, 2; blue, 27 }  ,draw opacity=1 ][fill={rgb, 255:red, 208; green, 2; blue, 27 }  ,fill opacity=1 ] (538.42,91.08) .. controls (538.42,89.7) and (539.54,88.58) .. (540.92,88.58) .. controls (542.3,88.58) and (543.42,89.7) .. (543.42,91.08) .. controls (543.42,92.46) and (542.3,93.58) .. (540.92,93.58) .. controls (539.54,93.58) and (538.42,92.46) .. (538.42,91.08) -- cycle ;
        \draw  [color={rgb, 255:red, 83; green, 171; blue, 16 }  ,draw opacity=1 ][fill={rgb, 255:red, 83; green, 171; blue, 16 }  ,fill opacity=1 ] (538.42,121.08) .. controls (538.42,119.7) and (539.54,118.58) .. (540.92,118.58) .. controls (542.3,118.58) and (543.42,119.7) .. (543.42,121.08) .. controls (543.42,122.46) and (542.3,123.58) .. (540.92,123.58) .. controls (539.54,123.58) and (538.42,122.46) .. (538.42,121.08) -- cycle ;

        \draw (245,52) node [anchor=north west][inner sep=0.75pt]  [font=\small] [align=left] {internal cell};
        \draw (245,83) node [anchor=north west][inner sep=0.75pt]  [font=\small] [align=left] {cut cell};
        \draw (245,115) node [anchor=north west][inner sep=0.75pt]  [font=\small] [align=left] {external cell};
        \draw (556,54) node [anchor=north west][inner sep=0.75pt]  [font=\small] [align=left] {internal node};
        \draw (556,83) node [anchor=north west][inner sep=0.75pt]  [font=\small] [align=left] {external node};
        \draw (556,115) node [anchor=north west][inner sep=0.75pt]  [font=\small] [align=left] {aggregated node};
    \end{tikzpicture}
    \subfloat[\label{fig: embedded cells} Classification of embedded cells.]{\hspace{.5\linewidth}}
    \subfloat[\label{fig: embedded nodes} Classification of embedded nodes.]{\hspace{.5\linewidth}}
    \caption{Cells and nodes in the aggregated cell method in the case of a circular hole cut from a background grid.}
    \label{fig: embedded setup}
\end{figure}
In this section, we assume that our problem is stated in $H^1(\Omega)$ and we use grad-conforming \ac{fe} spaces with full $\mathcal{C}^0$ continuity, whose basis functions can be constructed using Lagrangian bases (see, e.g., \cite{Ern2021}). With the classification of cells proposed above, we can now define the \ac{fe} spaces used in the aggregated cell method. First, we define a standard background \ac{fe} space over the background grid
\begin{equation}
    \widehat{\mathcal{V}}_h \doteq \{v_h \in \mathcal{C}^0(\Omega) \ : \ v_h|_K \in \mathcal{Q}_p(K) \ \forall \ K \in \widehat{\Omega}_h\},
\end{equation}
where $\mathcal{Q}_p(K)$ the space of polynomials of degree at most $p$ in each coordinate. Similarly, we can define the interior and active \ac{fe} spaces  $\mathcal{V}_h^{\texttt{IN}}$ and $\mathcal{V}_h^{\texttt{ACT}}$ over the union of interior cells and the union of active cells, resp. We consider the same order $p \geq 1$ for all these spaces, and we denote the number of \acp{dof} for background, internal and active spaces by $\widehat{\mathcal{N}}_h$, $\mathcal{N}_h$, $\mathcal{N}^{\texttt{ACT}}_h$, resp. 
 
Clearly, the internal and the background spaces are not suitable choices for the \ac{fe} approximation over $\widetilde{\Omega}$. Unfitted methods that simply employ $\mathcal{V}_h^{\texttt{ACT}}$ as approximation space are affected by severe ill-conditioning due to the presence of small cut cells \cite{DEPRENTER2017297}. There are two main ways to address this issue: either by using stabilization techniques, e.g., in the form of weakly consistent ghost penalty stabilization \cite{Burman2010}, or by using a more robust \ac{fe} space, e.g., in the form of an aggregated \ac{fe} space \cite{BADIA2018533}. In this work, we focus on the latter approach, though the \acp{rom} we propose can be applied to other unfitted methods as well.

The aggregated \ac{fe} method presented in \cite{BADIA2018533} addresses this issue by assigning a set of judiciously constructed linear constraints to the aggregated \acp{dof} associated with cut cells (see Fig.~\ref{fig: embedded nodes}). The \ac{fe} interpolation on the constrained \acp{dof} is uniquely determined by the internal values and the imposed constraints. We denote this space -- also called the aggregated \ac{fe} space -- by $\mathcal{V}_h$. By construction, the following inclusions hold:
\begin{equation*}
    \mathcal{V}^{\texttt{IN}}_h \subseteq \mathcal{V}_h \subseteq \mathcal{V}^{\texttt{ACT}}_h.   
\end{equation*}
In particular, $\mathcal{V}_h$ is entirely characterized by the internal space, and an extension operator to the cut cells, whose expression is determined by the aforementioned constraints. We denote this so-called discrete extension operator by $\mathcal{E}_h$, and define the aggregated space as
\begin{equation} 
    \label{eq: aggregated fe space}
    \mathcal{V}_h \doteq \mathcal{E}_h\circ\mathcal{V}_h^{\texttt{IN}},
\end{equation}
The dimension of $\mathcal{V}_h$ is therefore $\mathcal{N}_h$ -- that is, the number of internal \acp{dof}. The formulation of $\mathcal{E}_h$ along with further details can be found in \cite{BADIA2018533}.

\subsection{Full order model}
\label{sec: full order model}

Let us consider the following parameterized \ac{pde} on $\Omega(\bm{\mu})$: find $u(\bm{\mu}) \in H^1(\Omega(\bm{\mu}))$ such that 
\begin{equation}
    \label{eq: strong form pde}
    \left\{
    \begin{alignedat}{3}
        \mathcal{A}(u(\bm{\mu});\bm{\mu}) &\,=\, 0  &&\quad &&\text{in } \Omega(\bm{\mu}), \\
        u(\bm{\mu}) &\,=\, u_D(\bm{\mu}) &&\quad &&\text{on } \Gamma_D(\bm{\mu}), \\
        \bm{\nabla} u(\bm{\mu}) \cdot \vec{n}(\bm{\mu}) &\,=\, u_N(\bm{\mu}) &&\quad &&\text{on } \Gamma_N(\bm{\mu}).
    \end{alignedat}
\right.
\end{equation}
Without loss of generality, we assume that the Dirichlet boundary $\Gamma_D(\bm{\mu})$ coincides with the deformed boundary $\Gamma(\bm{\mu})$, whereas the Neumann boundary $\Gamma_N(\bm{\mu}) \equiv \partial\Omega(\bm{\mu}) \setminus \Gamma(\bm{\mu})$. Their respective data satisfy $u_D(\bm{\mu}) \in H^{1/2}(\Gamma_D(\bm{\mu}))$ and $u_N(\bm{\mu}) \in H^{-1/2}(\Gamma_N(\bm{\mu}))$.
Moreover, we assume for generality that the differential operator $\mathcal{A}(\cdot;\bm{\mu}) \in H^{-1}(\Omega(\bm{\mu}))$ is nonlinear, and continuously differentiable in its first argument. Then, the weak formulation associated with \eqref{eq: strong form pde} on the aggregated \ac{fe} space \eqref{eq: aggregated fe space} reads as: find $u_h(\bm{\mu}) \in \mathcal{V}_h$ such that 
\begin{equation}
    \label{eq: weak form pde}
    \widebar{a}(u_h(\bm{\mu}),v_h;\bm{\mu}) = 0 \quad \forall \ v_h \in \mathcal{V}_h,
\end{equation}
where
\begin{equation*}
    \begin{split}
        \widebar{a}(u_h,v_h;\bm{\mu}) = a(u_h,v_h;\bm{\mu}) &+ \int_{\Gamma_D(\bm{\mu})} \tau_h (u_h - u_D(\bm{\mu})) v_h - (\bm{\nabla} u_h \cdot \vec{n}(\bm{\mu})) v_h - (\bm{\nabla} v_h \cdot \vec{n}(\bm{\mu})) (u_h - u_D(\bm{\mu})) 
        \\ &- \int_{\Gamma_N(\bm{\mu})} u_N(\bm{\mu}) v_h.
    \end{split}
\end{equation*}
As is customary in unfitted \ac{fe} applications, the Dirichlet condition is enforced weakly via the (symmetric) Nitsche penalty method. The quantity $\tau_h > 0$ denotes the (global) Nitsche penalty constant, and the form $a$ and arises from testing the differential operator $\mathcal{A}$ against the test function $v_h$. The analysis in \cite{BADIA2018533} shows that $\tau_h$ is independent of the minimum cut cell volume, a major difference with respect to standard unfitted \ac{fe} methods. In particular, $\tau_h = \eta h^{-1}$ for a sufficiently large (mesh-independent) positive constant $\eta$ \cite{doi:10.1137/S0036142901384162}, similarly to standard body-fitted \ac{fe} schemes. Although $\tau_h$ should be defined with respect to the deformed mesh size, in practice it can be obtained from the penalty constant in the reference configuration, scaled by a factor accounting for the deformation map. This approach is reasonable here, as we do not operate in regimes involving large deformations. The scaling can be incorporated into an adjusted value of $\eta$ that ensures coercivity for any $\bm{\mu}$. Henceforth, we make no distinction between penalty constants in reference or in deformed configurations. 
Exploiting the differentiability of $\mathcal{A}$, we perform a Newton linearization of \eqref{eq: weak form pde} and solve the resulting problem via the Newton-Raphson scheme. The algebraic form of the linearized system is
\begin{equation}
	\label{eq: fom}
	\text{Given } \bm{w}^{(0)} \in \R^{\mathcal{N}_h}, \text{ compute } \bm{J}(\bm{w}^{(k)};\bm{\mu})\delta\bm{w}^{(k)} = -\bm{r}(\bm{w}^{(k)};\bm{\mu}), \text{ and update } \bm{w}^{(k+1)} = \bm{w}^{(k)} + \delta\bm{w}^{(k)}.
\end{equation}
The iteration is carried out for $k = 1,2,\ldots,$ until a chosen stopping condition is met, such as
\begin{equation}
	\|\delta\bm{w}^{(k)}\| < \kappa,
\end{equation}
with $\kappa$ denoting a sufficiently small threshold. Upon convergence, we define the solution as $\bm{u}(\bm{\mu}) \doteq \bm{w}^{(k+1)}$. In \eqref{eq: fom}, $\bm{J}(\bm{w};\bm{\mu}) \in \R^{\mathcal{N}_h \times \mathcal{N}_h}$ denotes the Jacobian matrix of the nonlinear operator, assembled through numerical integration of the Fréchet derivative \cite{salsa2016numerical,quarteroni2016numerical} of $\widebar{a}$, while $\bm{r}(\bm{w};\bm{\mu}) \in \R^{\mathcal{N}_h}$ represents the residual vector obtained by numerically integrating $\widebar{a}$. 

\subsection{Nodal value extension}
\label{subs: extension}

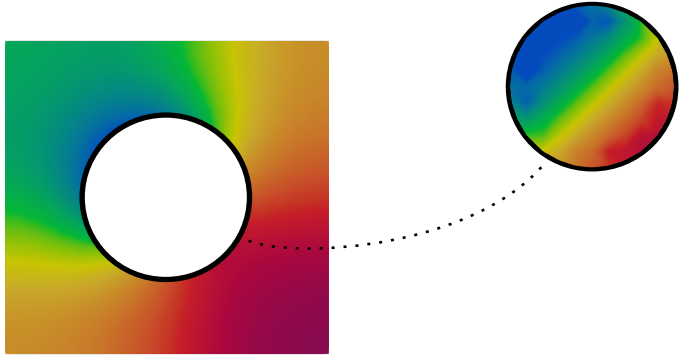
\begin{figure}[t]
    \tikzset{every picture/.style={line width=0.75pt}} 
    \resizebox{15cm}{7.75cm}{
    \begin{tikzpicture}[x=0.75pt,y=0.75pt,yscale=-1,xscale=1]
        
        \draw (68.71,78.25) node  {\includegraphics[width=187.06pt,height=126pt]{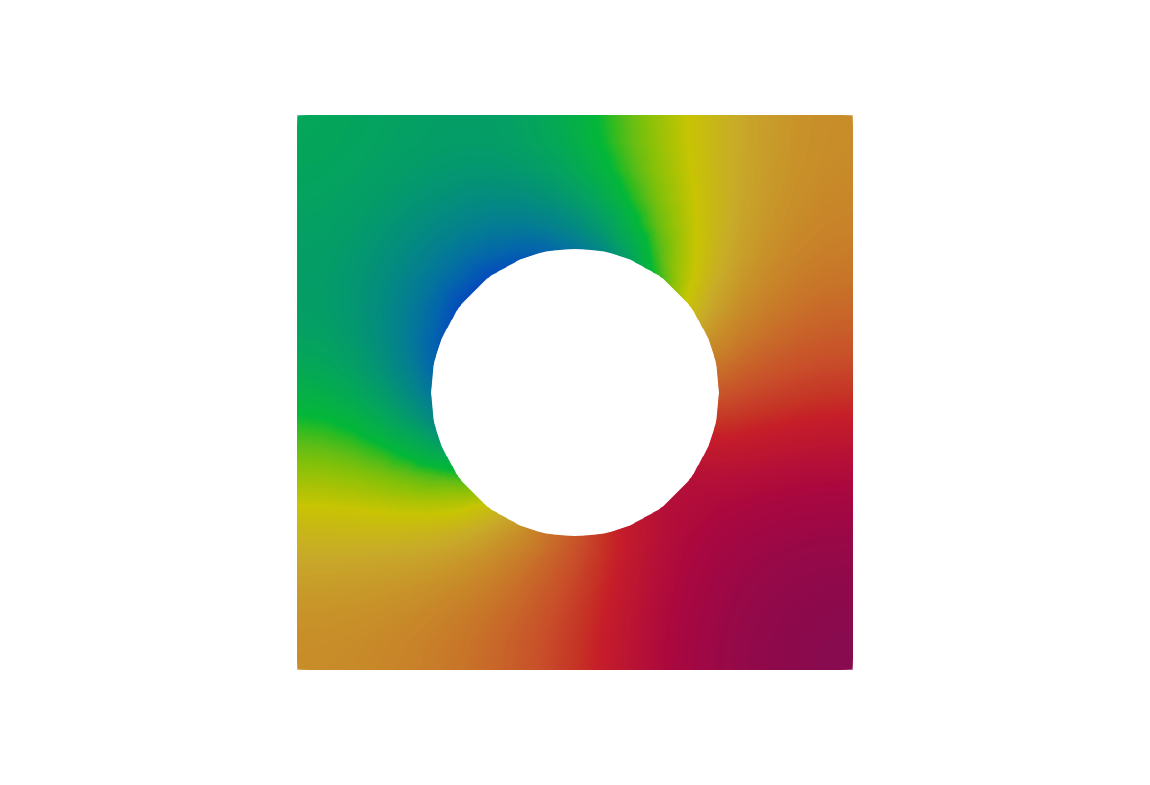}};
        \draw  [line width=1.5]  (37,78.21) .. controls (37,60.97) and (50.97,47) .. (68.21,47) .. controls (85.44,47) and (99.42,60.97) .. (99.42,78.21) .. controls (99.42,95.44) and (85.44,109.42) .. (68.21,109.42) .. controls (50.97,109.42) and (37,95.44) .. (37,78.21) -- cycle ;
        \draw  [line width=1.5]  (196,36.21) .. controls (196,18.97) and (209.97,5) .. (227.21,5) .. controls (244.44,5) and (258.42,18.97) .. (258.42,36.21) .. controls (258.42,53.44) and (244.44,67.42) .. (227.21,67.42) .. controls (209.97,67.42) and (196,53.44) .. (196,36.21) -- cycle ;
        \draw (227.17,36.25) node  {\includegraphics[width=189.75pt,height=130.5pt]{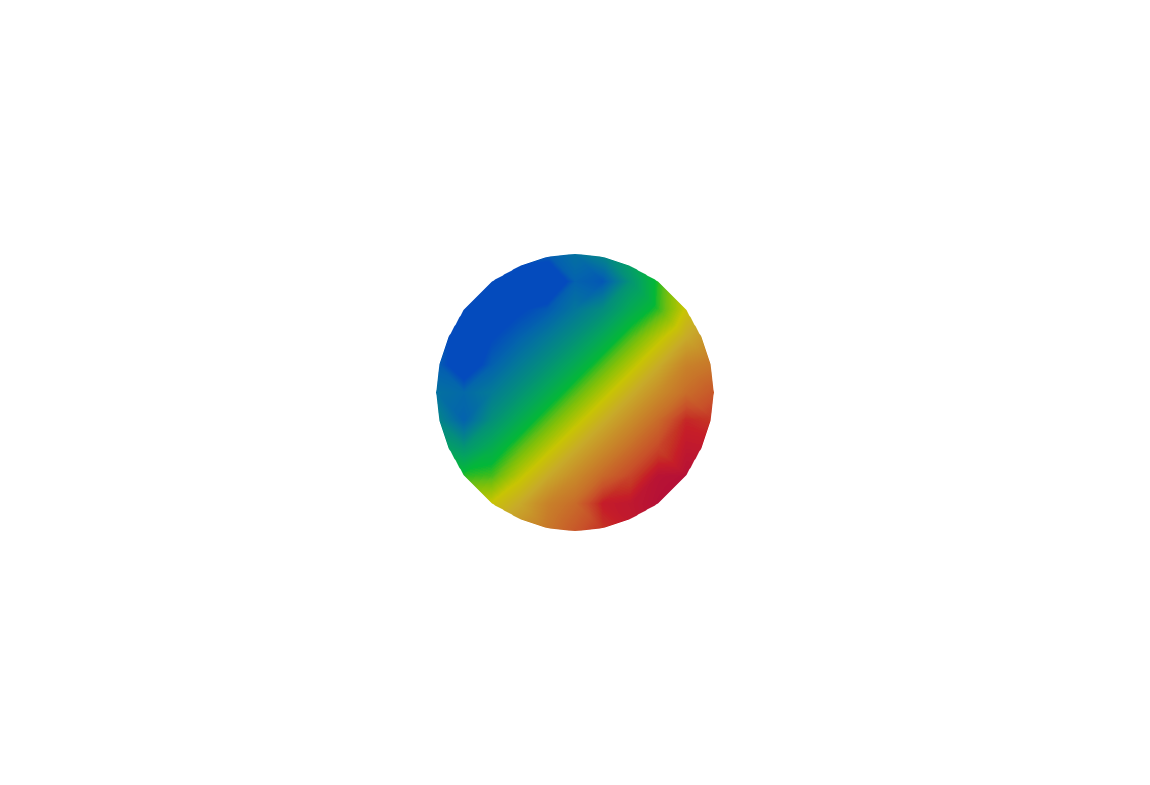}};
        \draw [line width=0.75]  [dash pattern={on 0.84pt off 2.51pt}]  (95.42,92.25) .. controls (102.42,102.25) and (181.42,101.25) .. (210.42,64.25) ;
    \end{tikzpicture}    
    }
    \caption{Harmonic extension into a circular hole cut from the background grid.}
    \label{fig: harmonic extension}
\end{figure}

As detailed in Section~\ref{sec: rb}, the \ac{ttrb} method compresses the \ac{fom} dimensions along each Cartesian direction. This requires constructing solution snapshots that can be assembled into a single $d$-dimensional tensor, which in turn demands solution values to be defined over the entire background domain. To achieve this, we apply a harmonic extension of the high-fidelity solution from the physical domain to the background mesh. This approach, commonly adopted in reduced-order modeling for unfitted discretizations \cite{CHASAPI2023115997,BALAJEWICZ2014489,KARATZAS2020833}, enables the construction of global basis vectors on a structured grid. We emphasize that nodal extension is used exclusively for implementing the \ac{ttrb} method, as the mesh deformation framework already suffices to support classical \ac{rb} techniques. In this work, we consider a discrete harmonic extension $\mathcal{H}_h : \mathcal{V}_h \to \widehat{\mathcal{V}}_h$ that takes the solution in the active mesh and returns an extended function on the whole background mesh $\widehat{u}_h(\bm{\mu}) \doteq \mathcal{H}_h \circ u_h(\bm{\mu})$. In particular, the extension of the solution is defined as 
\begin{equation}
    \label{eq: extended solution}
    \widehat{u}_h(\bm{\mu}) \doteq u_h(\bm{\mu}) \mathbbm{1}_{\Omega(\bm{\mu})} + u^{\texttt{OUT}}_h(\bm{\mu}) \mathbbm{1}_{\widehat{\Omega} \setminus \Omega(\bm{\mu})} \in H^1(\widehat{\Omega}),
\end{equation}
where $u^{\texttt{OUT}}_h(\bm{\mu}) \in \mathcal{V}^{\texttt{OUT}}_h$ is computed as
\begin{equation}
    \label{eq: extension}
    \int_{\widehat{\Omega} \setminus \Omega(\bm{\mu})} \bm{\nabla} \left(u_h(\bm{\mu})+u_h^{\texttt{OUT}}(\bm{\mu})\right) \bm{\nabla} v_h^{\texttt{OUT}} = 0 \quad \forall \ v_h^{\texttt{OUT}} \in \mathcal{V}^{\texttt{OUT}}_h, \quad u_h^{\texttt{OUT}}(\bm{\mu}) = u_h(\bm{\mu}) \quad \text{on} \quad \partial \Omega(\bm{\mu}).
\end{equation}
Here, $\mathcal{V}^{\texttt{OUT}}_h$ denotes the restriction of the background \ac{fe} space to external cells in the background mesh. In algebraic terms, solving \eqref{eq: fom}-\eqref{eq: extension} yields an $\widehat{\mathcal{N}}_h$-dimensional vector of nodal values $\widehat{\bm{u}}(\bm{\mu})$, which can be used to implement the \ac{ttrb} routines, given its $\bm{\mu}$-independent size. 

\subsection{Induced norms}
\label{subs: induced norms}

We conclude this section by defining the norms induced on the previously introduced \ac{fe} spaces. Since these are finite-dimensional spaces, we can equivalently describe them via their associated (symmetric, positive definite) norm matrices. \\
On the aggregated space, we define the matrix $\bm{X}(\bm{\mu}) \in \R^{\mathcal{N}_h \times \mathcal{N}_h}$ associated with the bilinear form:
\begin{equation}
    \label{eq: aggregated norm}
    \int_{\Omega(\bm{\mu})} \bm{\nabla} u_h \cdot \bm{\nabla} v_h d\Omega(\bm{\mu}) + \int_{\Gamma_D(\bm{\mu})} \tau_h u_h v_h  d\Gamma(\bm{\mu}).
\end{equation}
For convenience, throughout this work we often replace this parameter-dependent norm with a parameter-independent one $\bm{X} \in \R^{\mathcal{N}_h \times \mathcal{N}_h}$, associated with the bilinear form:
\begin{equation}
    \label{eq: mu-indep aggregated norm}
    \int_{\widetilde{\Omega}} \bm{\nabla} \widetilde{u}_h \cdot \bm{\nabla} \widetilde{v}_h d\widetilde{\Omega} + \int_{\widetilde{\Gamma}_D} \tau_h \widetilde{u}_h \widetilde{v}_h  d\widetilde{\Gamma}, \quad 
    \widetilde{u}_h \doteq u_h \circ \vec{\psi}^{-1}, \ \widetilde{v}_h \doteq v_h \circ \vec{\psi}^{-1}.
\end{equation}
This substitution is warranted by the fact that we always assume these induced norms to be equivalent, which holds under suitable hypotheses on the deformation map. In particular, let $\lambda_1(\widetilde{\bm{x}},\bm{\mu}),\hdots,\lambda_d(\widetilde{\bm{x}},\bm{\mu})$ denote the sorted eigenvalues of $\vec{\vec{J}}(\bm{\mu})$ evaluated at $\widetilde{\bm{x}} \in \widetilde{\Omega}$. Moreover, let $C_1(\bm{\mu})$ and $C_d(\bm{\mu})$ be positive constants such that
\begin{equation}
    \label{eq: constants norms}
    0 < C_1(\bm{\mu}) \leq \inf\limits_{\widetilde{\bm{x}} \in \widetilde{\Omega}}\lambda_1(\widetilde{\bm{x}},\bm{\mu}) \leq \sup\limits_{\widetilde{\bm{x}} \in \widetilde{\Omega}} \lambda_d(\widetilde{\bm{x}},\bm{\mu}) \leq C_d(\bm{\mu}) < \infty. 
\end{equation}
By virtue of \eqref{eq: constants norms}, for any $v_h \in \mathcal{V}_h$, the norm induced by $\bm{X}(\bm{\mu})$ can be bounded by that induced by $\bm{X}$ as follows:
\begin{equation}
    \label{eq: norm to ref norm}
    \begin{split}
        \|\bm{v}\|^2_{\bm{X}(\bm{\mu})} 
        &= \int_{\widetilde{\Omega}} \vec{\vec{J}}^{-T}(\bm{\mu})\widetilde{\bm{\nabla}} \widetilde{v}_h  \cdot \vec{\vec{J}}^{-T}(\bm{\mu})\widetilde{\bm{\nabla}} \widetilde{v}_h  \mathrm{det}(\vec{\vec{J}}(\bm{\mu})) d\widetilde{\Omega} + \int_{\widetilde{\Gamma}_D} \tau_h \widetilde{v}_h  \widetilde{v}_h  \|\vec{\vec{J}}^{-T}(\bm{\mu}) \widetilde{\vec{n}}\|_2 d\widetilde{\Gamma}
        \\ &\leq 
        \sup\limits_{\widetilde{\bm{x}} \in \widetilde{\Omega}}|\vec{\vec{J}}^{-1}(\bm{\mu})\vec{\vec{J}}^{-T}(\bm{\mu}) \mathrm{det}(\vec{\vec{J}}(\bm{\mu}))| \int_{\widetilde{\Omega}} \widetilde{\bm{\nabla}} \widetilde{v}_h  \cdot \widetilde{\bm{\nabla}} \widetilde{v}_h  +
        \sup\limits_{\widetilde{\bm{x}} \in \widetilde{\Omega}}|\vec{\vec{J}}^{-T}(\bm{\mu})| \int_{\widetilde{\Gamma}_D} \tau_h  \widetilde{v}_h  \widetilde{v}_h  d\widetilde{\Gamma}
        \\ &\lesssim
        \sup\limits_{\widetilde{\bm{x}} \in \widetilde{\Omega}}\max\left(\frac{\lambda_1(\widetilde{\bm{x}},\bm{\mu}) \cdots \lambda_d(\widetilde{\bm{x}},\bm{\mu})}{\lambda_1(\widetilde{\bm{x}},\bm{\mu})^2}, \frac{1}{\lambda_1(\widetilde{\bm{x}},\bm{\mu})}\right) \|\bm{v}\|^2_{\bm{X}} 
        \leq \frac{\max \left(C_d(\bm{\mu})^{d-1},1\right)}{C_1(\bm{\mu})} \|\bm{v}\|^2_{\bm{X}} , 
    \end{split} 
\end{equation}
where $\bm{v} \in \R^{\mathcal{N}_h}$ denotes the vector of nodal values associated with $v_h$ (and $\widetilde{v}_h$). Conversely:
\begin{equation}
    \label{eq: ref norm to norm}
    \|\bm{v}\|^2_{\bm{X}} \lesssim \frac{C_d(\bm{\mu})}{\min\left(1,C_1(\bm{\mu})^{d-1}\right)}\|\bm{v}\|^2_{\bm{X}(\bm{\mu})}.
\end{equation}
In other words, replacing $\bm{X}(\bm{\mu})$ with $\bm{X}$ introduces an error that is controlled by the ratio $C_d(\bm{\mu})/C_1(\bm{\mu})$. As we do not operate in regimes involving large deformations, we expect this error to be reasonable. \\
On $\widehat{\mathcal{V}}_h$, we can directly consider a norm that is independent of $\bm{\mu}$, since the background domain $\widehat{\Omega}$ is, by definition, parameter-independent. Given that harmonic extension is used to define functions outside the physical domain, we adopt the matrix $\widehat{\bm{X}} \in \R^{\widehat{\mathcal{N}}_h \times \widehat{\mathcal{N}}_h}$, representing the bilinear form:
\begin{equation}
    \label{eq: background norm}
    \int_{\widehat{\Omega}} \bm{\nabla} u_h \cdot \bm{\nabla} v_h d\widehat{\Omega}.
\end{equation}
\section{Reduced basis method}
\label{sec: rb}

In this section, we describe the \ac{rb} methodology, a data-driven, projection-based \ac{rom} involving two separate phases: 
\begin{itemize}
    \item A computationally intensive offline phase, during which the reduced subspace is constructed, and the (Petrov-) Galerkin projection of the \ac{fom} \eqref{eq: fom} is precomputed.
    \item A computationally efficient online phase, in which the reduced approximation is rapidly evaluated for any given parameter $\bm{\mu}$.
\end{itemize}
In Subsection \ref{subs: global subspace}, we detail the generation of the global \ac{rb} subspace via \ac{tpod}. In Subsection \ref{subs: reduced equations}, we present the hyper-reduced equations obtained by (i) approximating the residual and Jacobian via a hyper-reduction scheme, and (ii) projecting the results onto the reduced subspace. In Subsection \ref{subs: local subspaces}, we introduce the concept of local \ac{rb} \cite{https://doi.org/10.1002/nme.4371,doi:10.1137/130924408,PAGANI2018530}, which \cite{CHASAPI2023115997} shows improves the results of the global \ac{rb} when solving problems on parameterized domains.

\subsection{Global subspace generation}
\label{subs: global subspace}
Let $\mathcal{D}_{\texttt{OFF}} \subset \mathcal{D}$ be a set of $N_{\mu}$ sampled parameters. The sampling strategy should adequately cover the parameter space; see \cite{quarteroni2015reduced} for further discussion. To construct the global reduced subspace using \ac{tpod}, we begin by assembling a matrix of solution snapshots:
\begin{equation}
    \label{eq: snapshots matrix}
    \bm{U} = \left[\bm{u}(\bm{\mu}_1) | \hdots | \bm{u}(\bm{\mu}_{N_{\mu}})\right] \in \R^{\mathcal{N}_h \times N_{\mu}} ,
\end{equation}
where each column corresponds to the high-fidelity solution of Eq.~\eqref{eq: fom} for a parameter in $\mathcal{D}_{\texttt{OFF}}$. The snapshot matrix is then decomposed via \ac{svd}:
\begin{equation*}
    \bm{U} = \bm{Q}\bm{\Sigma}\bm{V}^T , \quad  \bm{Q} \in \R^{\mathcal{N}_h \times R}, \bm{\Sigma} \in \R^{R \times R}, \bm{V}^T \in \R^{R \times N_{\mu}}, 
\end{equation*}
where $R \leq \min\{\mathcal{N}_h, N_{\mu}\}$ is the rank of $\bm{U}$. Here, $\bm{Q}$ and $\bm{V}$ are orthogonal matrices of left and right singular vectors, and $\bm{\Sigma}$ is a diagonal matrix of sorted singular values. The reduced basis matrix $\bm{\Phi}$ is formed by selecting the first $n$ columns of $\bm{Q}$:
\begin{equation*}
    \bm{\Phi} \doteq \bm{Q}[:,1:n] \in \R^{\mathcal{N}_h \times n}.
\end{equation*}
The variable $n$ indicates the dimension of the subspace, and is usually adaptively chosen to meet a prescribed error tolerance $\varepsilon$, for e.g. using the energy selection criterion:
\begin{equation}
    \label{eq: energy criterion}
    n = \arg\min\limits_{m = 1,\hdots,R} \left( 1 - \sum\limits_{i=1}^m \bm{\Sigma}[i,i]^2 / \sum\limits_{i=1}^{R} \bm{\Sigma}[i,i]^2 \leq \varepsilon^2 \right).
\end{equation}
\begin{algorithm}
    \caption{\texttt{TPOD}: Construct the $\bm{X}$-orthogonal spatial basis $\bm{\Phi} \in \R^{\mathcal{N}_h \times n}$ from the matrix of snapshots $\bm{U} \in \R^{\mathcal{N}_h \times N_{\mu}}$, given a prescribed accuracy $\varepsilon$ and the norm matrix $\bm{X} \in \R^{\mathcal{N}_h \times \mathcal{N}_h}$.}
    \begin{algorithmic}[1]
        \Function{\texttt{TPOD}}{$\bm{U}, \bm{X}, \varepsilon $}
        \State Cholesky factorization: $\bm{H}^T\bm{H} = \texttt{Cholesky}\left(\bm{X}\right)$ 
        \State Spatial rescaling: $\widecheck{\bm{U}} = \bm{H}\bm{U}$ 
        \State Spatial reduction: $\widecheck{\bm{\Phi}}\widecheck{\bm{\Sigma}}\widecheck{\bm{V}}^T = \texttt{RSVD}(\widecheck{\bm{U}},\varepsilon)$ 
        \State Spatial inverse rescaling: $\bm{\Phi} = \bm{H}^{-1} \widecheck{\bm{\Phi}}$
        \State Return $\bm{\Phi}$
        \EndFunction
    \end{algorithmic}
    \label{alg: tpod}
\end{algorithm}
While it is possible to construct a reduced subspace using an \ac{rb} that is merely orthonormal with respect to the Euclidean inner product, such a choice does not guarantee the well-posedness of the resulting \ac{rom}. In contrast, selecting an $H^1$-orthogonal basis ensures that the reduced subspace remains a subset of $\mathcal{V}_h$, which is critical for preserving the stability and well-posedness of the reduced model. However, this condition alone is not sufficient; additional considerations are required, as discussed in Subsection~\ref{subs: supremizers}. For this reason, we adopt Algorithm~\ref{alg: tpod} to generate the \ac{rb}. The function \texttt{RSVD} refers to the randomized \ac{svd} described in \cite{halko2010findingstructurerandomnessprobabilistic}, which discards the final $R - n$ modes of $\bm{U}$ based on the energy criterion~\eqref{eq: energy criterion}. The matrix $\bm{X}$ denotes the norm matrix associated with the bilinear form~\eqref{eq: mu-indep aggregated norm}. The procedure satisfies the following accuracy estimate:
\begin{equation}
    \label{eq: offline accuracy}
    \sum\limits_{k=1}^{N_{\mu}}\|\bm{u}(\bm{\mu}_k) - \bm{\Phi}\bm{\Phi}^T\bm{X}\bm{u}(\bm{\mu}_k)\|^2_{\bm{X}} \leq \varepsilon^2 \sum\limits_{k=1}^{N_{\mu}} \|\bm{u}(\bm{\mu}_k)\|^2_{\bm{X}}.
\end{equation}

\subsection{Reduced order equations}
\label{subs: reduced equations}
Once $\bm{\Phi}$ is computed, we can project the \ac{fom} \eqref{eq: fom} onto the reduced subspace spanned by the \ac{rb}. While Petrov-Galerkin formulations can offer advantages over Galerkin projections in certain cases -- such as improved stability for reduced saddle-point problems \cite{negri2015reduced,doi:10.1137/22M1509114} -- these benefits typically come at the cost of increased offline computational effort. Therefore, in this work, we restrict our attention to the standard Galerkin projection. Adopting an algebraic perspective, the projected equations read as 
\begin{equation}
	\label{eq: rom}
	\text{Given } \bm{w}_n^{(0)} \in \R^{n}, \text{ compute } \bm{J}_n(\bm{w}_n^{(k)};\bm{\mu})\delta\bm{w}_n^{(k)} = -\bm{r}_n(\bm{w}_n^{(k)};\bm{\mu}), \text{ and update } \bm{w}_n^{(k+1)} = \bm{w}_n^{(k)} + \delta\bm{w}_n^{(k)}
\end{equation}
for all $k = 1,2,\hdots,$ where 
\begin{equation}
	\label{eq: rb jacobian/residual}
	\bm{J}_n(\cdot;\bm{\mu}) \doteq \bm{\Phi}^T\bm{J}(\cdot;\bm{\mu})\bm{\Phi} \in \R^{n \times n}, \quad 
	\bm{r}_n(\cdot;\bm{\mu}) \doteq \bm{\Phi}^T\bm{r}(\cdot;\bm{\mu}) \in \R^{n}.
\end{equation}

Since computing the quantities in~\eqref{eq: rb jacobian/residual} involves operations that scale with the full-order dimensions, employing a hyper-reduction strategy to approximate the residual and Jacobian is essential for efficiency. To this end, one can use the \ac{mdeim}~\cite{NEGRI2015431,MUELLER2024115767} to obtain the affine expansions:
\begin{equation}
    \label{eq: affine decomp}
    \bm{J}(\bm{w};\bm{\mu}) \approx \sum\limits_{i=1}^{n^J} \bm{\Phi}^J[:,:,i] \bm{\theta}^J(\bm{w};\bm{\mu})[i], \quad 
    \bm{r}(\bm{w};\bm{\mu}) \approx \sum\limits_{i=1}^{n^r} \bm{\Phi}^r[:,i] \bm{\theta}^r(\bm{w};\bm{\mu})[i].
\end{equation}
Here, 
\begin{equation}
    \label{eq: hypred bases}
    \bm{\Phi}^J \in \R^{\mathcal{N}_h \times \mathcal{N}_h \times n^J}, \quad 
    \bm{\Phi}^r \in \R^{\mathcal{N}_h \times n^r}
\end{equation}
represent two bases that span reduced-dimensional subspaces for the residual and Jacobian, respectively of dimension $n^J$ and $n^r$. On the other hand, 
\begin{equation}
    \label{eq: hypred coeffs}
    \bm{\theta}^J(\bm{w};\bm{\mu}) \in \R^{n^J}, \quad \bm{\theta}^r(\bm{w};\bm{\mu}) \in \R^{n^r}
\end{equation}
are the nonlinear coefficients associated with the bases $\bm{\Phi}^J$ and $\bm{\Phi}^r$, respectively. Substituting \eqref{eq: affine decomp} in \eqref{eq: rom}, we obtain the hyper-reduced equation: 
\begin{equation}
    \label{eq: hypred rom}
    \text{Given } \bm{w}_n^{(0)} \in \R^{n}, \text{ compute } \widebar{\bm{J}}_n(\bm{w}_n^{(k)};\bm{\mu})\delta\bm{w}_n^{(k)} = -\widebar{\bm{r}}_n(\bm{w}_n^{(k)};\bm{\mu}), \text{ and update } \bm{w}_n^{(k+1)} = \bm{w}_n^{(k)} + \delta\bm{w}_n^{(k)}
\end{equation}
for all $k = 1,2,\hdots,$ where 
\begin{equation}
    \label{eq: hypred quantities}
    \R^{n \times n} \ni \widebar{\bm{J}}_n(\bm{w}_n;\bm{\mu}) = \sum\limits_{i=1}^{n^J} \bm{\Phi}^T\bm{\Phi}^J[:,:,i] \bm{\Phi} \bm{\theta}^J(\bm{w}_n;\bm{\mu})[i], \quad 
    \R^{n} \ni \widebar{\bm{r}}_n(\bm{w}_n;\bm{\mu}) = \sum\limits_{i=1}^{n^r} \bm{\Phi}^T\bm{\Phi}^r[:,i] \bm{\theta}^r(\bm{w}_n;\bm{\mu})[i] .
\end{equation}
Since the bases in \eqref{eq: hypred bases} are $\bm{\mu}$-independent, most of the computations in \eqref{eq: hypred quantities} can be performed offline. During the online phase, we simply compute the reduced coefficients \eqref{eq: hypred coeffs} and multiply them by their respective projected bases. The final online step is the solution of \eqref{eq: hypred rom}. Because the cost of these operations depends only on $n$, $n^J$, $n^r$, and the number of Newton-Raphson iterations, the cost of the online phase is independent of the \ac{hf} dimensions. \\ 
For completeness, we recall in Alg.~\ref{alg: mdeim} the \ac{mdeim} procedure used to obtain an affine approximation of the residual. 
\begin{algorithm}
    \caption{\texttt{MDEIM}: Given the matrix of residual snapshots $ \bm{R} = [\bm{r}(\bm{u}(\bm{\mu}_1);\bm{\mu}_1)|\hdots|\bm{r}(\bm{u}(\bm{\mu}_{N_{\mu}});\bm{\mu}_{N_{\mu}})] \in \R^{\mathcal{N}_h \times N_{\mu}}$ , and a prescribed accuracy $ \varepsilon $, build the $\ell^2$-orthogonal basis $\bm{\Phi}^r \in \R^{\mathcal{N}_h \times n^r}$, and the sampling matrix $\bm{P}^r \in \{0,1\}^{\mathcal{N}_h \times n^r}$.} 
	\begin{algorithmic}[1]
    \Function{\texttt{MDEIM}}{$ \bm{R}, \varepsilon $}
	\State Compute $\bm{\Phi}^r\bm{\Sigma}^r\left(\bm{V}^r\right)^T = \texttt{RSVD}\left(\bm{R},\varepsilon\right)$ 
	\State Set $\bm{P}^r = \left[ \bm{e}(\mathcal{i}^1)  \right]$, where $\mathcal{i}^1 = \argmax \vert \bm{\Phi}^r\left[:,1\right] \vert$ 
	\For{$ q \in \{2,\ldots,{n^r}\}$} 
    \State Set $\bm{w} = \bm{\Phi}^r\left[:,q\right]$, $\bm{W} = \bm{\Phi}^r\left[:,1:q-1\right]$
	\State Compute remainder $\bm{\delta} = \bm{w} - \bm{W} \left( \left(\bm{P}^r\right)^T \bm{W} \right)^{-1} \left(\bm{P}^r\right)^T \bm{w}$
	\State Update $\bm{P}^r = \left[ \bm{P}^r, \bm{e}(\mathcal{i}^q)\right]$, where $\mathcal{i}^q = \argmax \vert \bm{\delta} \vert$
	\EndFor 
    \State \Return  $\bm{\Phi}^r$, $\bm{P}^r$ 
    \EndFunction
	\end{algorithmic}
	\label{alg: mdeim}
\end{algorithm}	
Here, $\bm{e}(j)$ indicates the $j$th vector of the $\mathcal{N}_h$-dimensional canonical basis. The procedure is run entirely offline, whereas during the online phase, given a new parameter $\bm{\mu}$, the \ac{mdeim} approximation reads as:
\begin{equation}
    \label{eq: mdeim approximation}
    \bm{r}(\bm{u}(\bm{\mu});\bm{\mu}) \approx 
    \bm{\Phi}^r \bm{\theta}^r(\bm{u}(\bm{\mu});\bm{\mu}), 
    \quad \text{where} \quad
    \bm{\theta}^r(\bm{u}(\bm{\mu});\bm{\mu}) = \left( \left(\bm{P}^r\right)^T\bm{\Phi}^r \right)^{-1} \left( \bm{P}^r \right)^T \bm{r}(\bm{u}(\bm{\mu});\bm{\mu}).
\end{equation}
The sampling matrix
\begin{equation*}
    \bm{P}^r \in \{0,1\}^{\mathcal{N}_h \times n^r}
\end{equation*}
encodes a list of $n^r$ interpolation indices, which can be exploited to quickly evaluate $\left( \bm{P}^r \right)^T \bm{r}(\bm{u}(\bm{\mu});\bm{\mu})$. Indeed, this operation only requires running the \ac{fe} integration and assembly routines on the cells containing at least an interpolation index \cite{NEGRI2015431}. In other words, the sampling matrix identifies a global ``reduced integration domain'' that can be exploited to efficiently perform the online approximation in \eqref{eq: mdeim approximation}. Lastly, note that for the \ac{mdeim} approximation of the Jacobian, Algorithm~\ref{alg: mdeim} can be applied directly -- without modification -- to the nonzero entries of the Jacobian snapshots, as detailed in~\cite{NEGRI2015431}.
\begin{remark}
    Our implementation of \ac{mdeim} is applied separately to each integral appearing in the weak formulation (for example, boundary integrals are treated independently of interior contributions), as well as to the linear and nonlinear components of the residual. These choices are motivated by empirical evidence showing that handling these components individually substantially reduces the hyper-reduction error. Conceptually, such splittings decrease the complexity of the residual manifold. As also discussed in Section~\ref{sec: saddle-point}, we apply \ac{mdeim} independently to each block when addressing multi-field problems, such as saddle-point systems. 
\end{remark}
\begin{remark}
    The weak imposition of the Dirichlet condition (see Eq.~\eqref{eq: weak form pde}) leads to a simpler treatment than the more common strong imposition, particularly in nonhomogeneous and potentially parameterized settings. In this framework, enforcing the boundary condition amounts to adding an extra contribution to the residual, which in turn requires only one extra application of \ac{mdeim} to account for the corresponding integral. 
\end{remark}
\begin{remark}
    Although this point has been made implicitly, we emphasize that the use of deformation maps enables a seamless application of standard empirical interpolation techniques, even for problems posed on parameterized domains. In contrast, unfitted approaches that do not employ deformation maps and rely solely on extension-based strategies \cite{Karatzas_2019,KARATZAS2020833} face two significant difficulties. First, they must extend snapshot values to the entire background domain -- a particularly challenging task for residuals and Jacobians. Second, they must perform the greedy selection on these extended quantities. This is potentially cumbersome, as the selected entries may correspond to background cells lying outside the active domain for certain parameters. In that case, it becomes unclear how the resulting ``reduced integration domain'' could be meaningfully exploited to achieve efficient online computations.
\end{remark}

\subsection{Local subspaces method}
\label{subs: local subspaces}

Local \ac{rb} methods have proven effective for problems involving discontinuities or moving fronts, as demonstrated in~\cite{https://doi.org/10.1002/nme.4371,doi:10.1137/130924408,PAGANI2018530}. More recently, \cite{CHASAPI2023115997} applied localized \ac{rb} techniques to unfitted \acp{fe} on parameterized geometries -- a strategy we also adopt in this work. The motivation behind these approaches is the general non-reducibility of problems posed on parameterized domains. This lack of reducibility typically necessitates large reduced bases and, correspondingly, a large number of affine components in both the residual and Jacobian. To address this, we implement a localized approach, summarized in Algs.~\ref{alg: offline phase}-\ref{alg: online phase}, where $N_c$ and $N_c^h$ denote the number of clusters used for the reduced subspace and for hyper-reduction, respectively.

\begin{algorithm}
    \caption{\texttt{OFFLINE}: Given a set of offline parameters $\bm{\mu}_1,\hdots,\bm{\mu}_{N_{\mu}} \in \mathcal{D}$, the number of clusters for the reduced subspace $N_c$, the number of clusters for the hyper-reduction $N_c^h$, and a prescribed accuracy $ \varepsilon $, compute all offline quantities.} 
    \begin{algorithmic}[1]
        \Function{\texttt{OFFLINE}}{$\bm{\mu}_1,\hdots,\bm{\mu}_{N_{\mu}}, N_c,N_c^h, \bm{X},\varepsilon $}
        \State Run clustering of parameters (reduced subspace): $\bm{\alpha}_1,\hdots,\bm{\alpha}_{N_c} = \texttt{Kmeans}([\bm{\mu}_1,\hdots,\bm{\mu}_{N_{\mu}}], N_c)$, 
        \State Run clustering of parameters (hyper-reduction): $\bm{\beta}_1,\hdots,\bm{\beta}_{N_c^h} = \texttt{Kmeans}([\bm{\mu}_1,\hdots,\bm{\mu}_{N_{\mu}}], N_c^h)$
        \State Compute solution, residual and Jacobian snapshots for every $\bm{\mu}_1,\hdots,\bm{\mu}_{N_{\mu}}$
        \For{$k = 1,\hdots,N_c$}
        \State Fetch $\bm{U}_k$, the solution snapshots whose parameter belongs to the $k$th subspace cluster 
        \State Compute $\bm{\Phi}_k = \texttt{TPOD}(\bm{U}_k,\bm{X},\varepsilon)$ 
        \For{$l = 1,\hdots,N_c^h$} 
        \State Fetch $\bm{J}_l$ and $\bm{R}_l$, the Jac. and res. snapshots whose parameter belongs to the $l$th hyper-reduction cluster
        \State Compute $\bm{\Phi}^J_l, \bm{P}^J_l = \texttt{MDEIM}(\bm{J}_l,\varepsilon)$, $\bm{\Phi}^r_l, \bm{P}^r_l = \texttt{MDEIM}(\bm{R}_l,\varepsilon)$
        \State Compute $\{\bm{\Phi}^J_{k,l}, \bm{\Phi}^r_{k,l}\}_{k,l}$: Galerkin projections with respect to $\bm{\Phi}_k$, as in \eqref{eq: hypred quantities}
        \EndFor
        \EndFor
        \State \Return $\{\bm{\Phi}_k, \bm{\Phi}^J_{k,l}, \bm{\Phi}^r_{k,l}, \bm{P}^J_l, \bm{P}^r_l\}_{k,l}$
        \EndFunction
    \end{algorithmic}
    \label{alg: offline phase}
\end{algorithm} 

\begin{algorithm}
    \caption{\texttt{ONLINE}: Given an online parameter $\bm{\mu} \in \mathcal{D}$ and the quantities computed offline, compute the reduced solution $\bm{u}_n(\bm{\mu}) \in \R^n$.} 
    \begin{algorithmic}[1]
        \Function{\texttt{ONLINE}}{$\bm{\mu},\{\bm{\Phi}_k, \bm{\Phi}^J_{k,l}, \bm{\Phi}^r_{k,l}, \bm{P}^J_l, \bm{P}^r_l\}_{k,l}$}
        \State Find closest subspace cluster: $k = \arg \min\limits_{i=1,\hdots,N_c} \|\bm{\mu} - \bm{\alpha}_i\|_2$
        \State Find closest hyper-reduction cluster: $l = \arg \min\limits_{i=1,\hdots,N^h_c} \|\bm{\mu} - \bm{\beta}_i\|_2$
        \State Fetch the offline quantities $\bm{\Phi}_k, \bm{\Phi}^J_{k,l} , \bm{\Phi}^r_{k,l}, \bm{P}^J_l, \bm{P}^r_l$ associated with $(k,l)$ 
        \State Compute $\bm{\theta}^J(\bm{\mu}), \bm{\theta}^r(\bm{\mu})$ according to \eqref{eq: mdeim approximation}
        \State Assemble and solve the hyper-reduced problem \eqref{eq: hypred rom}-\eqref{eq: hypred quantities}
        \State \Return $\bm{u}_n(\bm{\mu})$
        \EndFunction
    \end{algorithmic}
    \label{alg: online phase}
\end{algorithm} 

For additional details, we refer the reader to~\cite{CHASAPI2023115997}. Here, we briefly summarize the rationale behind the proposed strategy. The core idea is to partition the set of offline parameters into clusters and construct local reduced subspaces and hyper-reduction quantities within each cluster. This serves as a form of piecewise linearization of the parameter-to-solution map, which is typically highly nonlinear in the presence of parameterized geometries. In essence, the localization strategy aims to reduce the Kolmogorov $N$-width of the high-fidelity solution manifold, as well as that of the left- and right-hand sides of the governing equations. This approach is conceptually motivated by the following two assumptions:
\begin{enumerate}
    \item For any $\bm{\mu} \in \mathcal{D}$ belonging to the $k$th subspace cluster and $l$th hyper-reduction cluster, we assume:
    \begin{equation}
        \label{eq: ansatz 1}
        \|\bm{u}(\bm{\mu}) - \bm{u}(\bm{\alpha}_k)\|_{\bm{X}} \propto \mathrm{dist} \left( \bm{\mu}, \bm{\alpha}_k \right), \quad 
        \|\bm{J}(\bm{\mu}) - \bm{J}(\bm{\beta}_l)\|_F \propto \mathrm{dist} \left( \bm{\mu}, \bm{\beta}_l \right), \quad 
        \|\bm{r}(\bm{\mu}) - \bm{r}(\bm{\beta}_l)\|_2 \propto \mathrm{dist} \left( \bm{\mu}, \bm{\beta}_l \right).
    \end{equation}
    Here, $\bm{\alpha}_k$ and $\bm{\beta}_l$ are cluster centers for subspaces and hyper-reduction quantities, and $\mathrm{\mathrm{dist}}: \mathcal{D} \times \mathcal{D} \to \R$ is defined as
    \begin{equation*}
        \mathrm{dist}(\bm{\mu}_k,\bm{\mu}_l) \mapsto \max\limits_{i,j}\{\|\bm{x}_i - \bm{x}_j\|_2, \ \bm{x}_i, \bm{x}_j \in \Gamma(\bm{\mu}_i) \cap \Gamma(\bm{\mu}_l)\},
    \end{equation*} 
    and quantifies the geometric distance between the deformed boundaries associated with two parameters.
    \item For any $\bm{\mu} \in \mathcal{D}$ belonging to the $k$th subspace cluster, there exists a positive constant $C$ such that 
    \begin{equation}
        \label{eq: ansatz 2}
        \mathrm{dist} \left( \bm{\mu}, \bm{\alpha}_k \right) \leq C\|\bm{\mu} - \bm{\alpha}_k\|_2. 
    \end{equation}
\end{enumerate}
If both assumptions~\eqref{eq: ansatz 1} and~\eqref{eq: ansatz 2} hold, then the localization strategy is justified: the intra-cluster variation in solutions and operators is controlled by the parameter distance. While these assumptions are reasonable in many cases, they do not hold universally. Consequently, the localization procedure remains outside the certified framework of classical \ac{rb} methods. As noted in~\cite{CHASAPI2023115997}, the development of rigorous a posteriori error estimators for Algs.~\ref{alg: offline phase}-\ref{alg: online phase} remains an open and important challenge. Nevertheless, in Section~\ref{sec: results}, we demonstrate empirically that the proposed strategy performs well in practice, despite the lack of theoretical guarantees. Finally, we note that the framework could be further generalized by allowing for distinct clusterings of the residual and Jacobian, or pursue alternative clustering algorithms to \texttt{Kmeans} \cite{Likas2003GlobalKMeans}. However, we do not pursue these options in the present work.

\section{Tensor-train reduced basis method}
\label{sec: ttrb}

In this section, we extend the framework introduced in Section~\ref{sec: rb} to accommodate the \ac{ttrb} method, following the approach described in~\cite{MUELLER2026116790}. Subsection~\ref{subs: tt decomposition} provides a brief overview of the \ac{tt} decomposition for general tensors. In Subsection~\ref{subs: ttrb procedure}, we summarize the key components of the \ac{ttrb} methodology, highlighting the modifications it introduces to subspace generation, hyper-reduction, and projection procedures.

\subsection{Tensor-train decomposition}
\label{subs: tt decomposition}
\ac{tt} decompositions are a category of low-rank tensor approximations. They aim to approximate the entries of a given $d$-dimensional tensor as a (multi-dimensional) product of $d$ arrays, also called \ac{tt} cores. Specifically, given a tensor $\bm{U} \in \R^{\mathcal{N}_{h_1} \times \hdots \times \mathcal{N}_{h_d}}$, its \ac{tt} decomposition is given by the three-dimensional \ac{tt} cores 
\begin{equation}
    \label{eq: tt decomposition}
    \{\bm{\Phi}_i\}_{i=1}^d, \quad \bm{\Phi}_i \in \R^{r_{i-1} \times \mathcal{N}_{h_i} \times r_i},
\end{equation}
with $r_0 = r_d = 1$. Sequentially contracting the \ac{tt} cores yields an approximation of $\bm{U}$:
\begin{equation}
    \label{eq: tt product outer}
    \bm{U} \approx \bm{U}_r = \sum\limits_{i_1=1}^{r_1} \sum\limits_{i_2=1}^{r_2} \cdots \sum\limits_{i_{d-1}=1}^{r_{d-1}} \bm{\Phi}_1[1,:,i_1] \otimes \bm{\Phi}_2[i_1,:,i_2] \otimes \cdots \otimes \bm{\Phi}_d[i_{d-1},:,1].
\end{equation}
Here, $\otimes$ denotes the outer product between two arrays. We may simplify the notation above by introducing the contraction along a common axis of two multi-dimensional arrays 
\begin{equation*}
    \bm{R} \in \R^{N_a \times N_b \times N_c}, \qquad 
    \bm{S} \in \R^{N_c \times N_d \times N_e}
\end{equation*}
defined as 
\begin{equation}
    \label{eq: tt product}
    \R^{N_a \times N_b \times N_d \times N_e} \ni \bm{T} \doteq \bm{R} \cdot \bm{S}, \qquad \bm{T}\left[\alpha_a,\alpha_b,\alpha_d,\alpha_e\right] = \sum\limits_{\alpha_c} \bm{R} \left[\alpha_a,\alpha_b,\alpha_c\right] \bm{S} \left[\alpha_c,\alpha_d,\alpha_e\right].
\end{equation}
Since the last axis of a core matches the first axis of its successor, we may apply \eqref{eq: tt product} sequentially:
\begin{equation}
    \label{eq: tt approximation}
    \bm{U} \approx \bm{U}_r = \bm{\Phi}_1 \cdot \bm{\Phi}_2 \cdots \bm{\Phi}_d,
\end{equation}
The rank of the decomposition \eqref{eq: tt decomposition} is $\bm{r} = (r_1,\hdots,r_d)$. Each $r_i$ is referred to as a reduced rank, and is chosen according to the energy selection criterion \eqref{eq: energy criterion}. The \ac{tt} cores are usually computed by applying either the \ac{ttsvd} \cite{oseledets2011tensor} or the \ac{ttcross} \cite{oseledets2010tt} algorithms to $\bm{U}$. We recall the \ac{ttsvd} procedure, as first presented in~\cite{oseledets2011tensor}, in Algorithm~\ref{alg: ttsvd}.
\begin{algorithm}
    \caption{\texttt{TT-SVD}: Given the snapshots tensor in the ``split-axes'' format $\bm{U} \in \R^{\mathcal{N}_{h_1} \times \hdots \times \mathcal{N}_{h_d}}$ and a prescribed accuracy $ \varepsilon $, compute the \ac{tt} cores $\bm{\Phi}_1,\hdots,\bm{\Phi}_d$.} 
    \begin{algorithmic}[1]
        \Function{\texttt{TT-SVD}}{$ \bm{U}, \varepsilon $}
        \State Initialize unfolding matrix: $\bm{T} = \texttt{reshape}(\bm{U},\mathcal{N}_{h_1},:)$ 
        \For{i = $1, \hdots , d $} 
        \State $i$th reduction: $\bm{\Phi}_i\bm{\Sigma}_i\bm{V}_i^T = \texttt{RSVD}(\bm{T},\varepsilon)$ 
        \State Update unfolding matrix: $\bm{T} = \bm{\Sigma}_i\bm{V}_i^T$ \Comment{$\bm{T} \in \R^{r_i \times N_{\mu}\prod_{j=i+1}^d \mathcal{N}_{h_j} }$}
        \EndFor
        \State \Return  $\bm{\Phi}_1, \hdots, \bm{\Phi}_d$
        \EndFunction
    \end{algorithmic}
    \label{alg: ttsvd}
\end{algorithm} 
Here, the function \texttt{TSVD} denotes the truncated \ac{svd}, and $\varepsilon$ is the prescribed truncation tolerance. 

\subsection{Global TT-RB subspace generation}
\label{subs: ttrb procedure}

\ac{tt} decompositions can be combined with \ac{rb} techniques to construct efficient approximation strategies for high-dimensional solution manifolds. We refer to the resulting approach as the \ac{ttrb} method, following the terminology in~\cite{MUELLER2026116790}. Assuming that each \ac{hf} solution snapshot can be reshaped into a $d$-dimensional tensor, we concatenate them along a new dimension to form a $(d+1)$-dimensional snapshot tensor:
\begin{equation}
    \label{eq: tensor snapshots}
    \bm{U} = \left[\bm{U}(\bm{\mu}_1) | \hdots | \bm{U}(\bm{\mu}_{N_{\mu}})\right] \in \R^{\mathcal{N}_{h_1} \times \hdots \times \mathcal{N}_{h_d} \times N_{\mu}}, 
    \quad \bm{U}(\bm{\mu}_i) \in \R^{\mathcal{N}_{h_1} \times \hdots \times \mathcal{N}_{h_d}} \ \forall \ i, \quad \mathcal{N}_h = \prod_{i=1}^{d}\mathcal{N}_{h_i}.
\end{equation}
Similarly, a low-dimensional approximation of \eqref{eq: tensor snapshots} can be constructed as follows:
\begin{enumerate}
    \item Compute the first $d$ cores of the \ac{tt} decomposition of $\bm{U}$, namely $\{\bm{\Phi}_i\}_{i=1}^d$. These cores, conceptually, form the \ac{tt} decomposition of the basis that spans the reduced subspace approximating the manifold. Note that $r_0 = 1$ as before, but $r_d > 1$ in general, and represents the dimension of the reduced subspace.
    \item For any new parameter $\bm{\mu} \in \mathcal{D}$, find the reduced coefficient $\bm{u}_r(\bm{\mu}) \in \R^{r_d}$ such that 
    \begin{equation}
        \label{eq: tt approx}
        \bm{u}(\bm{\mu}) \approx \bm{\Phi} \bm{u}_r(\bm{\mu}), \quad \text{where} \quad \bm{\Phi} \doteq \bm{\Phi}_1 \cdots \bm{\Phi}_d \in \R^{\mathcal{N}_h \times r_d}.
    \end{equation} 
\end{enumerate}
As shown in \eqref{eq: tt approx}, the basis spanning the reduced subspace can be expressed as a matrix by sequentially contracting the \ac{tt} cores. \\
Representing snapshots as tensors, as in \eqref{eq: tensor snapshots}, is a key feature of the so-called ``split-axes'' format introduced in \cite{MUELLER2026116790}. Although this tensor representation is not strictly required to perform \ac{ttrb} -- since rank-reduction techniques like \ac{ttsvd} or \ac{ttcross} can also be applied to snapshots stored as matrices (see \eqref{eq: snapshots matrix}) -- the ``split-axes'' format is highly advantageous, particularly in spatial dimensions $d \geq 3$, as it significantly reduces the computational cost of constructing the reduced subspace. For further details, we refer to \cite{MUELLER2026116790}. To build the ``split-axes'' snapshots, we extend the \ac{hf} solution to the background grid using the harmonic extension operator described in Subsection~\ref{subs: extension}. Specifically, the $i$th snapshot in \eqref{eq: tensor snapshots} is given by
\begin{equation*}
    \bm{U}(\bm{\mu}_i) \doteq \texttt{reshape}(\widehat{\bm{u}}(\bm{\mu}_i), \mathcal{N}_{h_1}, \hdots, \mathcal{N}_{h_d}),
\end{equation*}
where $\mathcal{N}_{h_i}$ represents the number of \acp{dof} in the $i$th Cartesian direction. This size is well-defined because the background grid $\widehat{\Omega}_h$ is, by definition, a quadrangular partition. \\ 
In this work, we utilize a modified version of the \ac{ttsvd} algorithm (see Alg. \ref{alg: ttsvd}) that yields an $H^1(\widehat{\Omega})$-orthogonal basis. Since the procedure is rather involved, we refer the reader to \cite{MUELLER2026116790} for a detailed step-by-step explanation. Here, we recall the resulting accuracy estimate:
\begin{equation}
    \label{eq: offline accuracy ttrb}
    \sum\limits_{k=1}^{N_{\mu}}\|\widehat{\bm{u}}(\bm{\mu}_k) - \bm{\Phi}\bm{\Phi}^T\widehat{\bm{X}}\widehat{\bm{u}}(\bm{\mu}_k)\|^2_{\widehat{\bm{X}}} \leq d \varepsilon^2 \sum\limits_{k=1}^{N_{\mu}} \|\widehat{\bm{u}}(\bm{\mu}_k)\|^2_{\widehat{\bm{X}}},
\end{equation}
where $\widehat{\bm{X}}$ is the norm matrix associated with the bilinear form \eqref{eq: background norm}. This estimate shows that the \ac{ttsvd}-based approach degrades the accuracy of \ac{tpod} by a factor of $\sqrt{d}$. However, as noted in \cite{MUELLER2026116790}, the error incurred by \ac{ttrb} relative to \ac{tpod}-based methods can exceed this $\sqrt{d}$ factor -- a trend we also confirm in our numerical results.

\subsection{Notes on reduced order equations, hyper-reduction and localized strategy for TT-RB}
\label{subs: final notes ttrb}
The remaining steps of the \ac{ttrb} procedure follow analogously to those described in Subsections~\ref{subs: reduced equations}--\ref{subs: local subspaces}. Adopting an algebraic viewpoint once more, the hyper-reduced projected equations for \ac{ttrb} are given by: find $\bm{u}_r(\bm{\mu}) \in \R^{r_d}$ such that
\begin{equation}
    \label{eq: hypred ttrb}
    \text{Given } \bm{w}_r^{(0)} \in \R^{r}, \text{ compute } \widebar{\bm{J}}_r(\bm{w}_r^{(k)};\bm{\mu})\delta\bm{w}_r^{(k)} = -\widebar{\bm{r}}_r(\bm{w}_r^{(k)};\bm{\mu}), \text{ and update } \bm{w}_r^{(k+1)} = \bm{w}_r^{(k)} + \delta\bm{w}_r^{(k)}
\end{equation}
The quantities 
\begin{equation*}
    \widebar{\bm{J}}_r(\bm{\mu}) \in \R^{r_d \times r_d}, \quad 
    \widebar{\bm{r}}_r(\bm{\mu}) \in \R^{r_d}
\end{equation*}
are the hyper-reduced forms of the Jacobian and residual, respectively, and are defined as in \eqref{eq: hypred quantities}. These can be computed using the \ac{ttmdeim} procedure~\cite{MUELLER2026116790}, which extends the \ac{mdeim} algorithm to accommodate bases expressed in the \ac{tt} format. Since \ac{ttmdeim} and \ac{mdeim} exhibit comparable computational cost, we opt to use the latter for generating the numerical results presented in Section~\ref{sec: results}. Consequently, we use the term \ac{ttrb} to denote the \ac{rom} that combines \ac{ttsvd} for generating the reduced subspace with \ac{mdeim} for hyper-reduction. Finally, the localization strategy summarized in Algs.~\ref{alg: offline phase}-\ref{alg: online phase} can be readily adapted to the \ac{ttrb} framework by replacing the call to \ac{tpod} with a call to \ac{ttsvd} (line 7 of Alg.~\ref{alg: offline phase}).
\section{Approximation of saddle-point problems}
\label{sec: saddle-point}

In this section, we discuss the \ac{rb} approximation of a parameterized saddle-point problem discretized using the unfitted \ac{fe} method introduced in Section~\ref{sec: unfitted fe method}. We consider a benchmark problem governed by the steady, incompressible Stokes equations, formally presented in Subsection~\ref{subs: saddle-point full order model}. Next, in Subsection~\ref{subs: saddle-point reduced order model}, we state the reduced equations associated with this benchmark. Then, in Subsection~\ref{subs: supremizers}, we detail the supremizers enrichment -- a stabilization technique typically employed in \ac{rb} methods for saddle-point problems -- specifically adapted to address challenges arising from parameterized domains. For brevity of exposition, we focus our attention on a classical \ac{rb} method, without presenting the details of \ac{ttrb} implementations.

\subsection{Full order model}
\label{subs: saddle-point full order model}

We begin this subsection by stating the weak formulation of the incompressible, steady Stokes equations: find $(\vec{u}_h(\bm{\mu}),p_h(\bm{\mu})) \in \mathcal{V}_h \times \mathcal{Q}_h$ such that
\begin{align} 
    \label{eq: weak formulation stokes}
    \left\{
    \begin{aligned}
        a(\vec{u}_h(\bm{\mu}),\vec{v}_h;\bm{\mu}) - b(\vec{v}_h,p_h(\bm{\mu});\bm{\mu}) &= l(\vec{v}_h;\bm{\mu}) &&\forall \vec{v}_h \in \mathcal{V}_h, \\
        b(q_h,\vec{u}_h(\bm{\mu});\bm{\mu}) &= k(q_h;\bm{\mu}) &&\forall q_h \in \mathcal{Q}_h,
    \end{aligned} 
    \right.
\end{align}
where 
\begin{equation}
    \label{eq: stokes forms definition}
    \begin{split}
        a(\vec{u}_h,\vec{v}_h;\bm{\mu}) &= \int_{\Omega(\bm{\mu})} \bm{\nabla} \vec{u}_h : \bm{\nabla} \vec{v}_h + \int_{\Gamma_D(\bm{\mu})} \tau_h \vec{u}_h \cdot \vec{v}_h - (\bm{\nabla} \vec{u}_h \vec{n}(\bm{\mu})) \cdot \vec{v}_h - (\bm{\nabla} \vec{v}_h \vec{n}(\bm{\mu})) \cdot \vec{u}_h, \\ 
        b(\vec{u}_h,q_h;\bm{\mu}) &= \int_{\Omega(\bm{\mu})} q_h \bm{\nabla} \cdot \vec{u}_h - \int_{\Gamma_D(\bm{\mu})} q_h \vec{u}_h \cdot \vec{n}(\bm{\mu}), \\
        l(\vec{v}_h;\bm{\mu}) &= \int_{\Omega(\bm{\mu})} \vec{f}(\bm{\mu}) \cdot \vec{v}_h + \int_{\Gamma_D(\bm{\mu})} \tau_h \vec{u}_D(\bm{\mu}) \cdot \vec{v}_h - (\bm{\nabla} \vec{u}_D(\bm{\mu}) \vec{n}(\bm{\mu})) \cdot \vec{v}_h, \\
        k(q_h;\bm{\mu}) &= -\int_{\Gamma_D(\bm{\mu})} q_h \vec{u}_D(\bm{\mu}) \cdot \vec{n}(\bm{\mu}).
    \end{split}
\end{equation}
We consider an inf-sup stable velocity and pressure combination of \ac{fe} spaces $\mathcal{V}_h^{\texttt{IN}} \times \mathcal{Q}_h^{\texttt{IN}}$ on the interior mesh. In the numerical experiments, we use the Taylor-Hood pair of continuous piecewise quadratic velocities and piecewise linear pressures. Then, we introduce the discrete extension operators $\mathcal{E}_h^u$ and $\mathcal{E}_h^p$ for the velocity and pressure spaces, resp., and the corresponding aggregated velocity and pressure \ac{fe} spaces  
$\mathcal{V}_h \doteq \mathcal{E}_h^u \circ \mathcal{V}_h^{\texttt{IN}}$ and $
\mathcal{Q}_h \doteq \mathcal{E}_h^p \circ \mathcal{Q}_h^{\texttt{IN}}
$. We assume that the choice of extensions provides an inf-sup stability on $(\mathcal{V}_h,\mathcal{Q}_h)$ (see \cite{Badia2018} for the definition of such operators). We can algebraically express \eqref{eq: weak formulation stokes} as a type-1 saddle-point system of equations
\begin{equation}
    \label{eq: algebraic formulation stokes}
    \begin{bmatrix}
        \bm{A}(\bm{\mu}) & -\bm{B}(\bm{\mu})^T \\ 
        \bm{B}(\bm{\mu}) & \bm{0}
    \end{bmatrix}
    \begin{bmatrix}
        \bm{u}(\bm{\mu}) \\
        \bm{p}(\bm{\mu})
    \end{bmatrix}
    =
    \begin{bmatrix}
        \bm{l}(\bm{\mu}) \\
        \bm{k}(\bm{\mu})
    \end{bmatrix},
\end{equation}
where 
\begin{alignat*}{4}
    \bm{A}(\bm{\mu}) &\in \R^{\mathcal{N}_h^u \times \mathcal{N}_h^u}, \quad 
    \bm{B}(\bm{\mu}) &&\in \R^{\mathcal{N}_h^p \times \mathcal{N}_h^u}, \quad 
    \bm{u}(\bm{\mu}),\bm{l}(\bm{\mu}) &&\in \R^{\mathcal{N}_h^u}, \quad 
    \bm{p}(\bm{\mu}),\bm{k}(\bm{\mu}) &&\in \R^{\mathcal{N}_h^p}.
\end{alignat*}
Here, $\mathcal{N}_h^u(\bm{\mu})$ and $\mathcal{N}_h^p(\bm{\mu})$ represent the number of free \acp{dof} for the velocity and pressure. As before, we introduce the matrices
\begin{equation*}
    \bm{X} \in \R^{\mathcal{N}_h^u \times \mathcal{N}_h^u}, \quad 
    \bm{Y} \in \R^{\mathcal{N}_h^p \times \mathcal{N}_h^p}
\end{equation*}
associated with the bilinear forms 
\begin{equation}
    \label{eq: mu-indep aggregated norm saddlepoint}
    \int_{\widetilde{\Omega}} \bm{\nabla} \widetilde{\vec{u}}_h : \bm{\nabla} \widetilde{\vec{v}}_h d\widetilde{\Omega} + \int_{\widetilde{\Gamma}_D} \tau_h \widetilde{\vec{u}}_h \cdot \widetilde{\vec{v}}_h  d\widetilde{\Gamma}, \quad 
    \int_{\widetilde{\Omega}} \widetilde{p}_h \widetilde{q}_h d\widetilde{\Omega}
\end{equation}
representing the discrete inner products on $\mathcal{V}_h$ and $\mathcal{Q}_h$ in the reference configuration. We also introduce the $\bm{\mu}$-dependent norm matrices  
\begin{equation*}
    \bm{X}(\bm{\mu}) \in \R^{\mathcal{N}_h^u \times \mathcal{N}_h^u}, \quad 
    \bm{Y}(\bm{\mu}) \in \R^{\mathcal{N}_h^p \times \mathcal{N}_h^p}
\end{equation*}
which arise from 
\begin{equation}
    \label{eq: aggregated norm saddlepoint}
    \int_{\Omega(\bm{\mu})} \bm{\nabla} \vec{u}_h : \bm{\nabla} \vec{v}_h d\Omega(\bm{\mu}) + \int_{\Gamma_D(\bm{\mu})} \tau_h \vec{u}_h \cdot \vec{v}_h  d\Gamma(\bm{\mu}), \quad 
    \int_{\Omega(\bm{\mu})} p_h q_h d\Omega(\bm{\mu}).
\end{equation}
For the well-posedness of the problem, we need to show that the aggregated velocity and pressure spaces satisfy a discrete inf-sup condition on the deformed configuration. Before doing so, we state an analogous result on the reference configuration. 
\begin{theorem}
    \label{thm: coupling reference}
    Let $(\mathcal{V}_h,\mathcal{Q}_h)$ be the aggregated \ac{fe} spaces defined above. 
    Moreover, we define
    \begin{align}
        \label{eq: reference coupling quantities} 
        \begin{aligned}
            \widetilde{a}(\widetilde{\vec{u}}_h,\widetilde{\vec{v}}_h) &= \int_{\widetilde{\Omega}} \widetilde{\bm{\nabla}} \widetilde{\vec{u}}_h : \widetilde{\bm{\nabla}} \widetilde{\vec{v}}_h + \int_{\widetilde{\Gamma}_D} \tau_h \widetilde{\vec{u}}_h \cdot \widetilde{\vec{v}}_h - (\widetilde{\bm{\nabla}} \widetilde{\vec{u}}_h \widetilde{\vec{n}} ) \cdot \widetilde{\vec{v}}_h - (\widetilde{\bm{\nabla}} \widetilde{\vec{v}}_h \widetilde{\vec{n}} ) \cdot \widetilde{\vec{u}}_h, \\ 
            \widetilde{b}(\vec{\widetilde{v}}_h,\widetilde{q}_h) &= \int_{\widetilde{\Omega}} \widetilde{q}_h \widetilde{\bm{\nabla}} \cdot \widetilde{\vec{v}}_h - \int_{\widetilde{\Gamma}_D} \widetilde{q}_h \widetilde{\vec{v}}_h \cdot \widetilde{\vec{n}},
        \end{aligned}
    \end{align}
    where $\vec{\widetilde{u}}_h \doteq \vec{u}_h \circ \vec{\psi}^{-1}$, $\vec{\widetilde{v}}_h \doteq \vec{v}_h \circ \vec{\psi}^{-1}$ and $\widetilde{q}_h \doteq q_h \circ \vec{\psi}^{-1}$. Assuming $\eta$ is large enough \cite{Neiva2021}, there exists a positive constant $\beta$ uniform with respect to $h$ such that the following inf-sup condition holds:
    \begin{equation}
        \label{eq: reference inf-sup condition}
        \inf_{(\vec{\widetilde{u}}_h, \widetilde{p}_h) \in \mathcal{V}_h \times \mathcal{Q}_h} \sup_{(\vec{\widetilde{v}}_h, \widetilde{q}_h) \in \mathcal{V}_h \times \mathcal{Q}_h} 
        \frac{\widetilde{a}(\widetilde{\vec{u}}_h,\widetilde{\vec{v}}_h) + \widetilde{b}(\vec{\widetilde{v}}_h,\widetilde{p}_h) - \widetilde{b}(\vec{\widetilde{u}}_h,\widetilde{q}_h) }{\left( \|\vec{\widetilde{u}}_h\|_{\mathcal{V}_h} + \|\widetilde{p}_h\|_{\mathcal{Q}_h} \right) \left( \|\vec{\widetilde{v}}_h\|_{\mathcal{V}_h} + \|\widetilde{q}_h\|_{\mathcal{Q}_h} \right) } \geq \beta.
    \end{equation} 
    \begin{proof}
        We first prove that $\widetilde{a}$ is coercive by using Young's generalized inequality, the discrete trace inequality, the inverse inequality for aggregated spaces \cite{BADIA2018533}, and the assumption that $\eta$ is large enough \cite{Neiva2021}: 
        \begin{align*}
            \widetilde{a}(\widetilde{\vec{u}}_h,\widetilde{\vec{u}}_h) = \int_{\widetilde{\Omega}} \widetilde{\bm{\nabla}} \widetilde{\vec{u}}_h : \widetilde{\bm{\nabla}} \widetilde{\vec{u}}_h + \int_{\widetilde{\Gamma}_D} \tau_h \widetilde{\vec{u}}_h \cdot \widetilde{\vec{u}}_h - 2 (\widetilde{\bm{\nabla}} \widetilde{\vec{u}}_h \widetilde{\vec{n}} ) \cdot \widetilde{\vec{u}}_h \geq \left(1 - \frac{C}{\eta} \right)\| \widetilde{\vec{u}}_h\|_{(H^1(\widetilde{\Omega}))^d}^2 + \frac{\eta h^{-1}}{2} \|\widetilde{\vec{u}}_h\|_{(L^2(\widetilde{\Gamma}_D))^d}^2,
        \end{align*}
        where $C$ is a constant independent of $h$. As a result, we have that
        \begin{equation}
            \label{eq: proof 1 eq 1}
            \widetilde{a}(\widetilde{\vec{u}}_h,\widetilde{\vec{u}}_h) \gtrsim \| \widetilde{\vec{u}}_h\|_{(H^1(\widetilde{\Omega}))^d}^2 + \eta h^{-1} \|\widetilde{\vec{u}}_h\|_{(L^2(\widetilde{\Gamma}_D))^d}^2 \doteq \|\widetilde{\vec{u}}_h\|_{\mathcal{V}_h}^2.
        \end{equation}
        Now, let us split $\widetilde{b}$ into its bulk and boundary contributions:
        \begin{equation*}
            \widetilde{b}(\vec{\widetilde{v}}_h,\widetilde{q}_h) = \widetilde{b}_{\widetilde{\Omega}}(\vec{\widetilde{v}}_h,\widetilde{q}_h) + \widetilde{b}_{\widetilde{\Gamma}}(\vec{\widetilde{v}}_h,\widetilde{q}_h). 
        \end{equation*}
        The discrete inf-sup condition for the aggregated spaces reads as:
        \begin{equation}
            \label{eq: proof 1 eq 2}
            \forall \ \widetilde{q}_h \in \mathcal{Q}_h \ \exists \ \vec{\widetilde{s}}_h \in \mathcal{V}_h : \| \vec{\widetilde{s}}_h \|_{(H^1(\widetilde{\Omega}))^d} = \| \widetilde{q}_h \|_{L^2(\widetilde{\Omega})} \text{  and  } \widetilde{b}_{\widetilde{\Omega}}(\vec{\widetilde{s}}_h,\widetilde{q}_h) \gtrsim \| \widetilde{q}_h \|^2_{L^2(\widetilde{\Omega})}.
        \end{equation}
        Now, using the discrete trace and inverse inequalities, we bound $\widetilde{b}_{\widetilde{\Gamma}}$ as follows:
        \begin{align}
            \label{eq: proof 1 eq 3}
            \widetilde{b}_{\widetilde{\Gamma}}(\vec{\widetilde{v}}_h,\widetilde{q}_h) \lesssim \| \widetilde{q}_h \|_{L^2(\widetilde{\Gamma}_D)} \|\vec{\widetilde{v}}_h\|_{(L^2(\widetilde{\Gamma}_D))^d} \lesssim \|\widetilde{q}_h\|_{L^2(\widetilde{\Omega})}  h^{-1/2} \|\vec{\widetilde{v}}_h\|_{(L^2(\widetilde{\Gamma}_D))^d}, \quad \forall \ (\vec{\widetilde{v}}_h, \widetilde{q}_h) \in \mathcal{V}_h \times \mathcal{Q}_h.
        \end{align}
        Combining \eqref{eq: proof 1 eq 1}, \eqref{eq: proof 1 eq 2} and \eqref{eq: proof 1 eq 3} implies the existence of an $h$-independent $\beta$ such that \eqref{eq: reference inf-sup condition} holds. This concludes the proof.
    \end{proof}
\end{theorem}

We can proceed in the same way to prove well-posedness of the \ac{fom} in the deformed configuration, assuming that an analogue of \eqref{eq: proof 1 eq 2} holds on the deformed mesh (see, eg., \cite{boffi2013mixed} for a detailed discussion). However, in the following theorem, we prove stability on the deformed configuration by relying only on the stability on the reference configuration, under \textit{smallness} assumptions on the displacement map. The importance of this result will become evident in Subsection~\ref{subs: supremizers}, where we will use it to establish the inf-sup stability of the \ac{rom}.

\begin{theorem}
    \label{thm: coupling}
    Let \eqref{eq: constants norms} hold. Furthermore, assume that there exists a positive constant $c_K(\bm{\mu})$, bounded away from zero independently of $h$, such that for any active deformed cell $K$, and its corresponding active reference cell $\widetilde{K}$, the following relation holds: 
    \begin{equation}
        \label{eq: coupling equivalence condition}
        \int_K q_h \bm{\nabla} \cdot \vec{v}_h d\Omega(\bm{\mu}) 
        = c_K(\bm{\mu}) \int_{\widetilde{K}} \widetilde{q}_h \widetilde{\bm{\nabla}} \cdot \widetilde{\vec{v}}_h d\widetilde{\Omega}.
    \end{equation}
    Then, if the inf-sup stability condition \eqref{eq: reference inf-sup condition} holds, then there exists a positive constant $\beta(\bm{\mu})$ uniform with respect to $h$ such that the following inf-sup condition is satisfied:
    \begin{equation}
        \label{eq: inf-sup condition}
        \inf_{(\vec{u}_h, p_h) \in \mathcal{V}_h \times \mathcal{Q}_h} \sup_{(\vec{v}_h, q_h) \in \mathcal{V}_h \times \mathcal{Q}_h} 
        \frac{a(\vec{u}_h,\vec{v}_h;\bm{\mu}) + b(\vec{v}_h,p_h;\bm{\mu}) - b(\vec{u}_h,q_h;\bm{\mu}) }{\left( \|\vec{u}_h\|_{\mathcal{V}_h(\bm{\mu})} + \|p_h\|_{\mathcal{Q}_h(\bm{\mu})} \right) \left( \|\vec{v}_h\|_{\mathcal{V}_h(\bm{\mu})} + \|q_h\|_{\mathcal{Q}_h(\bm{\mu})} \right) } \geq \beta(\bm{\mu}).
    \end{equation}
    Here, with a slight abuse of notation, $\|\cdot\|_{\mathcal{V}_h(\bm{\mu})}$, $\|\cdot\|_{\mathcal{Q}_h(\bm{\mu})}$ indicate the velocity and pressure norms in the deformed configuration.
    \begin{proof}
        By virtue of \eqref{eq: constants norms}, we have that 
        \begin{equation}
            \label{eq: proof 2 eq 1}
            \|\widetilde{\vec{v}}_h\|_{\mathcal{V}_h} \simeq \|\widetilde{\vec{v}}_h\|_{\mathcal{V}_h(\bm{\mu})}, \quad 
            \|\widetilde{q}_h\|_{\mathcal{Q}_h} \simeq \|\widetilde{q}_h\|_{\mathcal{Q}_h(\bm{\mu})}, \quad \forall \ (\vec{\widetilde{v}}_h, \widetilde{q}_h) \in \mathcal{V}_h \times \mathcal{Q}_h.
        \end{equation} 
        We prove this result using \eqref{eq: norm to ref norm}-\eqref{eq: ref norm to norm} in the case of the velocity norm; an even simpler argument applies to the pressure. The equivalence constants of the expressions above depend only on the eigenvalues of $\vec{\vec{J}}(\bm{\mu})$. Next, repeating the steps in \eqref{eq: proof 1 eq 1} and \eqref{eq: proof 1 eq 3}, we obtain the coercivity of $a$ and the boundedness of $b_{\Gamma}$, i.e., the boundary contribution of the deformed coupling operator. Then, by exploiting \eqref{eq: coupling equivalence condition} and \eqref{eq: proof 1 eq 2}, we deduce that 
        \begin{equation}
            \label{eq: proof 2 eq 2}
            \forall \ \widetilde{q}_h \in \mathcal{Q}_h \ \exists \ \widetilde{\vec{s}}_h \in \mathcal{V}_h : \| \widetilde{\vec{s}}_h \|_{(H^1(\Omega(\bm{\mu})))^d} = \| \widetilde{q}_h \|_{L^2(\Omega(\bm{\mu}))} \text{  and  } b_{\Omega}(\widetilde{\vec{s}}_h,\widetilde{q}_h;\bm{\mu}) \gtrsim \| \widetilde{q}_h \|^2_{L^2(\Omega(\bm{\mu}))},
        \end{equation}
        where $b_{\Omega}$ denotes the bulk contribution of the deformed coupling operator.
        The coercivity of $a$, the boundedness of $b_{\Gamma}$, and \eqref{eq: proof 2 eq 2} imply the existence of $\beta(\bm{\mu}) > 0$ uniform with respect to $h$ such that
        \begin{equation}
            \label{eq: proof 2 eq 3}
            \inf_{(\vec{\widetilde{u}}_h, \widetilde{p}_h) \in \mathcal{V}_h \times \mathcal{Q}_h} \sup_{(\vec{\widetilde{v}}_h, \widetilde{q}_h) \in \mathcal{V}_h \times \mathcal{Q}_h} 
            \frac{a(\widetilde{\vec{u}}_h,\widetilde{\vec{v}}_h;\bm{\mu}) + b(\vec{\widetilde{v}}_h,\widetilde{p}_h;\bm{\mu}) - b(\vec{\widetilde{u}}_h,\widetilde{q}_h;\bm{\mu}) }{\left( \|\vec{\widetilde{u}}_h\|_{\mathcal{V}_h} + \|\widetilde{p}_h\|_{\mathcal{Q}_h} \right) \left( \|\vec{\widetilde{v}}_h\|_{\mathcal{V}_h} + \|\widetilde{q}_h\|_{\mathcal{Q}_h} \right) } \geq \beta(\bm{\mu}).
        \end{equation} 
        Eqs.~\eqref{eq: proof 2 eq 1} and \eqref{eq: proof 2 eq 3} complete the proof.
    \end{proof}
\end{theorem}
We note that \eqref{eq: coupling equivalence condition} can be interpreted as a smallness assumption on the displacement map, which is easy to check by observing that assumption \eqref{eq: coupling equivalence condition} holds for the unit constant when the displacement map is zero. Under the stated assumptions, Thm.~\ref{thm: coupling} therefore offers a practical criterion for assessing the stability of the problem in the deformed configuration based on that of the reference configuration.

\subsection{Reduced order model}
\label{subs: saddle-point reduced order model}
The first step is to generate two separate \ac{rb} reduced subspaces for velocity and pressure. To this end, we first define the velocity and pressure snapshots matrices
\begin{equation*}
    \bm{U} = \left[\bm{u}(\bm{\mu}_1) | \hdots | \bm{u}(\bm{\mu}_{N_{\mu}})\right] \in \R^{\mathcal{N}^u_h \times N_{\mu}}, \quad 
    \bm{P} = \left[\bm{p}(\bm{\mu}_1) | \hdots | \bm{p}(\bm{\mu}_{N_{\mu}})\right] \in \R^{\mathcal{N}^p_h \times N_{\mu}}.
\end{equation*}
We then run Alg.~\ref{alg: tpod} separately for velocity and pressure, thus yielding two bases 
\begin{equation}
    \bm{\Phi}^u \in \R^{\mathcal{N}_h^u \times n^u}, \quad 
    \bm{\Phi}^p \in \R^{\mathcal{N}_h^p \times n^p}
\end{equation}
that are $\bm{X}$ and $\bm{Y}$-orthogonal, respectively. Consequently, we can express the Galerkin projection of \eqref{eq: algebraic formulation stokes} onto the space 
\begin{equation*}
    \mathrm{span}\{\bm{\Phi}^u\} \times \mathrm{span}\{\bm{\Phi}^p\}, 
\end{equation*}
as follows: find $\left(\bm{u}_n(\bm{\mu}),\bm{p}_n(\bm{\mu})\right) \in \R^{n^u} \times \R^{n^p}$ such that: 
\begin{equation}
    \label{eq: stokes reduced equations}
    \begin{split}
        \begin{bmatrix}
            \bm{\Phi}^u \\ 
            \bm{\Phi}^p
        \end{bmatrix}^T
        \begin{bmatrix}
            \bm{A}(\bm{\mu})  & -\bm{B}(\bm{\mu})^T \\ 
            \bm{B}(\bm{\mu}) & \bm{0}
        \end{bmatrix}
        \begin{bmatrix}
            \bm{\Phi}^u\bm{u}_n(\bm{\mu}) \\ 
            \bm{\Phi}^p\bm{p}_n(\bm{\mu})
        \end{bmatrix}
        &= 
        \begin{bmatrix}
            \bm{\Phi}^u \\ 
            \bm{\Phi}^p
        \end{bmatrix}^T
        \begin{bmatrix}
            \bm{l}(\bm{\mu}) \\ 
            \bm{k}(\bm{\mu})
        \end{bmatrix}
        \\ \Longleftrightarrow
        \begin{bmatrix}
            \bm{A}_n(\bm{\mu})  & -\bm{B}_n(\bm{\mu})^T \\ 
            \bm{B}_n(\bm{\mu}) & \bm{0}
        \end{bmatrix}
        \begin{bmatrix}
            \bm{u}_n(\bm{\mu}) \\ 
            \bm{p}_n(\bm{\mu})
        \end{bmatrix}
        &= 
        \begin{bmatrix}
            \bm{l}_n(\bm{\mu}) \\ 
            \bm{k}_n(\bm{\mu})
        \end{bmatrix}.
    \end{split}
\end{equation}
As noted in Subsection \ref{subs: reduced equations}, the reduced system in \eqref{eq: stokes reduced equations} must be hyper-reduced to eliminate the dependence on \ac{hf} quantities. For this purpose, it suffices to run \ac{mdeim} for each \ac{lhs} and \ac{rhs} block. Once this is done, we may write the hyper-reduced equations
\begin{equation}
    \label{eq: stokes rb algebraic}
    \begin{bmatrix}
        \widebar{\bm{A}}_n(\bm{\mu})  & -\widebar{\bm{B}}_n(\bm{\mu})^T \\ 
        \widebar{\bm{B}}_n(\bm{\mu}) & \bm{0}
    \end{bmatrix}
    \begin{bmatrix}
        \bm{u}_n(\bm{\mu}) \\ 
        \bm{p}_n(\bm{\mu})
    \end{bmatrix}
    = 
    \begin{bmatrix}
        \widebar{\bm{l}}_n(\bm{\mu}) \\ 
        \widebar{\bm{k}}_n(\bm{\mu})
    \end{bmatrix},
\end{equation}
where the quantities 
\begin{equation*}
    \widebar{\bm{A}}_n(\bm{\mu}) \in \R^{n^u \times n^u}, \quad 
    \widebar{\bm{B}}_n(\bm{\mu}) \in \R^{n^p \times n^u}, \quad 
    \widebar{\bm{l}}_n(\bm{\mu}) \in \R^{n^u}, \quad 
    \widebar{\bm{k}}_n(\bm{\mu}) \in \R^{n^p}
\end{equation*}
are defined similarly to \eqref{eq: hypred quantities}. 

The system in \eqref{eq: stokes rb algebraic} constitutes the reduced equation for the steady, incompressible Stokes benchmark. However, as noted in \cite{ballarin2015supremizer}, the inf-sup stability of the operator $\bm{B}(\bm{\mu})$ does not guarantee the same for $\bm{B}_n(\bm{\mu})$. A trivial counterexample is the case where $n^p > n^u$, which violates the surjectivity of $\bm{B}_n(\bm{\mu})$ -- a necessary condition for the inf-sup stability. For this reason, a stabilization procedure is required to ensure the well-posedness of \eqref{eq: stokes rb algebraic}, which we describe in the next subsection.

\subsection{Supremizer enrichment}
\label{subs: supremizers}
The supremizer enrichment proposed in~\cite{ballarin2015supremizer} is arguably the most commonly used technique for ensuring the well-posedness of saddle-point problems in \ac{rom} applications. This approach consists, essentially, of augmenting the velocity \ac{rb} space with the \textit{supremizing velocity} corresponding to a given pressure mode $\bm{q} \in \R^{\mathcal{N}_h^p}$. By virtue of Thm.~\ref{thm: coupling}, this augmentation can be done by working on the reference configuration (under the assumption stated in the theorem). The supremizing velocity is none other than the function $\widetilde{\vec{s}}_h$ introduced in \eqref{eq: proof 1 eq 2}, and may algebraically be computed as 
\begin{equation}
    \label{eq: exact supremizer}
    \forall \ \bm{q} \in \R^{\mathcal{N}_h^p}, \ \bm{s} = \bm{X}^{-1/2} \bm{B_{\Omega}}^T \bm{q} = \arg\sup_{\bm{v} \in \R^{\mathcal{N}_h^u}} \frac{\bm{v}^T \bm{B_{\Omega}} \bm{q}}{\|\bm{v}\|_{\bm{X}} \|\bm{q}\|_{\bm{Y}}} \in \R^{\mathcal{N}_h^u}.
\end{equation}
Here, $\bm{B_{\Omega}} \in \R^{\mathcal{N}_h^p \times \mathcal{N}_h^u}$ is the matrix associated with $\widetilde{b}_{\widetilde{\Omega}}$, i.e. the bulk contribution of the coupling operator on the reference configuration. Adding $\bm{s}$ to the velocity basis, followed by an orthogonalization step
\begin{equation*}
    \bm{\Phi}^u \leftarrow \texttt{Gram-Schmidt}\left([\bm{\Phi}^u,\bm{s}], \bm{X}\right) \in \R^{\mathcal{N}_h^u \times (n^u+1)}
\end{equation*}
is what constitutes the supremizer enrichment process. Since pressure modes in the reduced model are linear combinations of the columns of $\bm{\Phi}^p$, we may simply solve \eqref{eq: exact supremizer} for each pressure basis vector and perform the enrichment. As explained in detail in \cite{ballarin2015supremizer}, this is enough to guarantee the surjectivity of $\bm{B}_n$ -- the reduced coupling operator on the reference configuration -- following the enrichment. We summarize the enrichment procedure in Alg.~\ref{alg: supremizer enrichment}. \\ 
Under the assumptions of Thm.~\ref{thm: coupling}, it follows that $\bm{B}_n(\bm{\mu})$ is surjective (i.e., the reduced problem \eqref{eq: stokes reduced equations} is inf-sup stable), as established in the following theorem.
\begin{theorem}
    \label{thm: reduced coupling}
    Let $\bm{\Phi}^u$ be the enriched velocity basis, computed according to Alg.~\ref{alg: supremizer enrichment}. If \eqref{eq: coupling equivalence condition} holds, the Galerkin projection of \eqref{eq: weak formulation stokes} onto the reduced space $\mathrm{span}\{\bm{\Phi}^u\} \times \mathrm{span}\{\bm{\Phi}^p\}$ satisfies the inf-sup stability condition \eqref{eq: inf-sup condition} restricted to these spaces. 
\end{theorem}
\begin{proof} 
    The proof relies on Thm.~\ref{thm: coupling reference} to prove stability in the reference configuration of the \ac{rom}. The key step is the possiblity to choose $\vec{\widetilde{s}}_h \in \mathrm{span}\{\bm{\Phi}^u\}$ that satisfies \eqref{eq: proof 1 eq 2}, for any $\widetilde{q}_h \in \mathrm{span}\{\bm{\Phi}^p\}$. Here, $\vec{\widetilde{s}}_h$ is none other than the supremizer corresponding to $\widetilde{q}_h$. As a result, the supremizer enrichment procedure guarantees the inf-sup stability of the reduced problem in the reference configuration. Finally, we can \emph{transport} the stability to the deformed configuration using Thm.~\ref{thm: coupling}. 
\end{proof}

Lastly, we note that assumption \eqref{eq: coupling equivalence condition} can readily be checked in practice. In particular, the constant can be tracked for every deformed mesh to assess the validity of the assumption. In principle, one could also extend the localization procedure in Algs.~\ref{alg: offline phase}-\ref{alg: online phase} to include the definition of (local) reference \ac{fe} spaces, ensuring that the displacement map never becomes too large. This option, however, is not further explored in the present work.

\begin{algorithm}
    \caption{\texttt{ENRICHMENT}: Given the velocity basis $\bm{\Phi}^u \in \R^{\mathcal{N}_h^u \times n^u}$, the pressure basis $\bm{\Phi}^p \in \R^{\mathcal{N}_h^p \times n^p}$, the reference velocity norm matrix $\bm{X} \in \R^{\mathcal{N}_h^u \times \mathcal{N}_h^u}$, and the bulk reference coupling matrix $\bm{B_{\Omega}} \in \R^{\mathcal{N}_h^p \times \mathcal{N}_h^u}$, return the $\mathcal{N}_h^u \times n^u+n^p$ enriched velocity basis.} 
    \begin{algorithmic}[1]
        \Function{\texttt{ENRICHMENT}}{$\bm{\Phi}^u,\bm{\Phi}^p,\bm{X},\bm{B_{\Omega}}$}
            \State Cholesky factorization: $\bm{H}^T\bm{H} = \texttt{Cholesky}\left(\bm{X}^u\right)$ \Comment{$\bm{H} \in \R^{\mathcal{N}_h^u \times \mathcal{N}_h^u}$} 
            \State Compute supremizers: $\bm{S} = \bm{H}^{-1} \bm{B_{\Omega}}^T \bm{\Phi}^p$ \Comment{$\bm{S} \in \R^{\mathcal{N}_h^u \times n^p}$} 
            \State Supremizer enrichment: $\bm{\Phi}^u \leftarrow \texttt{Gram-Schmidt}\left([\bm{\Phi}^u,\bm{S}], \bm{X}\right)$ \Comment{$\bm{\Phi}^u \in \R^{\mathcal{N}_h^u \times (n^u+n^p)}$} 
            \State \Return $\bm{\Phi}^u$
        \EndFunction
    \end{algorithmic}
    \label{alg: supremizer enrichment}
\end{algorithm}

\section{Numerical results}
\label{sec: results}

In this section, we assess the numerical performance of the proposed \acp{rom} across several tests. We consider four benchmark problems: a 2D Poisson equation, a 3D linear elasticity problem, the 3D Stokes equations, and the 3D Navier-Stokes equations. Their respective weak formulations are detailed in Appendix~\ref{sec: appendix A}. The first two problems are discretized using $\mathcal{Q}_2$ Lagrangian elements and solved with \ac{ttrb}, while the latter employ inf-sup stable $\mathcal{Q}_2$-$\mathcal{Q}_1$ pairs and are solved with a standard \ac{tpod}-\ac{mdeim} procedure\footnote{This choice is primarily motivated by the absence of a straightforward procedure for implementing an efficient supremizer enrichment when the velocity bases are represented in the \ac{tt} format. The main difficulty lies in the fact that neither the velocity norm matrix $\bm{X}$ nor the coupling operator $\bm{B_{\Omega}}$ used in Alg.~\ref{alg: supremizer enrichment} can be expressed in a form that preserves the \ac{tt} structure of the quantities involved -- namely, the pressure basis and the supremizers to be incorporated into the \ac{tt} velocity basis. Achieving this would require all these objects to be written as rank-$K$ tensors, as extensively discussed in \cite{MUELLER2026116790}. In the present setting, however, this is not feasible, as our study involves unfitted discretizations on parameterized domains. We refer to \cite{MUELLER2026116790} for additional discussion and examples illustrating how operations between \ac{tt}-formatted quantities can be carried out efficiently while avoiding expensive manipulations of their full tensor representations. Consequently, the only viable approach would be to convert the bases into full matrix format, perform the standard supremizer enrichment, and then convert the enriched bases back into \ac{tt} format. Although possible, this procedure is computationally unattractive. For these reasons, the \ac{tpod}-\ac{mdeim} approach emerges as the more natural and efficient strategy in the present context.}. The offline and online parameter sets are disjoint: offline realizations are generated using a Halton sequence \cite{PHARR2017401}, whereas the online set $\mathcal{D}_{\texttt{ON}}$ is sampled uniformly from $\mathcal{D}$. For each benchmark, we take $|\mathcal{D}_{\texttt{ON}}| = 10$. We compare the error and cost reduction of the methods against the \ac{hf} solutions for tolerances $\varepsilon \in \{10^{-2}, 10^{-3}, 10^{-4}\}$. Following the recommendation in \cite{quarteroni2015reduced}, the hyper-reduction tolerance is chosen to be $100$ times smaller than that used for constructing the reduced subspace. In evaluating the accuracy of the methods, we verified that the hyper-reduction error is indeed negligible compared to the projection error, as expected. The accuracies of the last two benchmarks are measured using
\begin{equation}
\label{eq: error definition}
    E^u = \frac{1}{|\mathcal{D}_{\texttt{ON}}|}\overset{|\mathcal{D}_{\texttt{ON}}|}{\underset{i=1}{\sum}}\dfrac{\norm{\bm{\Phi}^u\bm{u}_n(\bm{\mu}_i) - \bm{u}(\bm{\mu}_i)}_{\bm{X}(\bm{\mu}_i)}}{\norm{\bm{u}(\bm{\mu}_i)}_{\bm{X}(\bm{\mu}_i)}}, \quad 
    E^p = \frac{1}{|\mathcal{D}_{\texttt{ON}}|}\overset{|\mathcal{D}_{\texttt{ON}}|}{\underset{i=1}{\sum}}\dfrac{\norm{\bm{\Phi}^p\bm{p}_n(\bm{\mu}_i) - \bm{p}(\bm{\mu}_i)}_{\bm{Y}(\bm{\mu}_i)}}{\norm{\bm{p}(\bm{\mu}_i)}_{\bm{Y}(\bm{\mu}_i)}},
\end{equation}
with only $E^u$ being relevant in the first two tests. Cost reduction is quantified through (i) \textit{speedup}, defined as the ratio of \ac{hf} cost -- measured in wall time or memory usage -- to the corresponding \ac{rom} cost, and (ii) \textit{reduction factor}, measured as the ratio between the average \ac{fom} dimension and the dimension of the largest local subspace. All simulations are performed on a local machine equipped with 66~\si{GB} RAM and an Intel Core i7 processor at 3.40~\si{GHz}, using our \ac{rom} library \texttt{GridapROMs.jl} \cite{muellerbadiagridaproms} implemented in Julia.

\subsection{2D Poisson equation}
\label{subs: poisson}
In this subsection, we solve a 2D Poisson equation on a square with a parameter-dependent circular hole:
\begin{equation}
    \Omega(\bm{\mu}) \doteq [-1,1]^2 \setminus \mathcal{B}_{(\mu_1,\mu_1)}(\mu_2), 
\end{equation}
with $\bm{\mu} = (\mu_1,\mu_2) \in \mathcal{D} \doteq [-0.4,0.6] \times [0.25,0.35]$. The strong formulation of the problem reads as: 
\begin{align} 
    \label{eq: poisson}
    \left\{
    \begin{aligned}
        -\Delta u_h(\bm{\mu}) &= f(\bm{\mu}) &&\text{in } \Omega(\bm{\mu}), \\ 
        u_h(\bm{\mu}) &= u_D(\bm{\mu}) &&\text{on } \Gamma_D(\bm{\mu}), \\ 
        \dfrac{\partial u_h(\bm{\mu})}{\partial \vec{n}} &= u_N(\bm{\mu}) &&\text{on }  \Gamma_N(\bm{\mu}). \\
    \end{aligned} 
    \right.
\end{align}
The boundary and forcing terms are defined as:
\begin{itemize}
    \item Homogeneous Dirichlet conditions on the left face of the square and on the hole boundary.
    \item Homogeneous Neumann conditions on the remaining boundaries.
    \item Forcing term $f(\bm{\mu}) = (x,y) \mapsto 2xy$.
\end{itemize}
For the spatial discretization, we set $h = 0.05$ and use $N_{\mu} = 200$ offline samples for both reduced subspace construction and hyper-reduction. At this resolution, the \ac{fom} dimension is $6043$, requiring on average $0.04$\si{s} and $28.68$\si{Mb} per solution. We generate $N_c = N_c^h = 16$ clusters for the local subspaces and hyper-reduction quantities (see Alg.~\ref{alg: offline phase}). We note that both the number of offline parameters and the number of clusters are twice those used in the subsequent tests, owing to the more challenging parameterization in this benchmark. As shown in the following subsections, later examples involve varying two coordinates of the circle's center, which appears to be easier to capture than the present setting. The values adopted here follow the setup of \cite{CHASAPI2023115997}, where the authors used exactly the same choices for a closely related test case. \\ 
Tab.~\ref{tb: results poisson} reports the results for this test. The error of the \ac{ttrb} approximation decreases with the tolerance, but with a multiplicative constant of order $\mathcal{O}(10)$. As noted earlier, this behavior is typical of \ac{ttrb}, where the relative error often surpasses the theoretical $\sqrt{d}$ bound. Moreover, since this problem is highly non-reducible according to the definitions in \cite{quarteroni2015reduced,Unger_2019}, large multiplicative constants are expected regardless of the \ac{rb} procedure. For comparison, the same error measure obtained by a \ac{tpod}-\ac{mdeim} scheme for the same online parameters is
\begin{equation*}
    \text{E} / \varepsilon = [0.53,2.33,11.48].
\end{equation*}
Overall, the two methods exhibit comparable convergence rates. In terms of computational cost, no significant speedup is observed over the \ac{hf} simulations. Two main factors contribute to this: (i)~the problem size is too small to yield substantial speedups, as evidenced by the already low average runtime and memory requirements of the \ac{hf} solver; and (ii)~the online cost is dominated by the \ac{mdeim} stage, due to the relatively large number of cells selected in the reduced integration routine. For such small problems, this effect eliminates any potential speedup, despite the substantial reduction factor achieved. Replacing \ac{mdeim} with the \ac{rbf}-based hyper-reduction used in \cite{CHASAPI2023115997} -- formally presented in Appendix~\ref{sec: appendix B} for completeness -- does not resolve the issue, as shown in Tab.~\ref{tb: results poisson rbf}. Although this approach yields substantial speedups, the accuracy deteriorates significantly, even for the more accurate \ac{tpod}-based method -- see Appendix~\ref{sec: appendix B} for a detailed discussion. Interestingly, \cite{CHASAPI2023115997} reports good results with \acp{rbf}, though their setup uses far more offline snapshots. Apart from this difference, we suspect that enforcing Dirichlet conditions on the deformed boundary -- handled here via Nitsche’s method -- acts as a major obstacle for \acp{rbf}. The presence of the penalty term significantly increases the complexity of the manifold of \acp{lhs} (and \acp{rhs} for nonhomogeneous conditions). The numerical tests in \cite{CHASAPI2023115997} do not include this case, which may explain the disparity in accuracy.

\begin{table}[t]
    \begin{tabular}{ccccc} 
        \toprule
    
        &\multicolumn{1}{c}{\textbf{E} / $\bm{\varepsilon}$}
        &\multicolumn{1}{c}{\textbf{SU-WT}}
        &\multicolumn{1}{c}{\textbf{SU-MEM}}
        &\multicolumn{1}{c}{\textbf{RF}}\\
        
        \cmidrule{1-5}
    
        \multirow{1}{*}[0.1cm]{}
        $\bm{\varepsilon = 10^{-2}}$ &$18.64$ &$1.57$ &$0.54$ & $1510$ \\
    
        \cmidrule{1-5}
    
        \multirow{1}{*}[0.15cm]{}
        $\bm{\varepsilon = 10^{-3}}$ &$41.12$ &$1.50$ &$0.54$ & $1007$ \\

        \cmidrule{1-5}

        \multirow{1}{*}[0.15cm]{}
        $\bm{\varepsilon = 10^{-4}}$ &$83.71$ &$1.50$ &$0.54$ & $671$ \\
        \bottomrule 
    \end{tabular}
    \caption{Results for the Poisson equation benchmark, obtained with \ac{ttrb}. Metrics include: average accuracy in the energy norm ($E$); average computational speedup in terms of wall time (SU-WT); average computational speedup in terms of memory usage (SU-MEM); and reduction factor (RF). Results compare the performance of \ac{ttrb} relative to the \ac{hf} simulations.}
    \label{tb: results poisson}
\end{table}

\begin{table}[t]
    \begin{tabular}{ccccc} 
        \toprule
    
        &\multicolumn{1}{c}{\textbf{E} / $\bm{\varepsilon}$}
        &\multicolumn{1}{c}{\textbf{SU-WT}}
        &\multicolumn{1}{c}{\textbf{SU-MEM}}
        &\multicolumn{1}{c}{\textbf{RF}}\\
        
        \cmidrule{1-5}
    
        \multirow{1}{*}[0.1cm]{}
        $\bm{\varepsilon = 10^{-2}}$ &$9.27$ &$1449$ &$5402$ & $1510$ \\
    
        \cmidrule{1-5}
    
        \multirow{1}{*}[0.15cm]{}
        $\bm{\varepsilon = 10^{-3}}$ &$91.25$ &$1181$ &$3913$ & $1007$ \\

        \cmidrule{1-5}

        \multirow{1}{*}[0.15cm]{}
        $\bm{\varepsilon = 10^{-4}}$ &$919.27$ &$1162$ &$2312$ & $671$ \\
        \bottomrule 
    \end{tabular}
    \caption{Results for the Poisson equation benchmark, obtained with \ac{tpod} for the generation of the subspaces, and \acp{rbf} for the hyper-reduction. Metrics include: average accuracy in the energy norm ($E$); average computational speedup in terms of wall time (SU-WT); average computational speedup in terms of memory usage (SU-MEM); and reduction factor (RF). Results compare the performance of \ac{ttrb} relative to the \ac{hf} simulations.}
    \label{tb: results poisson rbf}
\end{table}

\begin{figure}[t]
    \centering
    \hspace*{-0.5cm}     
    \begin{tikzpicture}[x=1cm,y=1cm]
        \node[anchor=south west, inner sep=0] (img1) at (0,0)
            {\includegraphics[width=0.4\textwidth]{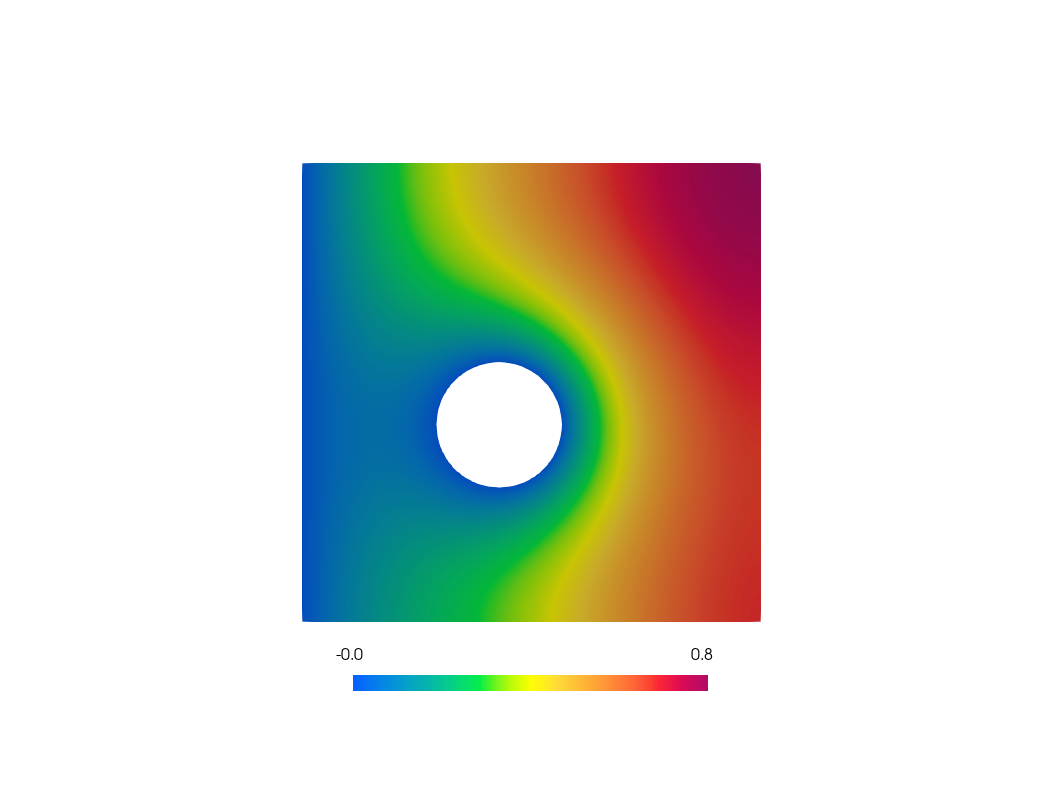}};
        \node at ($(img1.north) - (0.0,0.6)$) {$\bm{u}(\bm{\mu})$};
    
        \node[anchor=south west, inner sep=0] (img2) at (3.9,0)
            {\includegraphics[width=0.4\textwidth]{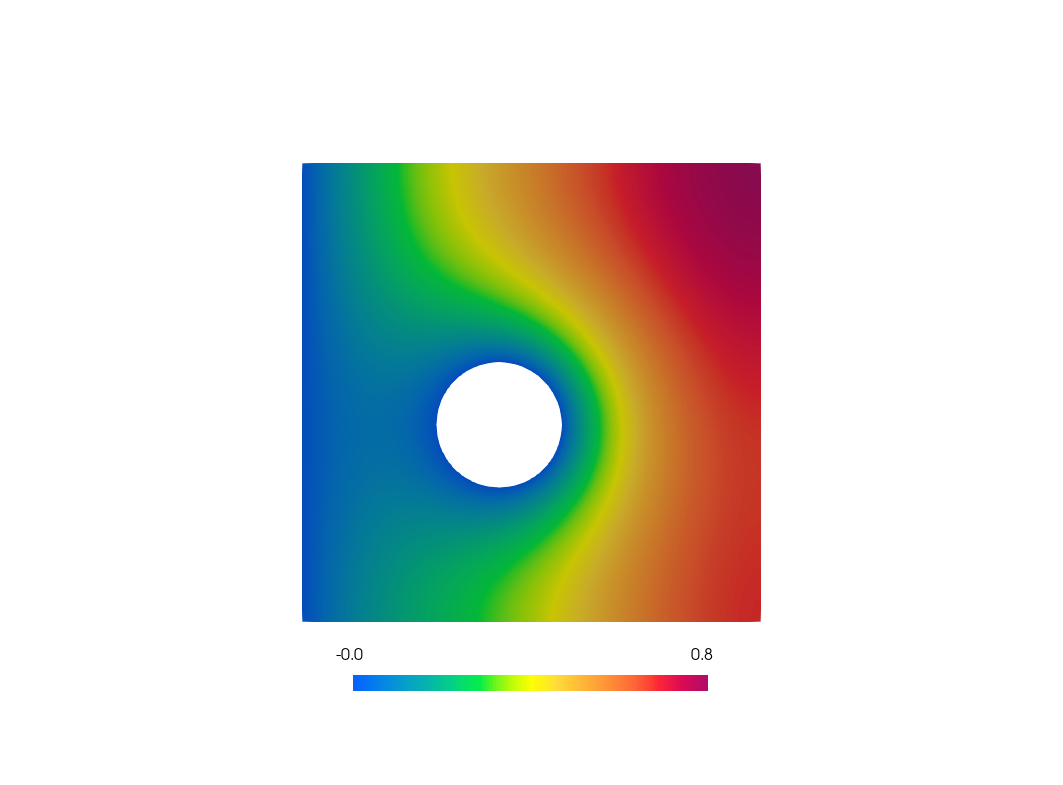}};
        \node at ($(img2.north) - (0.0,0.6)$) {$\bm{\Phi}\bm{u}_n(\bm{\mu})$, $\varepsilon = 10^{-4}$};

        \node[anchor=south west, inner sep=0] (img3) at (7.8,0)
            {\includegraphics[width=0.4\textwidth]{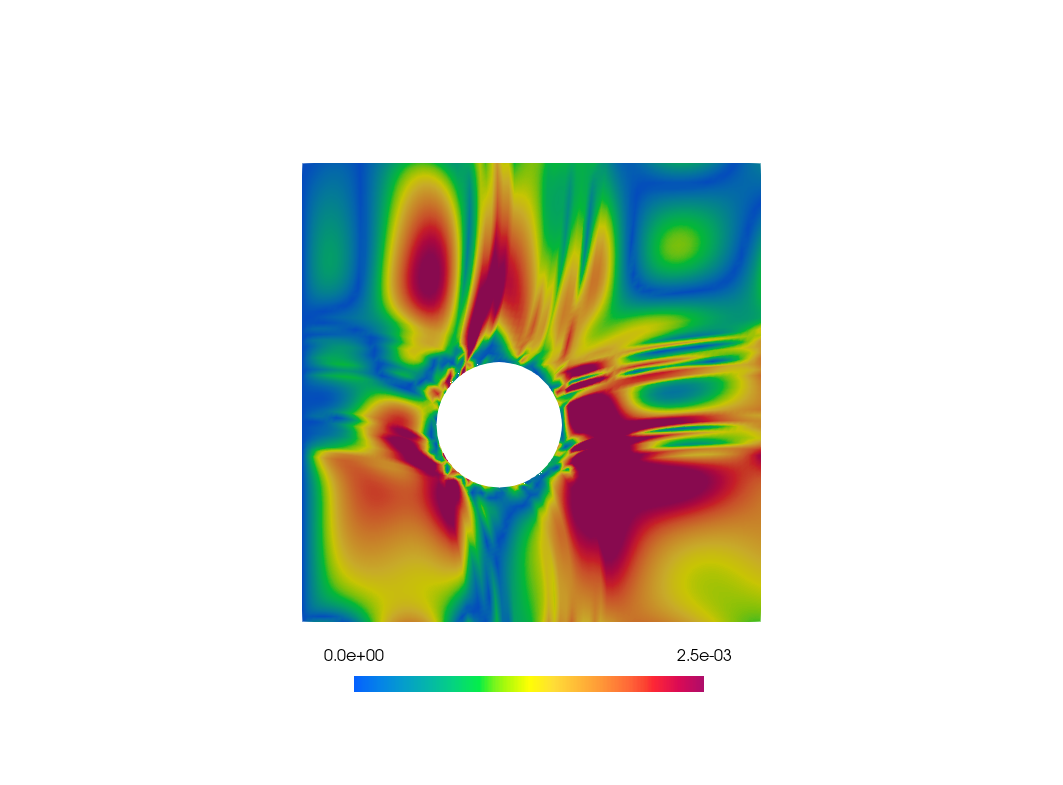}};
        \node at ($(img3.north) - (0.0,0.6)$) {$\bm{err}(\bm{\mu})$, $\varepsilon = 10^{-3}$};

        \node[anchor=south west, inner sep=0] (img4) at (11.7,0)
            {\includegraphics[width=0.4\textwidth]{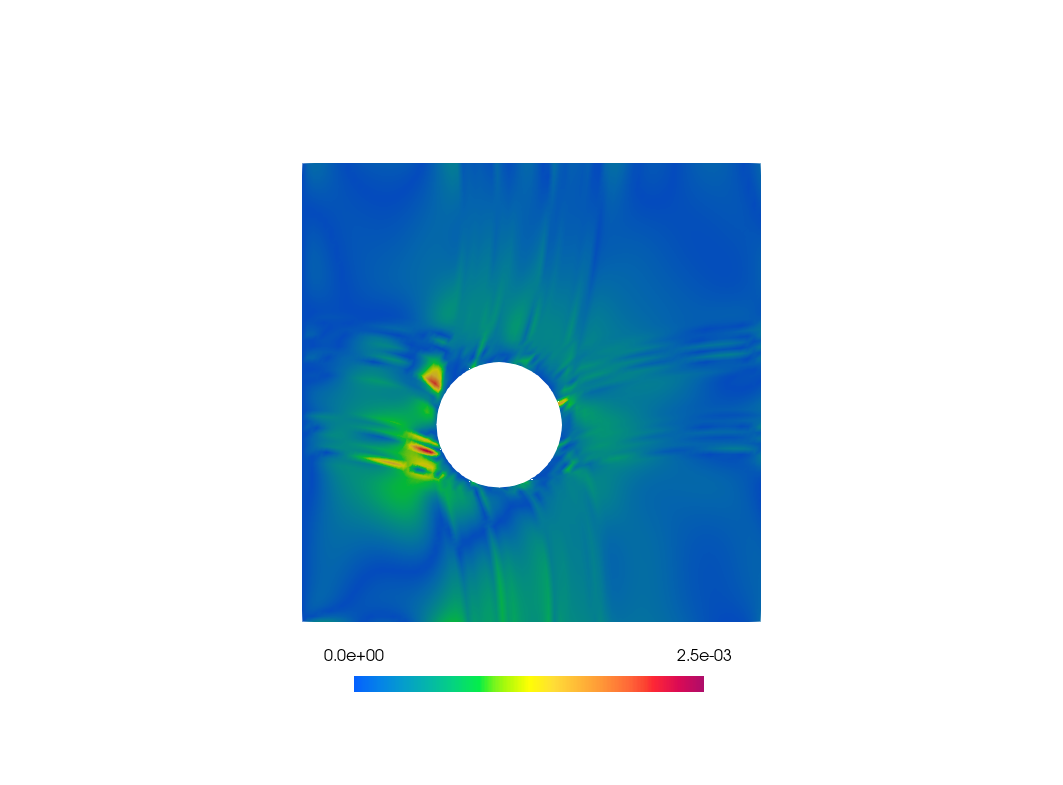}};
        \node at ($(img4.north) - (0.0,0.6)$) {$\bm{err}(\bm{\mu})$, $\varepsilon = 10^{-4}$};
    \end{tikzpicture}
    \caption{Results for the Poisson equation benchmark, obtained with \ac{ttrb}. \ac{fe} solution (left), \ac{ttrb} solution obtained with a tolerance $\varepsilon = 10^{-4}$ (center-left), and point-wise error (centre-right, with a tolerance $\varepsilon = 10^{-3}$, and right, with a tolerance $\varepsilon = 10^{-4}$). Value of the test parameter: $\bm{\mu} = (0.43, 0.26)^T$.}
    \label{fig: results poisson}
\end{figure}
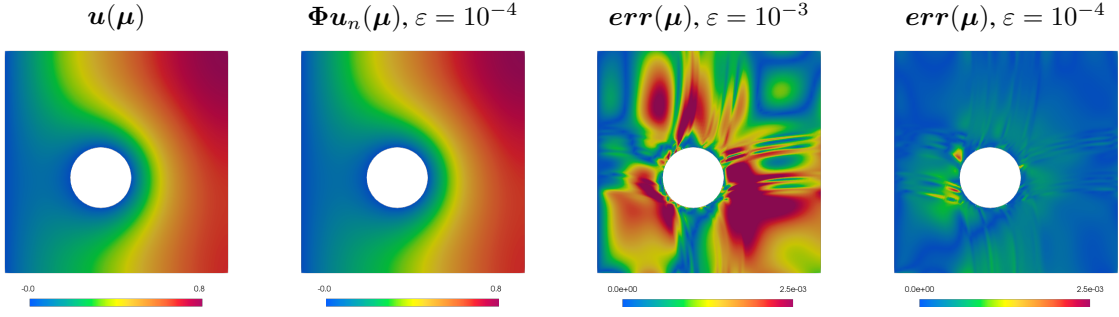

\subsection{3D linear elasticity equation}
\label{subs: elasticity}

In this subsection, we solve a linear elasticity equation~\eqref{eq: strong form elasticity equation} -- with a forcing term $\vec{f}(\bm{\mu})$ replacing $\vec{0}$ in the momentum equation -- on a 3D rectangular domain of size $(L,W,H) = (2,1,0.15)$, featuring a circular hole $\mathcal{B}{(\mu_1,\mu_2)}(R)$ of fixed radius $R = 0.25$. The two shape parameters are chosen from the space $\mathcal{D} \doteq [0.5,1.5] \times [0.4,0.6]$. Homogeneous Dirichlet conditions are imposed on all boundaries, and the following forcing term is prescribed:
\begin{equation*}
    \vec{f}(\bm{\mu}): (x,y,z) \mapsto (2xy,2xy,0)^T.
\end{equation*}
The Young modulus and Poisson ratio are set to $E = 1.0$ and $\nu = 0.3$, respectively. For the spatial discretization, we take $h = 0.05$, yielding a full-order space of dimension $41310$ and requiring, on average, $4.57$\si{s} and $1.02$\si{Gb} per solution. Given the simpler parameterization in this test, we use only $N_{\mu} = 100$ offline snapshots and set the cluster sizes to $N_c = N_c^h = 8$. \\ 
The results are shown in Tab.~\ref{tb: results elasticity} and Fig.~\ref{fig: results elasticity}. Due to the larger \ac{fom} dimension, the speedup is more pronounced than in the previous test case, though, once again, the online cost of the hyper-reduction partially offsets these gains. In fact, for reduction factors of $\mathcal{O}(10^3)$-$\mathcal{O}(10^4)$, one would typically expect higher speedups than those obtained here. In terms of accuracy, the relative errors decrease with the tolerance, but with a marked spike at $\varepsilon = 10^{-3}$. This phenomenon, already noted in~\cite{MUELLER2026116790}, stems from the less regular growth of the \ac{ttrb} ranks with decreasing tolerance compared to a standard \ac{tpod}-\ac{mdeim} approach. For reference, the latter yields the following improved accuracy estimates:
\begin{equation*}
    \text{E} / \varepsilon = [0.70,0.87,3.34].
\end{equation*}

\begin{table}[t]
    \begin{tabular}{ccccc} 
        \toprule
    
        &\multicolumn{1}{c}{\textbf{E} / $\bm{\varepsilon}$}
        &\multicolumn{1}{c}{\textbf{SU-WT}}
        &\multicolumn{1}{c}{\textbf{SU-MEM}}
        &\multicolumn{1}{c}{\textbf{RF}}\\
        
        \cmidrule{1-5}

        \multirow{1}{*}[0.1cm]{}
        $\bm{\varepsilon = 10^{-2}}$ &$38.46$ &$7.07$ &$2.57$ & $8262$ \\
    
        \cmidrule{1-5}
    
        \multirow{1}{*}[0.15cm]{}
        $\bm{\varepsilon = 10^{-3}}$ &$115.80$ &$7.07$ &$2.56$ & $4131$ \\

        \cmidrule{1-5}

        \multirow{1}{*}[0.15cm]{}
        $\bm{\varepsilon = 10^{-4}}$ &$30.12$ &$6.22$ &$2.55$ & $3178$ \\
        \bottomrule 
    \end{tabular}
    \caption{Results for the linear elasticity equation, obtained with \ac{ttrb}. Metrics include: average accuracy in the energy norm ($E$); average computational speedup in terms of wall time (SU-WT); average computational speedup in terms of memory usage (SU-MEM); and reduction factor (RF). Results compare the performance of \ac{ttrb} relative to the \ac{hf} simulations.}
    \label{tb: results elasticity}
\end{table}

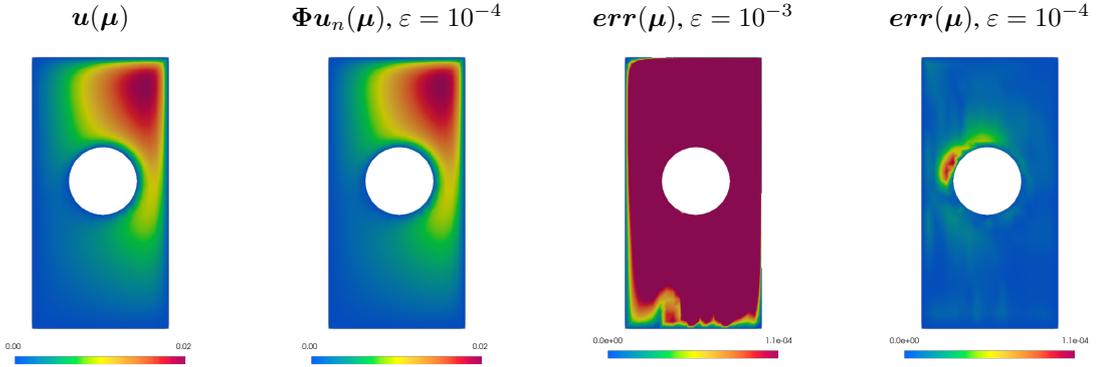
\begin{figure}[t]
    \centering
    \hspace*{-0.5cm}     
    \begin{tikzpicture}[x=1cm,y=1cm]
        \node[anchor=south west, inner sep=0] (img1) at (0,0)
            {\includegraphics[width=0.4\textwidth]{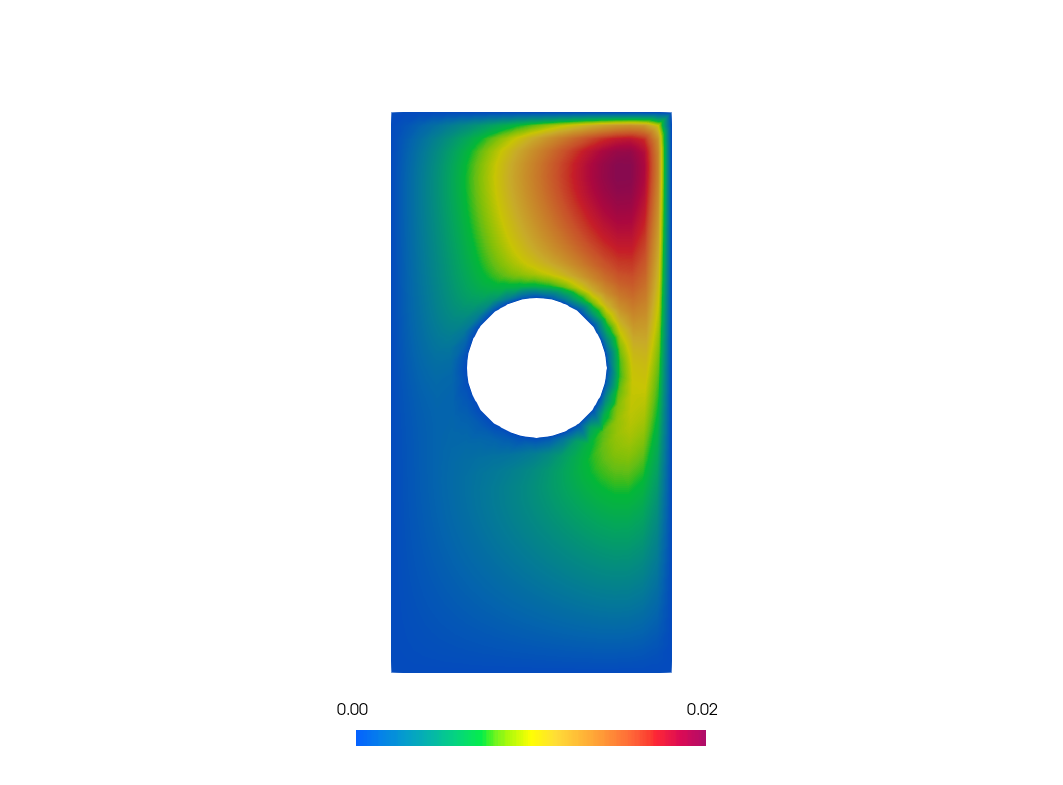}};
        \node at ($(img1.north) - (0.0,0.2)$) {$\bm{u}(\bm{\mu})$};
    
        \node[anchor=south west, inner sep=0] (img2) at (3.9,0)
            {\includegraphics[width=0.4\textwidth]{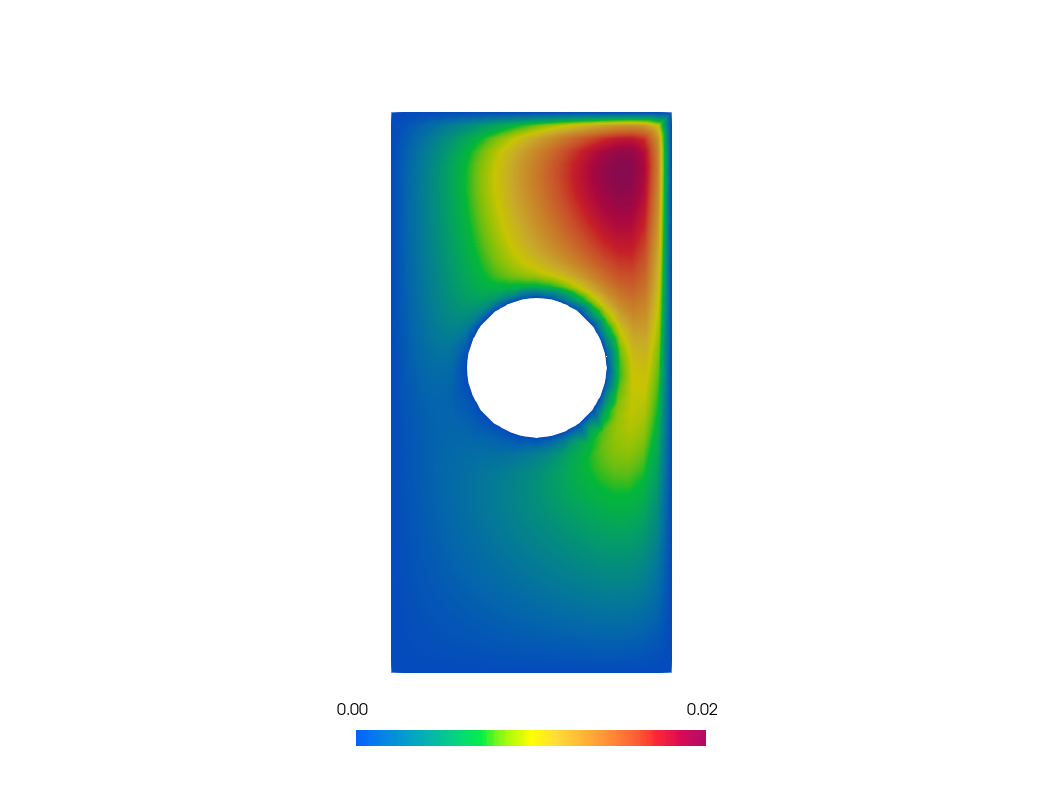}};
        \node at ($(img2.north) - (0.0,0.2)$) {$\bm{\Phi}\bm{u}_n(\bm{\mu})$, $\varepsilon = 10^{-4}$};

        \node[anchor=south west, inner sep=0] (img3) at (7.8,0)
            {\includegraphics[width=0.4\textwidth]{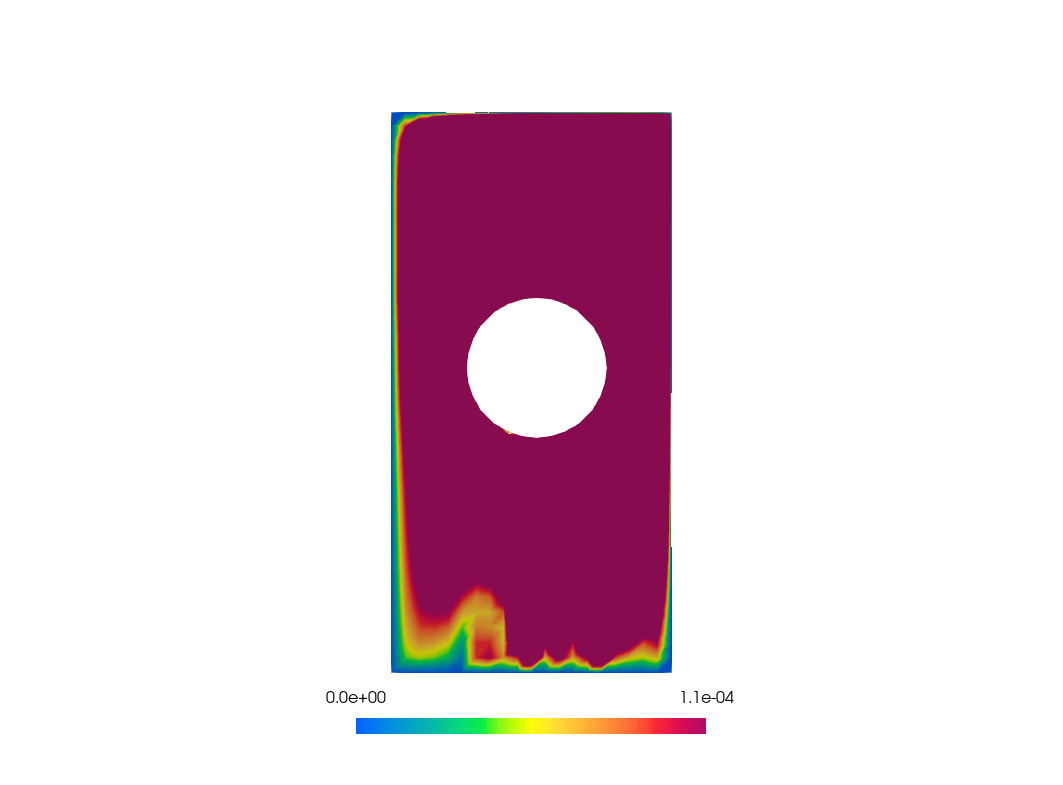}};
        \node at ($(img3.north) - (0.0,0.2)$) {$\bm{err}(\bm{\mu})$, $\varepsilon = 10^{-3}$};

        \node[anchor=south west, inner sep=0] (img4) at (11.7,0)
            {\includegraphics[width=0.4\textwidth]{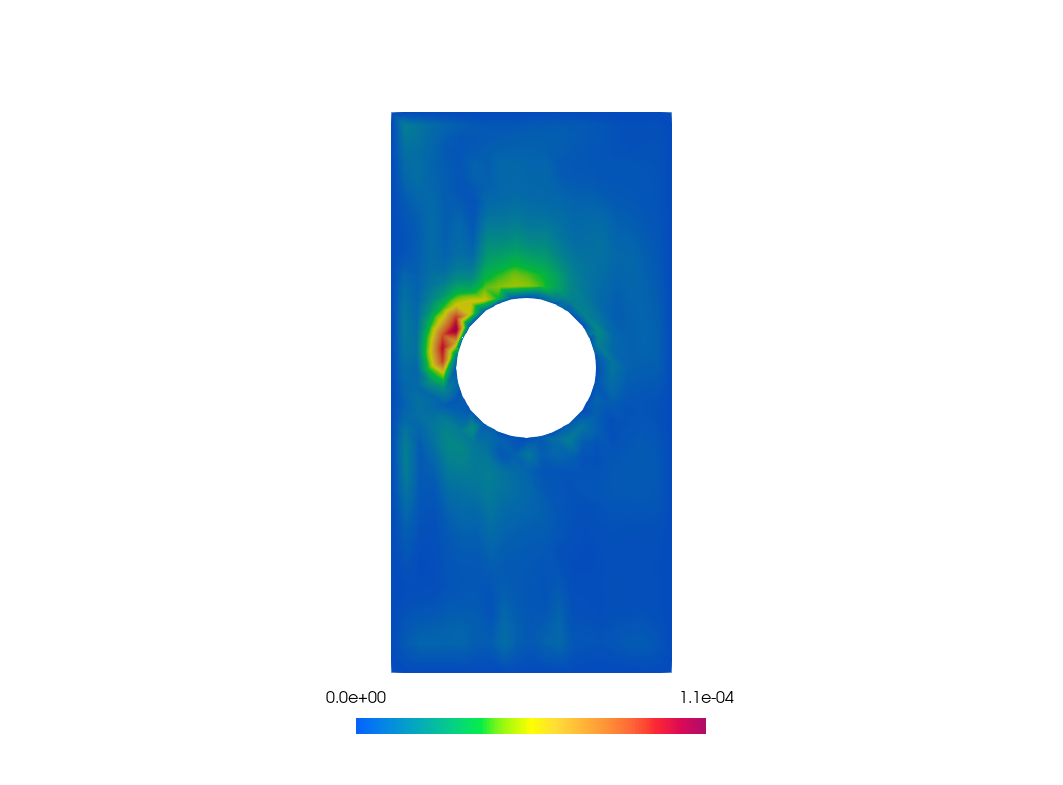}};
        \node at ($(img4.north) - (0.0,0.2)$) {$\bm{err}(\bm{\mu})$, $\varepsilon = 10^{-4}$};
    \end{tikzpicture}
    \caption{Results for the linear elasticity equation, obtained with \ac{ttrb}. Longitudinal mid-section view of the \ac{fe} solution (left), \ac{ttrb} solution obtained with a tolerance $\varepsilon = 10^{-4}$ (center-left), and point-wise error magnitude (centre-right, with a tolerance $\varepsilon = 10^{-3}$, and right, with a tolerance $\varepsilon = 10^{-4}$). Value of the test parameter: $\bm{\mu} = (1.09, 0.52)^T$.}
    \label{fig: results elasticity}
\end{figure}

\subsection{3D Stokes equation}
\label{subs: stokes}

We now solve a 3D fluid-dynamics problem modeled by the steady, incompressible Stokes equation \eqref{eq: weak formulation stokes}-\eqref{eq: stokes forms definition}. As in the previous cases, we consider a 3D rectangular domain of size $(L,W,H) = (1.25,1.25,0.15)$, containing a circular hole $\mathcal{B}{(\mu_1,\mu_2)}(R)$ with fixed radius $R = 0.25$. The two shape parameters are selected from the space $\mathcal{D} \doteq [0.625 \pm 0.2]^2$. The boundary conditions are as follows:
\begin{itemize}
    \item The parabolic inflow
    \begin{equation*}
        \vec{u}_D(\bm{\mu}): (x,y,z) \mapsto (y(W-y),0,0)^T.
    \end{equation*}
    is imposed at the inlet $\{(x,y,z) \in \Omega(\bm{\mu}): x = 0\}$. 
    \item A homogeneous Neumann condition is set at the outlet $\{(x,y,z) \in \Omega(\bm{\mu}): x = L\}$.
    \item No-slip and no-penetration conditions are enforced on the sides $\{(x,y,z) \in \Omega(\bm{\mu}): y \in \{0,W\}\}$. 
    \item No-penetration conditions are considered on the top and bottom $\{(x,y,z) \in \Omega(\bm{\mu}): z \in \{0,H\}\}$.
\end{itemize}
The full-order spaces for velocity and pressure have dimensions $(39691,2384)$, respectively, and computing a single solution requires an average of $331.41$\si{s} and $16.78$\si{Gb}. As before, we set $N_{\mu} = 100$ and $N_c = N_c^h = 8$. \\
The results are shown in Tab.~\ref{tb: results stokes} and Fig.~\ref{fig: results stokes}. Once more, we note an increase of the attained speedup, due to the larger \ac{fom} dimension. In terms of accuracy, the errors for both velocity and pressure converge with the tolerance $\varepsilon$, with multiplicative constants $\sim \mathcal{O}(1)$. This result fully showcases the approximation abilities of the proposed procedure: the accuracy is analogous to that of a Stokes problem solved on non-parametric domains. Indeed, the results seem to suggest that this problem is reducible according to the definitions in \cite{quarteroni2015reduced,Unger_2019} -- i.e., fairly low Kolmogorov $N$-width -- despite it being well-known that problems on parameterized domains are not reducible. The two reasons behind this apparent contradiction are that (i) the use of local subspaces, and (ii) employing deformation maps -- which we recall allow for integration and \ac{fe} interpolation routines to be performed on a reference configuration, instead of different domains for every value of $\bm{\mu}$ -- have a regularizing effect on the manifold of \ac{hf} solutions. This, in turn, allows to effectively use the compression of the snapshots to build accurate \acp{rom}. Finally, the results provide empirical validation of our earlier discussion on the inf-sup stability of the \ac{rom} (see Thm.~\ref{thm: reduced coupling}) for this particular choice of deformation.

\begin{table}[t]
    \begin{tabular}{ccccc} 
        \toprule
    
        &\multicolumn{1}{c}{($\textbf{E}^{\bm{u}} / \bm{\varepsilon}, \textbf{E}^{\bm{p}} / \bm{\varepsilon}$)}
        &\multicolumn{1}{c}{\textbf{SU-WT}}
        &\multicolumn{1}{c}{\textbf{SU-MEM}}
        &\multicolumn{1}{c}{\textbf{RF}}\\
        
        \cmidrule{1-5}
    
        \multirow{1}{*}[0.15cm]{}
        $\bm{\varepsilon = 10^{-2}}$ &$(0.63,1.06)$ &$19.48$ &$4.33$ & $3825$ \\

        \cmidrule{1-5}

        \multirow{1}{*}[0.15cm]{}
        $\bm{\varepsilon = 10^{-3}}$ &$(1.01,1.11)$ &$17.43$ &$4.29$ & $2338$ \\

        \cmidrule{1-5}

        \multirow{1}{*}[0.15cm]{}
        $\bm{\varepsilon = 10^{-4}}$ &$(2.11,1.84)$ &$16.23$ &$4.28$ & $1357$ \\

        \bottomrule 
    \end{tabular}
    \caption{Results for the Stokes equation benchmark, obtained with \ac{tpod}-\ac{mdeim}. Metrics include: average accuracy in the energy norm for velocity and pressure ($E^u,E^p$) normalized with respect to $\varepsilon$; average computational speedup in terms of wall time (SU-WT); average computational speedup in terms of memory usage (SU-MEM); and reduction factor (RF). Results compare the performance of \ac{ttrb} relative to the \ac{hf} simulations.}
    \label{tb: results stokes}
\end{table}

\begin{figure}[t]
    \centering
    \includegraphics[scale=0.35]{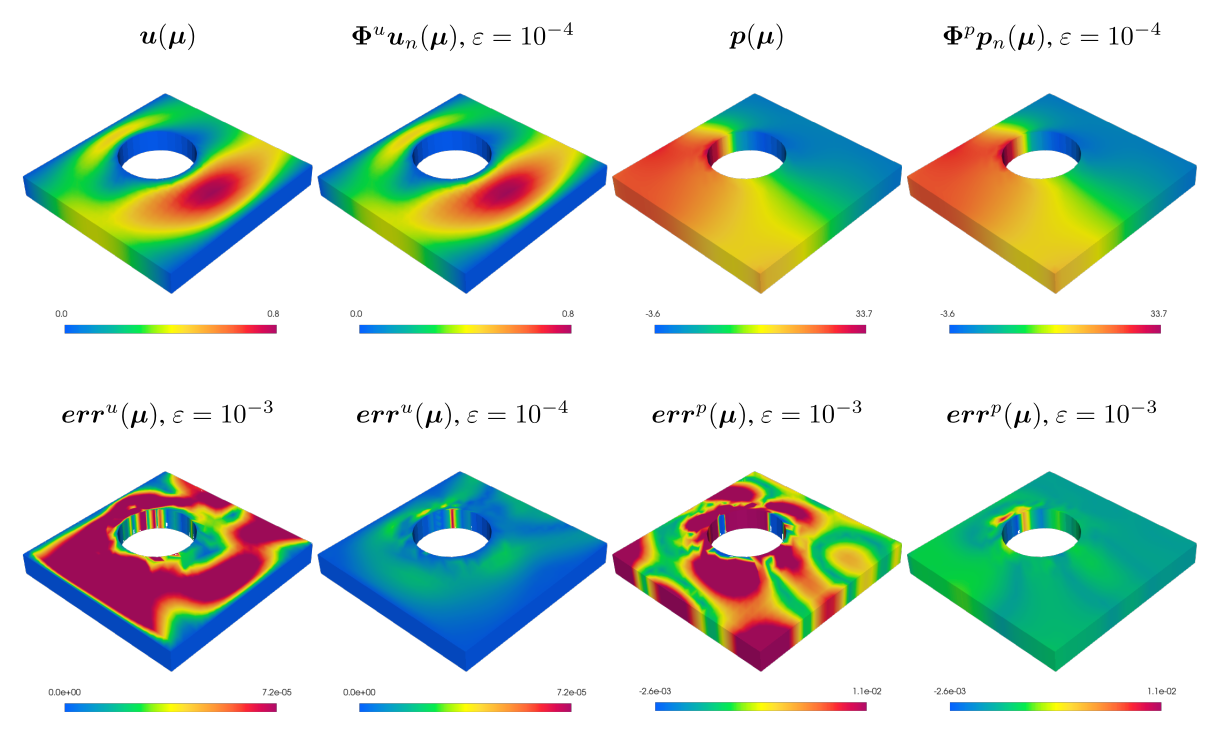}
    \caption{Results for the Stokes equation benchmark, obtained with \ac{tpod}-\ac{mdeim}. Top row: \ac{fe} velocity magnitude (left), \ac{rb} velocity magnitude (centre-left), \ac{fe} pressure (centre-right), and \ac{rb} pressure (right); the \ac{rb} velocity and pressure are obtained with a tolerance $\varepsilon = 10^{-4}$. Bottom row: point-wise velocity error magnitude (left, with a tolerance $\varepsilon = 10^{-3}$, and centre-left, with a tolerance $\varepsilon = 10^{-4}$), and point-wise pressure error (centre-right, with a tolerance $\varepsilon = 10^{-3}$, and right, with a tolerance $\varepsilon = 10^{-4}$). Value of the test parameter: $\bm{\mu} = (0.63, 0.81)^T$.} 
    \label{fig: results stokes}
\end{figure}

\subsection{3D Navier-Stokes equation}
\label{subs: navier-stokes}

We lastly solve a 3D fluid-dynamics problem modeled by the steady, incompressible Navier-Stokes equation, whose weak form is analogous to that of the Stokes problem \eqref{eq: weak formulation stokes}-\eqref{eq: stokes forms definition}, with the following modified velocity-velocity term:
\begin{equation*}
    \begin{split}
        a(\vec{u}_h,\vec{v}_h;\bm{\mu}) &= \int_{\Omega(\bm{\mu})} Re^{-1} \bm{\nabla} \vec{u}_h : \bm{\nabla} \vec{v}_h + \left((\vec{u}_h \cdot \bm{\nabla}) \vec{u}_h\right) \cdot \vec{v}_h 
        \\ & + \int_{\Gamma_D(\bm{\mu})} \tau_h \vec{u}_h \cdot \vec{v}_h - Re^{-1} (\bm{\nabla} \vec{u}_h \vec{n}(\bm{\mu})) \cdot \vec{v}_h - Re^{-1} (\bm{\nabla} \vec{v}_h \vec{n}(\bm{\mu})) \cdot \vec{u}_h.
    \end{split}
\end{equation*}
Here, $Re$ denotes the Reynolds number of the problem, which is fixed to $100$ in our test. The problem specifications, both concerning the \ac{hf} and \ac{rom} parts, are analogous to those of the previous benchmark. Computing a single solution requires an average of $1838.48$\si{s} and $89.89$\si{Gb}. \\
The results are presented in Tab.~\ref{tb: results navier-stokes} and Fig.~\ref{fig: results navier-stokes}. They closely mirror those of the previous benchmark, both in terms of accuracy and computational speedup. This outcome is expected, as the small Reynolds number yields an advection-dominated flow, thereby limiting the influence of the nonlinear term. We also observe a less regular trend in the evolution of speedups with respect to the tolerance: for instance, the speedup obtained at $\varepsilon = 10^{-3}$ exceeds that at $\varepsilon = 10^{-2}$. This effect is due to the increased number of Newton-Raphson iterations required at $\varepsilon = 10^{-2}$. We also note that even higher speedups could be expected for this test case. Compared to the Stokes test, the \ac{hf} costs are more than five times higher, while the \ac{rom} speedups are only about twice as large. Again, the potential gains are limited by the online cost of \ac{mdeim}, which must be performed at each Newton-Raphson iteration, as well as by the comparatively larger number of such iterations. Specifically, while $3$-$4$ iterations suffice for the \ac{fom}, approximately $10$ iterations are required for the \ac{rom}.

\begin{table}[t]
    \begin{tabular}{ccccc} 
        \toprule
    
        &\multicolumn{1}{c}{($\textbf{E}^{\bm{u}} / \bm{\varepsilon}, \textbf{E}^{\bm{p}} / \bm{\varepsilon}$)}
        &\multicolumn{1}{c}{\textbf{SU-WT}}
        &\multicolumn{1}{c}{\textbf{SU-MEM}}
        &\multicolumn{1}{c}{\textbf{RF}}\\
        
        \cmidrule{1-5}
    
        \multirow{1}{*}[0.15cm]{}
        $\bm{\varepsilon = 10^{-2}}$ &$(1.01,1.03)$ &$35.70$ &$8.52$ & $3664$ \\

        \cmidrule{1-5}

        \multirow{1}{*}[0.15cm]{}
        $\bm{\varepsilon = 10^{-3}}$ &$(0.92,1.20)$ &$36.11$ &$10.17$ & $2121$ \\

        \cmidrule{1-5}

        \multirow{1}{*}[0.15cm]{}
        $\bm{\varepsilon = 10^{-4}}$ &$(3.08,2.73)$ &$26.76$ &$9.73$ & $1300$ \\

        \bottomrule 
    \end{tabular}
    \caption{Results for the Navier-Stokes equation benchmark, obtained with \ac{tpod}-\ac{mdeim}. Metrics include: average accuracy in the energy norm for velocity and pressure ($E^u,E^p$) normalized with respect to $\varepsilon$; average computational speedup in terms of wall time (SU-WT); average computational speedup in terms of memory usage (SU-MEM); and reduction factor (RF). Results compare the performance of \ac{ttrb} relative to the \ac{hf} simulations.}
    \label{tb: results navier-stokes}
\end{table}

\begin{figure}[t]
    \centering
    \includegraphics[scale=0.35]{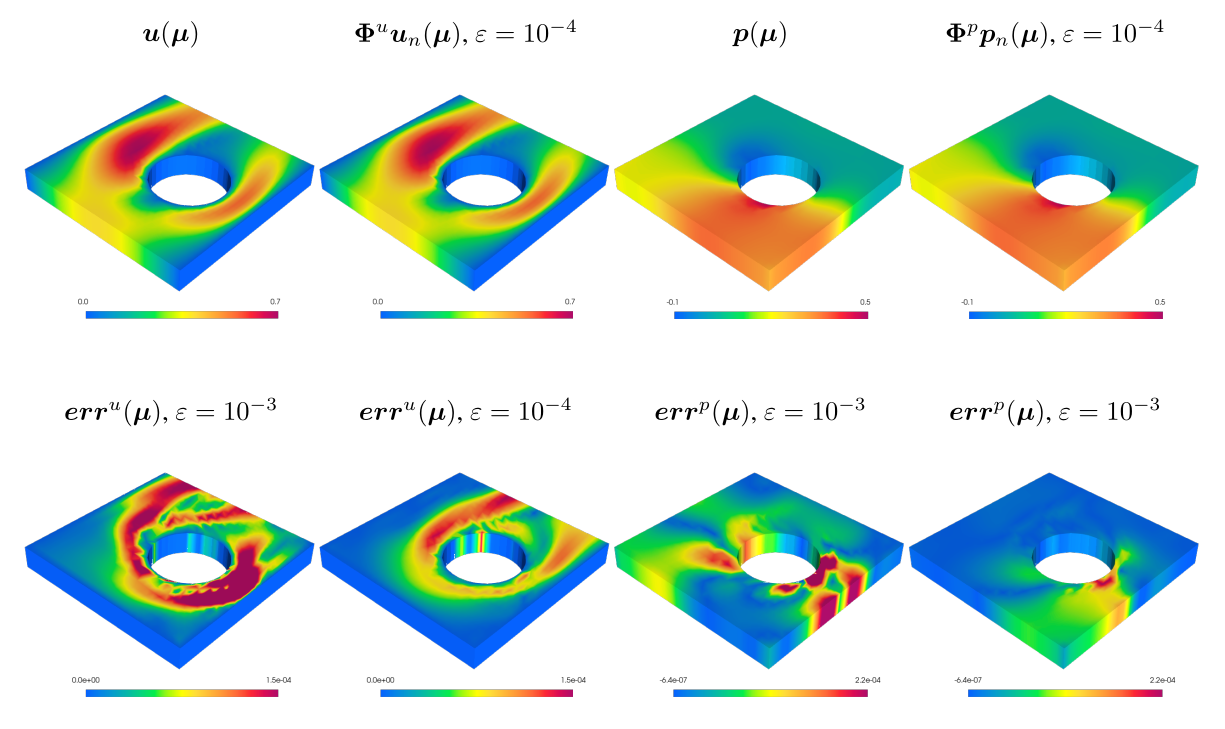}
    \caption{Results for the Navier-Stokes equation benchmark, obtained with \ac{tpod}-\ac{mdeim}. Top row: \ac{fe} velocity magnitude (left), \ac{rb} velocity magnitude (centre-left), \ac{fe} pressure (centre-right), and \ac{rb} pressure (right); the \ac{rb} velocity and pressure are obtained with a tolerance $\varepsilon = 10^{-4}$. Bottom row: point-wise velocity error magnitude (left, with a tolerance $\varepsilon = 10^{-3}$, and centre-left, with a tolerance $\varepsilon = 10^{-4}$), and point-wise pressure error (centre-right, with a tolerance $\varepsilon = 10^{-3}$, and right, with a tolerance $\varepsilon = 10^{-4}$). Value of the test parameter: $\bm{\mu} = (0.62, 0.73)^T$.} 
    \label{fig: results navier-stokes}
\end{figure}
\section{Conclusions and future work}
\label{sec: conclusions}
In this work, we present a unified framework for reduced basis approximations of \acp{pde} defined on parameterized domains. The pipeline of the method includes: (i) representing shape parameterizations via deformation mappings from a chosen reference configuration; (ii) unfitted \ac{fe} discretizations -- which rely on background Cartesian grids, along with an analytical representation of the cut boundaries -- of weak formulations on the reference configuration; (iii) subdividing offline parameters into several clusters; and (iv) for each cluster, building (local) \ac{rb} subspaces and \ac{mdeim} approximations of the problem's residual and Jacobian. We discuss the use of both \ac{tpod}-based reduced subspaces and \ac{tt}-based representations of these subspaces. We also present an efficient supremizer enrichment procedure for stabilizing reduced saddle-point problems on parameterized geometries. In particular, it is possible to enforce the inf-sup stability of the problem by working on the reference configuration, thus avoiding a costly parameter-dependent procedure.

We assess the proposed methodology on several numerical benchmarks, including the Poisson equation, a linear elasticity problem, an incompressible Stokes equation and an incompressible Navier-Stokes equation. We show that the \ac{rb} approximations yield high accuracy, with convergence rates closely tracking the decay of the chosen reduction tolerance. The quality of the approximations stems first from the use of local subspaces, which allows for local linearization of the manifolds of \ac{hf} quantities -- solutions, residuals and Jacobians  -- around the centroids in the parameter space; and second from the use of the deformation map, which also has a regularizing effect on these \ac{hf} manifolds. In terms of computational gains, the methods achieve moderately high speedups relative to \ac{hf} simulations, particularly in wall time. For problems of medium-to-low size -- such as those solved in this manuscript -- the cost of the online hyper-reduction step prevents truly ground-breaking gains. This limitation arises from the large number of \ac{fe} cells selected during the offline phase of \ac{mdeim}, which requires assembling large portions of \ac{fe} matrices and vectors to approximate the residual and Jacobian. However, the gains are expected to increase dramatically with the size of the \ac{hf} problem.

Despite these positive results, we envision two main avenues for future work. Firstly, a posteriori error estimators should be developed to certify the proposed methods, as such certification is currently lacking. Lastly, alternative strategies to the \ac{mdeim}-based interpolation could be explored, since this step accounts for most of the online cost. 

\printbibliography

\appendix

\section[Appendix A]{Weak formulation of the numerical tests in the reference configuration}
\label{sec: appendix A}

In this appendix, we describe the weak formulations of the benchmarks presented in Section~\ref{sec: results}, using the pull-backs introduced in Eq.~\eqref{eq: mapped operators}. For the sake of generality, the boundary data and forcing terms are left unspecified. \\
We begin with the Poisson equation solved in Subsection~\ref{subs: poisson}. On the deformed configuration, we have:
\begin{equation}
    \label{eq: weak poisson deformed}
    a_{\Omega}(u_h(\bm{\mu}),v_h;\bm{\mu}) + a_{\Gamma_D}(u_h(\bm{\mu}),v_h;\bm{\mu}) = l_{\Omega,\Gamma_N}(v_h;\bm{\mu}) + l_{\Gamma_D}(v_h;\bm{\mu}),
\end{equation}
where 
\begin{equation}
    \label{eq: weak poisson deformed terms}
    \begin{split}
        a_{\Omega}(u_h,v_h;\bm{\mu}) &\doteq \int_{\Omega(\bm{\mu})} \bm{\nabla} u_h \cdot \bm{\nabla} v_h,
        \\
        a_{\Gamma_D}(u_h,v_h;\bm{\mu}) &\doteq \int_{\Gamma_D(\bm{\mu})} \tau_h u_h v_h - (\bm{\nabla} u_h \cdot \vec{n}(\bm{\mu})) v_h - (\bm{\nabla} v_h \cdot \vec{n}(\bm{\mu})) u_h,
        \\ 
        l_{\Omega,\Gamma_N}(v_h;\bm{\mu}) &\doteq \int_{\Omega(\bm{\mu})} f(\bm{\mu}) v_h + \int_{\Gamma_N(\bm{\mu})} u_N(\bm{\mu}) v_h,
        \\ 
        l_{\Gamma_D}(v_h;\bm{\mu}) &\doteq \int_{\Gamma_D(\bm{\mu})} \tau_h u_D(\bm{\mu}) v_h - (\bm{\nabla} v_h \cdot \vec{n}(\bm{\mu}))u_D(\bm{\mu}).
    \end{split}  
\end{equation}
Now, omitting the parameter dependence of the deformation Jacobian for notational simplicity, the forms above can be pulled-back on the reference configuration as
\begin{equation}
    \label{eq: weak poisson reference terms}
    \begin{split}
        a_{\Omega}(u_h,v_h;\bm{\mu}) &\equiv \int_{\widetilde{\Omega}} \left(\vec{\vec{J}}^{-T}\widetilde{\bm{\nabla}} \widetilde{u}_h \cdot \vec{\vec{J}}^{-T}\widetilde{\bm{\nabla}} \widetilde{v}_h\right) \mathrm{det}(\vec{\vec{J}}),
        \\  
        a_{\Gamma_D}(u_h,v_h;\bm{\mu}) &\equiv \int_{\widetilde{\Gamma}_D} \left( \tau_h \widetilde{u}_h \widetilde{v}_h \|\vec{\vec{J}}^{-T}\widetilde{\vec{n}}\|_2 - \left(\vec{\vec{J}}^{-T}\widetilde{\bm{\nabla}} \widetilde{u}_h \cdot \vec{\vec{J}}^{-T} \widetilde{\vec{n}}\right) \widetilde{v}_h - \left(\vec{\vec{J}}^{-T}\widetilde{\bm{\nabla}} \widetilde{v}_h \cdot \vec{\vec{J}}^{-T}\widetilde{\vec{n}}\right) \widetilde{u}_h \right) \mathrm{det}(\vec{\vec{J}}),
        \\
        l_{\Omega,\Gamma_N}(v_h;\bm{\mu}) &\equiv \int_{\widetilde{\Omega}} \widetilde{f}(\bm{\mu}) \widetilde{v}_h\mathrm{det}(\vec{\vec{J}}) + 
        \int_{\widetilde{\Gamma}_N} \widetilde{u}_N(\bm{\mu}) \widetilde{v}_h\|\vec{\vec{J}}^{-T}\widetilde{\vec{n}}\|_2\mathrm{det}(\vec{\vec{J}}),
        \\
        l_{\Gamma_D}(v_h;\bm{\mu}) &\equiv \int_{\widetilde{\Gamma}_D} \left(\tau_h \widetilde{u}_D(\bm{\mu}) \widetilde{v}_h\|\vec{\vec{J}}^{-T}\widetilde{\vec{n}}\|_2 - \left(\vec{\vec{J}}^{-T}\widetilde{\bm{\nabla}} \widetilde{v}_h \cdot \vec{\vec{J}}^{-T}\widetilde{\vec{n}}\right)\widetilde{u}_D(\bm{\mu})\right)\mathrm{det}(\vec{\vec{J}}),\\  
        \widetilde{u}_D(\bm{\mu}) &\doteq u_D(\bm{\mu}) \circ \vec{\psi}(\bm{\mu})^{-1}, \quad 
        \widetilde{u}_N(\bm{\mu}) \doteq u_N(\bm{\mu}) \circ \vec{\psi}(\bm{\mu})^{-1}, \quad 
        \widetilde{f}(\bm{\mu}) \doteq f(\bm{\mu}) \circ \vec{\psi}(\bm{\mu})^{-1}.
    \end{split}  
\end{equation}

Now we tackle the linear elasticity problem solved in Subsection \ref{subs: elasticity}. On the deformed configuration, the weak form reads as:
\begin{equation}
    \label{eq: weak elasticity deformed}
    a_{\Omega}(\vec{u}_h(\bm{\mu}),\vec{v}_h;\bm{\mu}) + a_{\Gamma_D}(\vec{u}_h(\bm{\mu}),\vec{v}_h;\bm{\mu}) = l_{\Omega,\Gamma_N}(\vec{v}_h;\bm{\mu}) + l_{\Gamma_D}(\vec{v}_h;\bm{\mu}),
\end{equation}
where 
\begin{equation}
    \label{eq: weak elasticity deformed terms}
    \begin{split}
        a_{\Omega}(\vec{u}_h,\vec{v}_h;\bm{\mu}) &\doteq \int_{\Omega(\bm{\mu})} \vec{\vec{\sigma}}(\vec{u}_h) : \vec{\vec{\epsilon}}(\vec{v}_h),
        \\
        a_{\Gamma_D}(\vec{u}_h,\vec{v}_h;\bm{\mu}) &\doteq \int_{\Gamma_D(\bm{\mu})} \tau_h \vec{u}_h \cdot \vec{v}_h - (\vec{\vec{\sigma}}(\vec{u}_h) \vec{n}(\bm{\mu})) \cdot  \vec{v}_h - (\vec{\vec{\sigma}}(\vec{v}_h) \vec{n}(\bm{\mu})) \cdot \vec{u}_h,
        \\ 
        l_{\Gamma_D}(\vec{v}_h;\bm{\mu}) &\doteq \int_{\Gamma_D(\bm{\mu})} \tau_h \vec{u}_D(\bm{\mu}) \cdot \vec{v}_h - (\vec{\vec{\sigma}}(\vec{v}_h) \vec{n}(\bm{\mu})) \cdot \vec{u}_D(\bm{\mu}).
    \end{split}  
\end{equation} 
The expression for $l_{\Omega,\Gamma_N}$ is analogous to that in \eqref{eq: weak poisson deformed terms}, so we omit it here. On the reference configuration, we have:
\begin{equation}
    \label{eq: weak elasticity reference terms}
    \begin{split}
        a_{\Omega}(\vec{u}_h,\vec{v}_h;\bm{\mu}) &\equiv \int_{\widetilde{\Omega}} \left(\widetilde{\vec{\vec{\sigma}}}(\widetilde{\vec{u}}_h) : \widetilde{\vec{\vec{\epsilon}}}(\widetilde{\vec{v}}_h)\right) \mathrm{det}(\vec{\vec{J}}),
        \\  
        a_{\Gamma_D}(\vec{u}_h,\vec{v}_h;\bm{\mu}) &\equiv \int_{\widetilde{\Gamma}_D} \left( \tau_h \widetilde{\vec{u}}_h \cdot \widetilde{\vec{v}}_h \|\vec{\vec{J}}^{-T}\widetilde{\vec{n}}\|_2 - \left(\widetilde{\vec{\vec{\sigma}}}(\widetilde{\vec{u}}_h) \vec{\vec{J}}^{-T} \widetilde{\vec{n}}\right) \cdot \widetilde{\vec{v}}_h - \left(\widetilde{\vec{\vec{\sigma}}}(\widetilde{\vec{v}}_h) \vec{\vec{J}}^{-T}\widetilde{\vec{n}}\right) \cdot \widetilde{\vec{u}}_h \right) \mathrm{det}(\vec{\vec{J}}),
        \\
        l_{\Gamma_D}(\vec{v}_h;\bm{\mu}) &\equiv \int_{\widetilde{\Gamma}_D} \left(\tau_h \widetilde{\vec{u}}_D(\bm{\mu}) \cdot \widetilde{\vec{v}}_h\|\vec{\vec{J}}^{-T}\widetilde{\vec{n}}\|_2 - \left(\widetilde{\vec{\vec{\sigma}}}(\widetilde{\vec{v}}_h) \vec{\vec{J}}^{-T}\widetilde{\vec{n}}\right) \cdot \widetilde{\vec{u}}_D(\bm{\mu})\right)\mathrm{det}(\vec{\vec{J}}),\\ 
        \widetilde{\vec{\vec{\sigma}}}(\widetilde{\vec{v}}_h) &\doteq 2\gamma \widetilde{\vec{\vec{\epsilon}}}(\widetilde{\vec{v}}_h) + \lambda \mathrm{tr}\left( \vec{\vec{J}}^{-T}\widetilde{\bm{\nabla}} \widetilde{\vec{v}}_h\right) \ \vec{\vec{I}}_d, \quad 
        \widetilde{\vec{\vec{\epsilon}}} \doteq \frac{1}{2}\left(\vec{\vec{J}}^{-T}\widetilde{\bm{\nabla}} + \left(\vec{\vec{J}}^{-T}\widetilde{\bm{\nabla}}\right)^T \right), \quad 
        \widetilde{\vec{u}}_D(\bm{\mu}) \doteq \vec{u}_D(\bm{\mu}) \circ \vec{\psi}(\bm{\mu})^{-1}.
    \end{split}  
\end{equation} 
We recall that $\lambda$ and $\gamma$ are the Lamé coefficients, as introduced in \eqref{eq: stress tensor}.

We now proceed similarly for the Stokes equation solved in Subsection~\ref{subs: stokes}. The deformed equations were already presented in \eqref{eq: weak formulation stokes}-\eqref{eq: stokes forms definition} and are therefore not repeated here. The velocity-velocity block and the \ac{rhs} of the momentum equation can be written on the reference configuration as in \eqref{eq: weak poisson reference terms}, so we focus here only on the forms $b$ and $k$: 
\begin{equation}
    \label{eq: weak formulation stokes reference terms}
    \begin{split}
        b(\vec{u}_h,q_h) &\equiv \int_{\widetilde{\Omega}} \widetilde{q}_h \vec{\vec{J}}^{-T}\widetilde{\bm{\nabla}} \cdot \widetilde{\vec{u}}_h \mathrm{det}(\vec{\vec{J}}) - \int_{\Gamma_D} \widetilde{q}_h \widetilde{\vec{u}}_h \cdot \vec{\vec{J}}^{-T}\widetilde{\vec{n}} \ \mathrm{det}(\vec{\vec{J}}), \\
        k(q_h) &\equiv -\int_{\widetilde{\Gamma}_D} \widetilde{q}_h \widetilde{\vec{u}}_D \cdot \vec{\vec{J}}^{-T}\widetilde{\vec{n}} \ \mathrm{det}(\vec{\vec{J}}).
    \end{split}
\end{equation}

Finally, for the Navier-Stokes equations solved in Subsection~\ref{subs: navier-stokes}, it remains only to express the convective term in terms of quantities defined on the reference configuration. This can be done as follows:
\begin{equation}
    \label{eq: weak formulation navier-stokes reference terms}
    \begin{split}
        \int_{\Omega(\bm{\mu})} \left((\vec{u}_h \cdot \bm{\nabla}) \vec{u}_h\right) \cdot \vec{v}_h \equiv 
        \int_{\widetilde{\Omega}} \left((\widetilde{\vec{u}}_h \cdot \vec{\vec{J}}^{-T}\widetilde{\bm{\nabla}}) \widetilde{\vec{u}}_h\right) \cdot \widetilde{\vec{v}}_h \mathrm{det}(\vec{\vec{J}})
    \end{split}
\end{equation}

\section[Appendix B]{Hyper-reduction with radial basis functions}
\label{sec: appendix B}
In this section, we briefly outline the \ac{rbf}-based hyper-reduction technique used in \cite{CHASAPI2023115997}. For simplicity, we summarize only the global variant of the method, omitting the localized extension introduced by the authors. Nevertheless, incorporating locality would simply require replacing the calls to \ac{mdeim} in Algs.~\ref{alg: offline phase}-\ref{alg: online phase} with the steps described below. The essential workflow is as follows:
\begin{enumerate}
    \item In the offline phase, a dictionary of coefficients for the affine decompositions of the residuals and Jacobians (see Eqs.~\eqref{eq: affine decomp} and \eqref{eq: hypred coeffs}) is computed for every offline parameter. This list is then used to train an efficient interpolation procedure.
    \item In the online phase, given a new parameter $\bm{\mu}^*$, the corresponding coefficient is rapidly computed using the aforementioned interpolation method.
\end{enumerate}
Following the presentation in Subsection~\ref{subs: reduced equations}, we focus on the approximation of the residual, as the Jacobian case is a straightforward extension. First, the set $\{\bm{\theta}^r(\bm{u}(\bm{\mu}_k),\bm{\mu}_k)\}_{k=1}^{N_{\mu}}$ is computed using the \ac{mdeim} procedure detailed in Subsection~\ref{subs: reduced equations}. Next, the interpolation weight matrix $\bm{W} \in \R^{N_{\mu} \times n^r}$ is computed so that
\begin{equation}
\label{eq: rbf offline}
    \sum_{i=1}^{N_{\mu}} \varphi_{q,i}\left(\|\bm{\mu}_k - \bm{\mu}_i\|_2\right) \bm{W}[i,q] \equiv \bm{\theta}^r(\bm{u}(\bm{\mu}_k),\bm{\mu}_k)[q] \quad \forall \ q = 1,\hdots,n^r,\ k = 1,\hdots,N_{\mu}.
\end{equation}
Here, $\varphi_{q,i}(\bm{\mu}) \doteq \varphi_{q,i}(\|\bm{\mu} - \bm{\mu}_i\|_2)$ denotes the $q$th \ac{rbf} associated with the $i$th center point $\bm{\mu}_i$. Then, given the online parameter $\bm{\mu}^*$, its corresponding coefficient $\bm{\theta}^r(\bm{u}(\bm{\mu}^*),\bm{\mu}^*)$ is interpolated using the weight matrix computed offline:
\begin{equation}
    \label{eq: rbf online}
    \bm{\theta}^r(\bm{u}(\bm{\mu}^*),\bm{\mu}^*)[q] \approx \sum_{i=1}^{N_{\mu}} \varphi_{q,i}\left(\|\bm{\mu}^* - \bm{\mu}_i\|_2\right) \bm{W}[i,q] \quad \forall \ q = 1,\hdots,n^r.
\end{equation}
We refer to \cite{Buhmann_2000} for a comprehensive review of the \ac{rbf} method, including the well-posedness analysis of \eqref{eq: rbf offline}. As illustrated by the numerical results in Subsection~\ref{subs: poisson}, this procedure features extremely low offline costs but suffers from reduced accuracy (see Tb.~\ref{tb: results poisson rbf} for further details). Indeed, compared to \ac{mdeim}, the interpolation procedure bypasses the need for the online evaluation of the residual on the mesh cells selected by the greedy scheme, as described in Eq.~\eqref{eq: mdeim approximation}. This operation constitutes the overwhelming majority of the online cost of \ac{mdeim}; since it is entirely absent in the \ac{rbf} scheme, the latter achieves substantially higher speedups. On the other hand, the online evaluation of (a portion of) the residual in \ac{mdeim} - and more broadly in \acp{eim} -- controls the hyper-reduction error for any online parameter (see \cite{chaturantabut2010nonlinear} for a detailed discussion). This mechanism is absent in the \ac{rbf} approach, which instead relies solely on the interpolation \eqref{eq: rbf online}. As a consequence, the \ac{rbf} method incurs substantially larger hyper-reduction errors, effectively preventing it from being reliably applicable in general settings.

\end{document}